\documentclass[a4paper,11pt,reqno]{amsart}
\usepackage{amsmath,amsthm,amssymb}
\usepackage[body={16cm,22.5cm},centering]{geometry} 
\usepackage{latexsym}
\usepackage{stmaryrd}
\usepackage{graphicx}
\usepackage{subfigure}
\usepackage{overpic}
\usepackage{enumerate}
\usepackage{esint}
\usepackage{color}

\theoremstyle{plain}
\newtheorem{theorem}{Theorem}[section]
\newtheorem{lemma}[theorem]{Lemma}
\newtheorem{proposition}[theorem]{Proposition}

\newtheorem{definition}[theorem]{Definition}

\theoremstyle{remark}

\numberwithin{equation}{section}

\newcommand{\R}{\mathbb{R}} 
\newcommand{\N}{\mathbb{N}}
\newcommand{\Z}{\mathbb{Z}}

\newcommand{\Om}{\Omega}
\newcommand{\Sf}{\mathbb{S}}

\newcommand{\A}{\mathcal{A}}
\newcommand{\Ar}{\mathcal{A}^r}
\newcommand{\B}{\mathcal{B}}
\newcommand{\Ha}{\mathcal{H}}
\newcommand{\M}{\mathcal{M}}
\newcommand{\E}{\mathcal{E}}
\newcommand{\Le}{\mathcal{L}}
\newcommand{\Rcal}{\mathcal{R}} 
\newcommand{\BB}{\mathfrak{S}}
 
\newcommand{\dd}{\mathrm{d}}
\newcommand{\id}{\mathbf{id}}
\newcommand{\e}{\varepsilon}
\newcommand{\p}{\partial}
\renewcommand{\t}{\theta}

\newcommand{\Set}{O}

\DeclareMathOperator{\At}{At}
\DeclareMathOperator{\cof}{cof}
\DeclareMathOperator{\Det}{Det}
\DeclareMathOperator{\imG}{im_G}
\DeclareMathOperator{\imT}{im_T}

\DeclareMathOperator{\adj}{adj}
\DeclareMathOperator{\dive}{div}
\DeclareMathOperator{\supp}{supp}
\DeclareMathOperator{\rel}{rel}
\DeclareMathOperator{\tr}{tr}
\DeclareMathOperator{\dist}{dist}
\DeclareMathOperator{\Per}{Per}

\renewcommand{\vec}[1]{\text{\boldmath $#1$}}
\newcommand{\vecg}[1]{\text{\boldmath $#1$}}
\newcommand{\matrizzDet}[1]{\begin{vmatrix} #1 \end{vmatrix}}

\newcommand{\weakcs}{\overset{*}{\rightharpoonup}}
\newcommand{\res}{\mathop{\hbox{\vrule height 7pt width .5pt depth 0pt \vrule height .5pt width 6pt depth 0pt}}\nolimits\,}

\begin{document}

\title[Lack of compactness in the axisymmetric neo-Hookean model]{On the lack of compactness in the axisymmetric neo-Hookean model}
\author{Marco Barchiesi}
\author{Duvan Henao}
\author{Carlos Mora-Corral}
\author{R\'emy Rodiac}
\date{January 30, 2024}

\address[Marco Barchiesi]{
Dipartimento di Matematica, Informatica e Geoscienze, Universit\`a degli Studi di Trieste,
Via Weiss 2 - 34128 Trieste, Italy.
}
\email{barchies@gmail.com}

\address[Duvan Henao]{Faculty of Mathematics and Institute for Mathematical and Computational Engineering, Pontificia Universidad Cat\'olica de Chile, Vicu\~na Mackenna 4860, Macul, Santiago, Chile.
Present address: Instituto de Ciencias de la Ingenier\'ia, Universidad de O'Higgins. Rancagua, Chile.}
\email{duvan.henao@uoh.cl}

\address[Carlos Mora-Corral]{Departamento de Matem\'aticas, Universidad Aut\'onoma de Madrid,
28049 Madrid, Spain \& Instituto de Ciencias Matem\'aticas,
CSIC-UAM-UC3M-UCM, 28049 Madrid, Spain.
}
\email{carlos.mora@uam.es}

\address[R\'emy Rodiac]{Universit\'e Paris-Saclay, CNRS,  Laboratoire de math\'ematiques d'Orsay, 91405, Orsay, France \& Institute of Mathematics, University of Warsaw, Banacha 2, 02-097
Warszawa, Poland}
\email{rrodiac@mimuw.edu.pl}

\begin{abstract}
We provide a fine description of the weak limit of sequences of regular axisymmetric
maps with equibounded neo-Hookean energy, under the assumption that they have finite 
surface energy. We prove that these weak limits have a dipole structure, 
showing that the singular map described by Conti \& De Lellis is generic in some sense. 
On this map we provide the explicit relaxation of the neo-Hookean energy.
We also make a link with Cartesian currents showing that the candidate for the relaxation 
we obtained presents strong similarities with the relaxed energy in the context 
of \(\mathbb{S}^2\)-valued harmonic maps.
\end{abstract}



\maketitle

\tableofcontents

\vspace*{14pt}

\section{Introduction}\label{se:intro}

One of the most used models in nonlinear elasticity is that of neo-Hookean materials:
given a body in a reference configuration $\Om\subset\R^3$, its deformation $\vec u:\Om\to\R^3$ observed 
in response to given boundary conditions is postulated to minimise in a certain admissible function space a stored 
energy functional of the form
\begin{equation*}
 E (\vec u)=\int_\Om \left[|D \vec u|^2+H(\det D\vec u)\right] \dd \vec x,
\end{equation*}
where $H:(0,+\infty)\to[0,+\infty)$ is some convex function penalizing volume changes, satisfying
\begin{equation}\label{eq:growth_H}
\lim_{t\rightarrow +\infty} \frac{H(t)}{t}=\lim_{s\rightarrow 0} H(s)=+\infty.
\end{equation} 
As discussed, e.g., in \cite{Ball01,Ball10}, since minimisers in different function spaces 
can be different, the choice of the function space is part of the model.
Because of the growth condition of $E$, the function space is a suitable subfamily of $H^1(\Om,\R^3)$.
In order to be physically realistic, the deformations have to be at least one-to-one a.e.\ and orientation preserving, 
i.e., to satisfy \(\det D \vec u>0\) a.e\@. We set as boundary condition a bounded $C^1$ orientation-preserving
diffeomorphism $\vec b:\Om\rightarrow \R^3$ and we choose as basic function space
\begin{equation*}
 \A := \{ \vec u \in H^1(\Om,\R^3) : \,  \vec u = \vec b \text{ in } \Om\setminus\widetilde\Om, 
 \, \vec u \text{ is one-to-one a.e.}, \, \det D\vec u>0 \text{ a.e.},  \text{ and }  E(\vec u)<\infty  \}.
\end{equation*}
For technical convenience, we work with a strong form of the Dirichlet boundary condition, 
i.e., we choose a smooth bounded domain $\widetilde{\Om}$ compactly included in $\Om$ and
we require that deformations coincide with $\vec b$ not only on $\partial \Om$ but on the 
whole $\Om\setminus\widetilde{\Om}$.
To avoid interpenetration of matter, the well-known INV condition (see \cite{MuSp95, CoDeLe03}) has to be satisfied.
Simplifying, the INV condition means that after the deformation, matter coming from any subregion $U$ remains 
enclosed by the image of $\partial U$ and matter coming from outside $U$ remains exterior to the region enclosed 
by the image of $\partial U$.
Because of that, a reasonable function space where to look for realistic deformations is
\begin{equation*}
 \Ar := \{ \vec u \in \A : \text{ the divergence identities are satisfied}\},
\end{equation*}
with superscript $r$ standing for ``regular''. 
We recall that the \emph{divergence identities} are
\begin{equation}\label{eq:divergence_identities}
\mathrm{Div}((\adj D \vec u)\vec g \circ \vec u)=(\dive \vec g) \circ \vec u \det D\vec u \quad \forall \vec g\in C^1_c(\R^3,\R^3).
\end{equation}
The identity $\Det D\vec u=(\det D \vec u)\Le^3$, with the distributional determinant defined by
\begin{equation}\label{eq:defDet}
\langle \Det \vec u,\varphi\rangle =-\frac13 \int_\Om \vec u(\vec x)\cdot (\cof D\vec u (\vec x))D \varphi(\vec x) \dd \vec x, \quad \varphi \in C^1_c(\Om),
\end{equation}
is a particular case. 
One can use the Brezis-Nirenberg degree and adapt \cite[Lemma 5.1]{BaHeMo17} to show that condition 
INV holds for maps in $\Ar$.
Moreover, by \cite[Th.\ 3.4]{HeMo15}, the inverse map $\vec u^{-1}$ belongs to $W^{1,1}(\Om_{\vec b},\R^3)$, where 
$\Om_{\vec b}:=\vec b(\Om)$. 

The existence of minimisers in the space $\Ar$ has not yet been obtained, since this space is not sequentially compact with respect to the \(H^1\) weak convergence. Indeed, Conti \& De Lellis \cite[Sect.~6]{CoDeLe03} (see also Section \ref{se:LimitMap}) provided a sequence of 
orientation-preserving bi-Lipschitz deformations with uniformly bounded neo-Hookean energy whose limit $\vec u$ presents a change of orientation and an interpenetration in some region. Therefore, it does not satisfy INV or the divergence identities. 

To prove the existence of minimisers for the neo-Hookean energy in the class $\Ar$,
in \cite{BaHeMoRo_21,BaHeMoRo_23} we proposed a new strategy. 
Firstly, we provided a larger space $\B\supset\Ar$ that is \emph{compact} for sequences with equibounded energy:
\begin{equation*}
\B:= \{\vec u \in \A : \Om_{\vec b} = \imG(\vec u,\Om) \text{ a.e\@. and }
\vec u^{-1}\in BV(\Om_{\vec b},\R^3)\},
\end{equation*}
As usual, $BV$ denotes  the space of functions of bounded variation, while $\imG$ denotes 
the geometric image (see Definition \ref{def:geometric_image}). Our choice for the family 
$\B$ is driven by the fact that if $\vec u$ belongs to  $\Ar$, then \(\imG(\vec u,\Om)=\Om_{\vec b}\) 
(by using the Brezis-Nirenberg degree and adapting \cite[Th.\ 4.1]{BaHeMo17}) and its inverse has Sobolev regularity.

Secondly, we extended $E$ to $\B$ through a \emph{lower semicontinuous} energy:
\begin{equation}\label{eq:relaxedF}
F(\vec u):= E(\vec u)+2 \| D^s \vec u^{-1} \|,
\end{equation}
for \(\vec u \in \B\).
Here $D^s \vec u^{-1}$ is the singular part of the distributional gradient of the inverse, $|D^s \vec u^{-1}|$ is its total variation, 
and $\|D^s \vec u^{-1}\| = |D^s \vec u^{-1}| (\Om_{\vec b})$.
Then, in \cite[Th.\ 1.1]{BaHeMoRo_23} (see also \cite[Th.\ 1.1]{BaHeMoRo_21} for the axisymmetric case), 
we obtained by using the direct method of calculus of variations that the energy $F$ admits a minimiser $\vec u$ on $\B$. 

\begin{theorem}\label{th:previous_article}
Let $(\vec u_n)_n$ be a sequence in $\B$ such that $(F(\vec u_n))_n$ is equibounded.
Then there exists $\vec u\in\B$ such that, up to a subsequence,
$\vec u_n \rightharpoonup \vec u$ in $H^1(\Om,\R^3)$ and 
\begin{equation*}
\liminf_{n\to\infty} F(\vec u_n)\geq F(\vec u).
\end{equation*} 
In particular, the energy $F$ has a minimiser in $\B$. 
\end{theorem}

In this way the existence of a minimiser for $E$ is reduced to showing that $\vec u$ belongs to $\Ar$.
Indeed, the hope is that creating a discontinuity on the inverse and paying the cost $2\|D^s\vec u^{-1}\|$ is incompatible with being a minimiser of $F$ in the class $\B$; this would then yield the existence of a minimiser of the original neo-Hookean energy $E$ in the
regular subclass $\Ar$ of maps where the divergence identities \eqref{eq:divergence_identities} are satisfied, and $\vec u^{-1}$ belongs to $W^{1,1}(\Om_{\vec b},\R^3)$.

We remark that, by definition of the relaxed energy, we have \( F(\vec u)\leq E_{\rel}(\vec u)\) for every \(\vec u\) in the weak \(H^1\) closure of maps in \(\Ar\). We recall that the relaxed energy is defined abstractly by  
\begin{equation}\label{eq:Erel}
 E_{\rel}(\vec u):= \inf \{ \liminf_{n\rightarrow \infty} E(\vec u_n) : (\vec u_n)_{n} \subset \Ar \text{ and } \vec u_n \rightharpoonup \vec u \text{ in } H^1 (\Om, \R^3) \}.
\end{equation}
It is desirable that $F$ coincides with the relaxation of $E$, in order to get, possibly, a negative result: 
if none of the minimisers of the relaxed energy belong to $\Ar$, then $E$ has no minimisers in $\Ar$.

\emph{First goal of this paper:} to show that, for at least the singular map $\vec u$ provided by Conti \& De Lellis, there exists
a sequence $(\vec u_n)_n$ in $\Ar$ such that  $\lim_{n\rightarrow \infty} E(\vec u_n) = F(\vec u)$, i.e., 
$F(\vec u)=E_{\rel}(\vec u)$.

In general proving that $F$ is the relaxed energy by constructing a matching upper bound is a difficult task because of the injectivity constraint. However, even without showing that $F$ is the relaxation of $E$, its explicit expression (compared to that for $E_{\rel}$ in \eqref{eq:Erel}), as well as the more explicit definition of the admissible class $\B$ (compared to the abstract notion of the $H^1$-weak closure of the set of regular orientation-preserving maps), make the proposed variational problem of minimising $F$ in $\B$ likely to be better suited for the study of the regularity of the minimisers.  

Since the map of Conti--De Lellis is axisymmetric (see Subsection \ref{axisimmetry} for a precise definition), 
we will assume $\Om$, $\widetilde{\Om}$, and $\vec b$ axisymmetric and mainly work in the spaces 
\begin{equation}\label{eq:defAs}
\Ar_s:=\{\vec u\in\Ar \colon \vec u \text{ is axisymmetric}\} \text{ and }
\B_s:=\{\vec u\in\B \colon \vec u \text{ is axisymmetric}\}.
\end{equation}
If $\vec u\in\B_s$, then by \cite[Prop.\ 4.15]{BaHeMoRo_21} the first two components of $\vec u^{-1}$ are regular:
\begin{equation*}
\vec u^{-1}=(u^{-1}_1,u^{-1}_2,u^{-1}_3)\in W^{1,1}(\Om_{\vec b}, \R^2) \times BV(\Om_{\vec b}).
\end{equation*}

\begin{figure}[hbt!]
\begin{center}
	\includegraphics[width=0.95\textwidth]{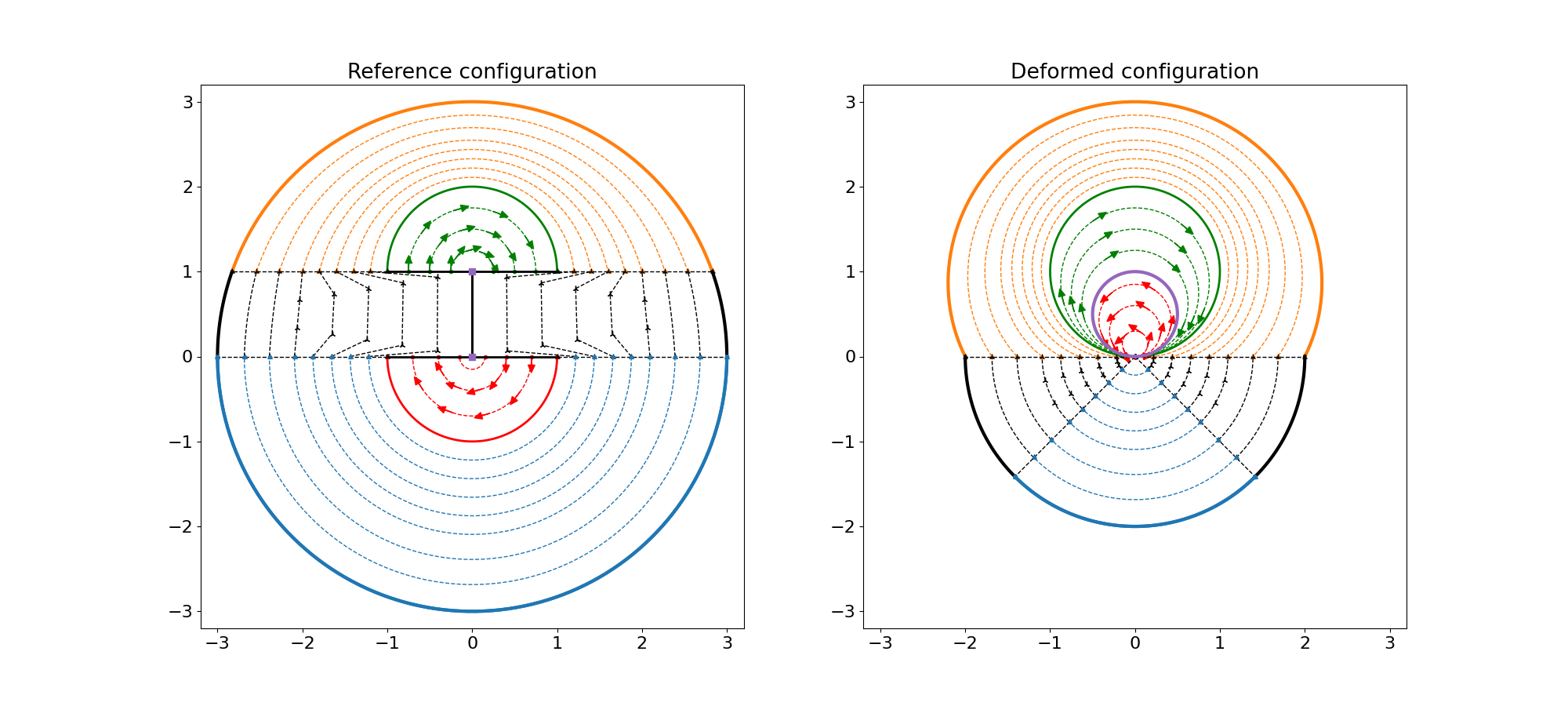}
\end{center}
\caption{
The $2D$ section of (a possible realization of) the Conti--De Lellis map \cite{CoDeLe03}. 
The purple circle $\{y_1^2 + y_2^2+ (y_3-\tfrac{1}{2})^2=\tfrac{1}{2}^2\}$ on the right is not attained 
as the image of any set of material points $\vec x$ in $\Om=B(\vec 0,3)$. It is, instead, new surface created 
by the map, that is, part of the boundary of the image of $\Om \setminus \{\vec 0,\vec 0'\}$ by $\vec u$, 
where $\vec 0=(0,0,0)$ and $\vec 0'=(0,0,1)$ are the only points where $\vec u$ is singular}
\label{fig:python-original_map}
\end{figure}

Let us describe briefly the singular map $\vec u$ of Conti--De Lellis (see Figure \ref{fig:python-original_map}).
We remark that, since $\vec u$ does not satisfy the INV condition,
$\vec u\notin\Ar_s$ and it does not correspond to a physical deformation.
Given any smooth open set $\Set$ containing the origin $\vec 0=(0,0,0)$ and contained in $B(\vec 0, 1)$, it sends (as depicted in Figure \ref{fig:unphysical}) one part of $\Set$ (the one lying in the first two quadrants of the planar representation of this axisymmetric map) into the region enclosed by $\vec u(\p \Set)$, and other part of $\Set$ (the one in the lower half-plane) to the unbounded region outside $\vec u(\p \Set)$.
Also, two parts of the body that  were at unit distance apart, namely, those initially occupying the half-balls
$$
	a:=\{\vec x:\, x_1^2 +x_2^2 +x_3^2 <1,\, x_3 < 0\}
\quad \text{and}\quad 
	e:= \{\vec x:\, x_1^2 +x_2^2 +(x_3-1)^2,\, x_3 >1\},
$$
are put  in contact with each other across the ``bubble''
\begin{equation}\label{bubble}
\Gamma:=\{(y_1,y_2,y_3):\, y_1^2 + y_2^2 + (y_3-\tfrac{1}{2})^2 = \tfrac{1}{2^2}\},
\end{equation}
which in turn comes entirely from only two singular points: the origin $\vec 0$ and $\vec 0'=(0,0,1)$.
Note that from $\vec 0'$ a cavity is created and is filled by material coming from 
the half-ball $a$ through the origin. A similar structure has been already described in the setting
of harmonic maps \cite{BrCoLi86}. We refer to this structure as \emph{dipole}.
The third component of the inverse, $u^{-1}_3$, is not Sobolev, but it belongs to the class $SBV$ 
(special functions of bounded variation). 
Its jump set coincides with the sphere $\Gamma$, and the amplitude of the jump is given by the distance between the
poles $\vec 0$ and~$\vec 0'$.

\begin{figure}[hbt!]
\begin{center}
\begin{overpic}[width=0.8\textwidth,tics=10]{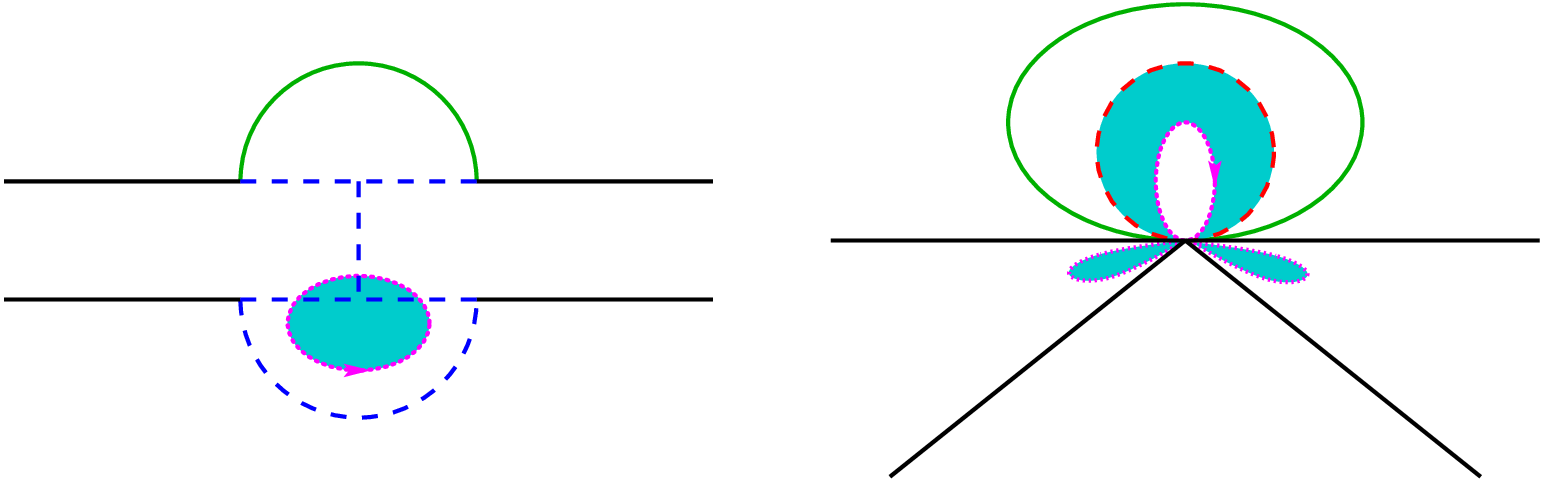}
\put (23,21) {$e$}
\put (23,29) {$f$}
\put (17,15) {$d$}
\put (28,15) {$d$}
\put (21,9) {$\Set$}
\put (23,5) {$a$}
\put (23,1) {$b$}
\put (76,33) {$f$}
\put (76,28) {$e$}
\put (82,24) {$\Gamma$}
\put (76,20) {$a$}
\put (76,9) {$b$}
\put (61,7) {$d$}
\put (91,7) {$d$}
\put (61,13) {\small $\vec u (\p \Set)$}
\put (85,13) {\small $\vec u (\p \Set)$}
\put (73,23) {\footnotesize $\vec u (\p \Set)$}
\end{overpic}
\end{center}
\caption{The Conti--De Lellis map \cite{CoDeLe03} takes a portion of a given region $\Set$ and sends it outside itself. The two closed curves in the right figure play a prominent role. One, on top, $\Gamma$, represented with a dashed circle, is a bubble  created from two cavitation-like singularities. The other, $\vec u (\p \Set)$, with self-intersections, enclosing three connected components, is represented with a dash-dotted line. Part of the coloured region on the right figure lies outside the dash-dotted loop, even though it consists of material points that were inside the dash-dotted curve in the reference configuration. Regions $a$--$f$ are defined in Section 3; see also Figure~\ref{fig:regions}}
\label{fig:unphysical}
\end{figure}

\begin{figure}[hbt!]
\begin{center}
	\includegraphics[width=0.95\textwidth]{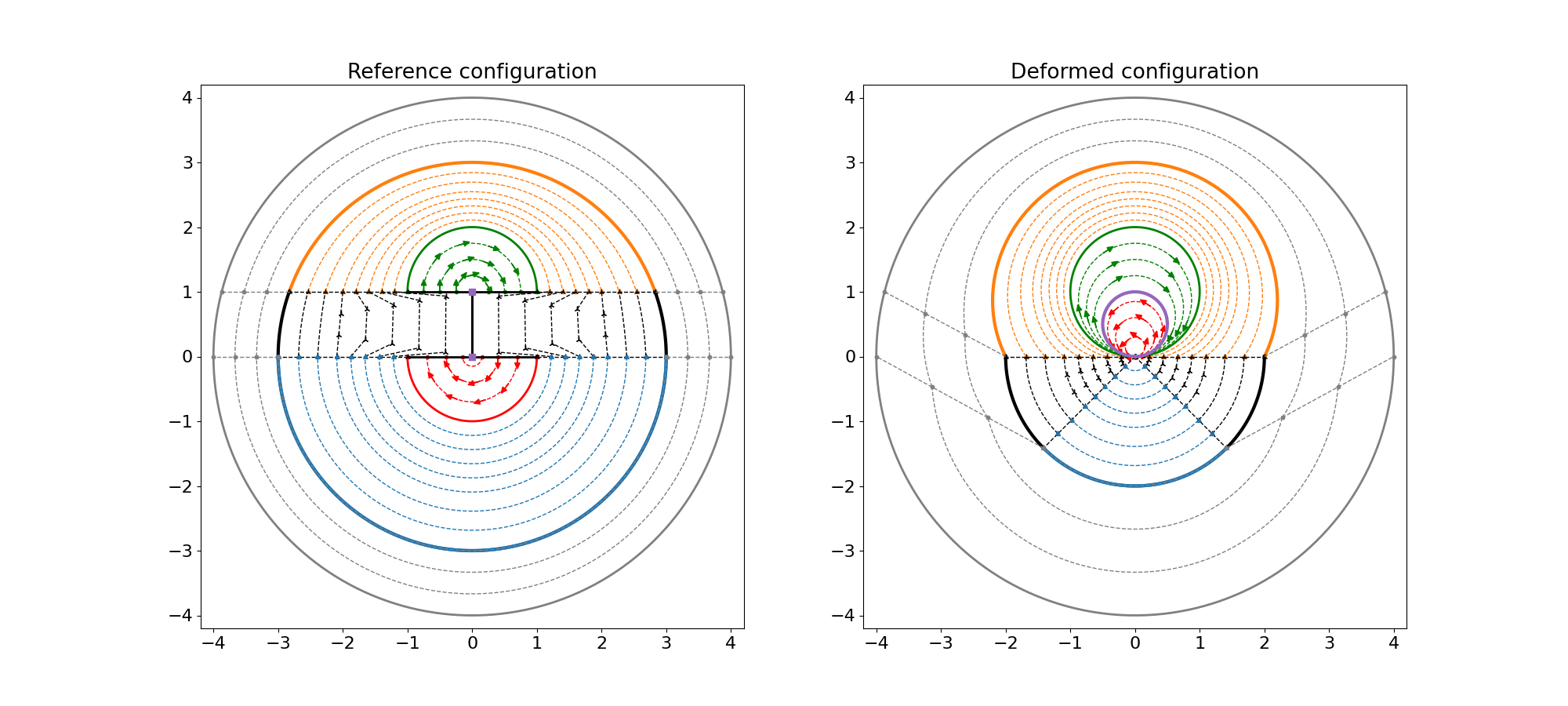}
\end{center}
\caption{An extension of the Conti--De Lellis map \cite{CoDeLe03} that satisfies the Dirichlet 
condition $\vec u(\vec x)=\vec x$ on the boundary}
\label{fig:python-bdry_condition}
\end{figure}

Regarding the Dirichlet boundary condition that facilitates the proof of the lower semicontinuity in 
\cite[Th.\ 1.1]{BaHeMoRo_23} and \cite[Th.\ 1.1]{BaHeMoRo_21}, let us recall that the Conti-De Lellis 
construction can be rescaled and translated, so it can appear as the singular part of maps defined in many 
domains and matching many different Dirichlet data. An example of an extension of the Conti-De Lellis map 
to the domain $B(\vec 0,4)$, so that it satisfies the Dirichlet condition $\vec u(\vec x)=\vec x$ on 
$\partial B(\vec 0, 4)$, is depicted in Figure~\ref{fig:python-bdry_condition}.

As we said, in the present article we are able to prove that the lower bound we obtained previously is optimal in 
the particular case when \(\vec u\) is the dipole of Conti--De Lellis and under very mild hypotheses on \(H\).

\begin{theorem}\label{th:upper_bound}
Let $\vec u$ be the $H^1(B(\vec 0, 3), \R^3)$ axisymmetric map
of Conti--De Lellis, as defined in Section \ref{se:LimitMap}. 
Let \(H: (0,+\infty) \rightarrow [0,+\infty)\) be a convex function satisfying  \eqref{eq:growth_H} and such that
\begin{equation}\label{eq:finiteHu}
\int_{B(\vec 0, 3)} H(\det D\vec u) \dd\vec x < \infty.
\end{equation}
Then there exists a sequence of axisymmetric maps \( (\vec u_n)_n \subset H^1(B(\vec 0,3),\R^3)\)
 such that:
\begin{enumerate}[i)]
\item \( \vec u_n\) is bi-Lipschitz (and therefore satisfies the 
divergence identities \eqref{eq:divergence_identities}) for every \(n\in \N\),
\item \(\vec u_n \rightharpoonup \vec u\) in \(H^1(B(\vec 0,3),\R^3)\),
\item \(\vec u \)  is one-to-one a.e., but it does not satisfy condition INV.
Moreover, one has the equality \( \Det D \vec u=(\det D \vec u)\Le^3+\frac{\pi}{6}(\delta_{(0,0,1)}-\delta_{(0,0,0)})\)
and a fortiori the divergence identities are not satisfied,
\item \( \lim_{n \rightarrow \infty} \int_{B(\vec 0,3)} |D\vec u_n|^2\dd\vec x  =\int_{B(\vec 0,3)} |D\vec u|^2\dd \vec x +2\pi\),
\item \( \lim_{n \rightarrow \infty} \int_{B(\vec 0,3)} H(\det D\vec u_n)= \int_{B(\vec 0,3)} H(\det D \vec u)\),
\item $u^{-1}_3$ has $SBV$ regularity, its jump set is the sphere $\Gamma$ as defined in \eqref{bubble},
and the amplitude of the jump is one. Therefore $\|D^s u^{-1}_3\|=\pi$.
\end{enumerate}
Gathering the last three items shows that
\begin{equation}\label{eq:relaxed_energy_Conti_De_Lellis}
 E_{\text{rel}}(\vec u)=E(\vec u)+2\|D^s u^{-1}_3\|=F(\vec u).
\end{equation}
\end{theorem}

The limiting map \( \vec u \) is the same as in the example of Conti--De Lellis, cf.\ \cite[Th.\ 6.1]{CoDeLe03}. 
However, it can be seen by direct computations that the approximating sequence \(\tilde{\vec u}_n\) of Conti--De Lellis 
satisfies \( \lim_{n \to \infty} E (\tilde{\vec u}_n)\geq E(\vec u)+\frac{8\pi}{3}\) and thus, in view of \eqref{eq:relaxed_energy_Conti_De_Lellis}, it does not give a matching upper bound with the lower bound 
on the relaxed energy obtained in Theorem \ref{th:previous_article}. 
The challenge, therefore, is to regularize the Conti--De Lellis dipole with maps:
\begin{itemize}
 \item that respect the non-interpenetration of matter and the preservation of orientation,
 \item whose  determinant can be controlled as much as possible,
 \item with a negligible amount of extra elastic energy (that is, no energy on top of the $2\pi$ singular energy coming from the Dirichlet term, which is unavoidable according to the lower bound in Theorem \ref{th:previous_article}).
\end{itemize}
We are able to reach the optimal amount of extra energy by constructing a recovery sequence which almost satisfies in 
a neighbourhood of the segment $\{(0,0,t) \colon t\in[0,1]\}$ the equality in the ``area-energy'' inequality 
\( 2|(\cof D \vec u) \vec e_3| \leq |D \vec u|^2\), valid for axisymmetric maps. 
Being optimal for this inequality means that \(\vec u\) restricted to the planes perpendicular to the symmetry axis is conformal. 
Another difference with the  recovery sequence of Conti--De Lellis is that our recovery  sequence 
\(\vec u_n\) is incompressible near the set of concentration. Hence our construction is valid for a general choice of the 
convex function \(H\). This shows that the lack of compactness of the problem is due to the Dirichlet part of the 
neo-Hookean energy and not to the determinant part.

\smallskip
In \cite{HeMo10} a way to measure the amount of new surface created by a deformation  was introduced. 
This can be done by looking at the failure of the divergence identities via the quantity $\E(\vec u)$ in Definition \ref{def:measure_mu}. 
In particular, the divergence identities are satisfied if and only if the surface energy $\E$ is identically zero.  
If $\vec u$ is the Conti--De Lellis dipole, then $\E(\vec u)$ is strictly positive but finite
(see Appendix \ref{se:surface_energy_dipole}).

\emph{Second goal of this paper:} to give a fine description 
of maps $\vec u$ in the weak \(H^1\) closure of \(\Ar_s\) under the supplementary hypothesis 
that the amount of the new surface created by $\vec u$ is finite. 
In this case we show that $\vec u$ has a multi-dipole structure. This also enlightens the importance 
in providing the optimality of our lower bound for the Conti--De Lellis dipole. 

\begin{theorem}\label{th:main2}
Let \(\vec u \in \overline{\Ar_s}\) be such that \(\E(\vec u)<+\infty\). Then:
\begin{itemize}
\item[i)] There exists a countable set of points \(C(\vec u)\) such that
\begin{equation*}
\Det D \vec u=(\det D \vec u) \Le^3+\sum_{ \vec a\in C(\vec u)} \Det D \vec u(\{\vec a\}) \delta_{\vec a}.
\end{equation*}
\item[ii)] \(\vec u^{-1} \in SBV(\Om_{\vec b},\R^3)\). 
\item[iii)] Let $J_{\vec u^{-1}}$ be the jump set of the inverse, and $(\vec u^{-1})^{\pm}$ its lateral traces.
Defined for \(\vec \xi,\vec \xi' \in \R^3\) 
\begin{equation*}
\Gamma_{\vec \xi}^{\pm}:=\{ \vec y \in J_{\vec u^{-1}} : (\vec u^{-1})^{\pm}(\vec y)=\vec \xi\}, \quad \Gamma_{\vec \xi}:=\Gamma_{\vec \xi}^-\cup \Gamma_{\vec \xi}^{+}, \quad \Gamma_{\vec \xi,\vec \xi'}:= \Gamma_{\vec \xi}^-\cap \Gamma_{\vec \xi}^{+} ,
\end{equation*}
we have that
\begin{equation*}
\|D^{s}\vec u^{-1}\| =\sum_{\vec \xi,\vec \xi' \in C(\vec u)} |\vec \xi-\vec \xi'|\Ha^2(\Gamma_{\vec \xi,\vec \xi'}).
\end{equation*}

\item[iv)] Let \(\vec x \in\Om\) and \(r>0\) be such that \(B(\vec x,r) \subset\Om\) and \(\deg(\vec u, \p B(\vec x,r),\cdot)\) is well defined. We set
\begin{equation*}
\Delta_{\vec x,r}:= \deg(\vec u, \p B(\vec x,r),\cdot)-\chi_{\imG(\vec u,B(\vec x,r))}.
\end{equation*}
Then 
\begin{itemize}
\item[a)] \(\Delta_{\vec x, r} \in BV(\R^3)\) and is integer-valued. There exists \(\Delta_{\vec x} \in BV(\R^3)\) integer-valued such that \(\Delta_{\vec x,r_n} \rightharpoonup \Delta_{\vec x}\) weakly\(^*\) in \(BV(\R^3)\) for all sequences \(r_n \rightarrow 0\) such that \( \Delta_{\vec x,r_n}\) is well defined.

\item[b)] \(\Delta_{\vec x} \neq 0\) if and only if \(\vec x\in C(\vec u)\).

\item[c)] For \(\vec \xi \in C(\vec u)\) we have 
$$
	\Gamma_{\vec \xi}=\bigcup_{k \in \Z} \partial^*\{\vec y \in \R^3: \Delta_{\vec \xi}(\vec y) =k\}\quad \Ha^2 \text{-a.e.}
$$
\item[d)] \(\sum_{\vec \xi \in C(\vec u)} \Delta_{\vec \xi}=0\).
\end{itemize}
\end{itemize}
\end{theorem}

Our starting point is Item i), namely, that the singular part of the distributional  determinant \eqref{eq:defDet} consists only of Dirac masses, a result which, in the $H^1$ setting, is due to Mucci \cite{Mucci05,Mucci10a,Mucci10b}, \cite[Th.\ 6.2]{HeMo12}. It is a generalization of the result by M\"uller \& Spector \cite[Th.\ 8.4]{MuSp95}, who obtained it for maps in $W^{1,p}$, with $p>2$, producing a deformed configuration with finite perimeter and satisfying condition INV. 
In that more classical setting, conceived for the modelling of cavitation, the notion of the 
topological image of a point $\vec\xi\in\Om$, introduced by \v Sver\'ak \cite{Sverak88}, is important: 
it is the intersection of $\overline{\imT( \vec u, B(\vec\xi, r))}$ taken over a suitable $\Le^1$-full measure set of radii $r$ (where the degree with respect to $B(\vec\xi, r)$ is well defined, among other requirements). 
Simplifying, the topological image $\imT( \vec u, B(\vec\xi, r))$ is the subregion enclosed by $\partial B(\vec\xi, r)$, 
see Definition \ref{def:topim}.
In particular, the topological image of a singular point $\vec\xi$ provides the cavity opened in the deformed configuration
(a region inside every $\overline{\imT( \vec u, B(\vec\xi, r))}$, but not containing any material coming from a 
neighbourhood of $\vec\xi$).
In contrast, in the $H^1$ setting here considered, where condition INV is not necessarily fulfilled, the notion of the topological 
image of a point is no longer valid since the family $\imT(\vec u, B(\vec\xi, r))$
is not necessarily decreasing for decreasing $r$, as can be seen in the map by Conti \& De Lellis, see Figure \ref{fig:unphysical}; note that in the white region enclosed by dash-dotted loop the degree is $-1$, whereas  the degree is zero outside.
Nevertheless, there happens to be a natural analogue of the topological image of a point: to look at the maps $\Delta_{\vec x, r}$ of Item iv),  which are defined in the deformed configuration, and taking then their $L^1$ limit as $r\to 0$. For example, in the map by Conti \& De Lellis, when $\vec x=(0,0,1)$ the  limit map $\Delta_{\vec x}$ is piecewise constant, equal to $+1$ inside the bubble and equal to zero elsewhere. When $\vec x=(0,0,0)$, the map $\Delta_{\vec x}$ is equal to $-1$ inside the bubble, and zero elsewhere. For any other $\vec x$, the map $\Delta_{\vec  x}$ is identically zero (up  to an $\Le^3$-null set).

The description of the singularities provided by Theorem \ref{th:main2} is the following: the bubbles $\Gamma$ created by an axisymmetric map $\vec  u$ with finite surface energy are not just arbitrary $2$-rectifiable sets, but they are the (reduced) boundaries of a family of volumes (of $3D$ sets with finite perimeter). These volumes, in turn, are not just any volumes, but the level sets of the maps $\Delta_{\vec \xi}$ obtained from the Brezis--Nirenberg degree. Across ($\Ha^2$-almost) every point on $\Gamma$, two portions of the body (that were separated in the reference configuration) are put in contact with each other (the bubbles are the jump set of the inverse of $\vec u$). 
We observe a dipole structure in the contribution $|\vec\xi-\vec \xi'|\Ha^2(\Gamma_{\vec\xi,\vec \xi'})$ of each pair $\vec\xi, \vec \xi'$ to the singular term $\|D^s \vec u^{-1}\|$ in our modified functional $F(\vec u)$, precisely as in the map by Conti \& De Lellis.
Finally, the sum of the degrees $\Delta_{\vec \xi}$, over all singular points $\vec\xi$, vanishes identically: 
this may be indicative of the singularities coming by pairs (positive cavitations cancel out with negative ones, meaning that  holes opened at singular points with positive degrees are filled with material points that in the reference configuration were next to singular points with negative degrees).  
We remark that by Theorem \ref{th:main2} deformations in $\overline{\Ar_s}\setminus\Ar_s$ are not physically realistic.
This fact reinforces the conjecture that the energy $F$ (and so $E$) attains is minimum in $\Ar_s$.
Moreover, as in the context of harmonic maps,  the analysis of the dipole structure 
in \cite{BrCoLi86,BeBrCo90,Giaquinta_Modica_Soucek_1989a}
made it possible to rule out the presence of dipoles coming from smooth axisymmetric minimising sequences \cite{Hardt_Lin_1992},
the partial characterization provided by Theorem \ref{th:main2} could be useful in future attempts to solve the conjecture itself. 

\smallskip
Finally, \emph{third goal of this paper:} to translate the results we obtained about the lower bound on the relaxed energy into the language of Cartesian currents. Specifically, we show (Proposition \ref{prop:relaxed_energy_current}) that the extra term in our candidate for the relaxed energy  \( \| D^su_3^{-1}\|\) can be expressed as the mass of the defect current generated by any sequence converging weakly to \(\vec u\) in \(H^1\). This reinforces the analogy between the problem of finding a minimiser for the neo-Hookean energy and the problem of finding a minimiser for the Dirichlet energy
\begin{equation*}
\int_\Om |D\vec u|^2 \quad \text{in} \quad \bigl\{ \vec u \in C^0 (\Om,\Sf^2) \cap H^1(\Om,\Sf^2) : \vec u=\vec g \text{ on } \p \Om \bigr\}.
\end{equation*}
Indeed, for this problem, which was raised by Hardt and Lin in \cite{Hardt_Lin_1986} and where the occurrence of the Lavrentiev gap phenomenon is shown, Bethuel, Brezis and Coron derived an explicit formula for the relaxed energy in \cite{BeBrCo90}. This relaxed energy can be expressed in terms of the ``length of minimal connections'' for maps with a finite number of singularities. In this case it bears resemblance with our extra lower bound in the case of finite surface energy. In the general case, the supplementary term in the relaxed energy of Bethuel-Brezis-Coron  can be expressed as the mass of the defect currents associated to \(\vec u\), as seen for example in \cite{GiMoSo89}. This is exactly the same for our candidate relaxed energy. Hence both problems have the same flavour in terms of lack of compactness.


The paper is organized as follows.
In Section \ref{se:notation} we introduce our notations and some definitions. In particular, we define the surface energy $\E$ and the geometric and topological images of maps. Section \ref{sec:towards_upper_bound} is devoted to the description of the limit of Conti--De Lellis map and to the construction of the new optimal recovery sequence which allows us to obtain Theorem \ref{th:upper_bound}. In Section \ref{sec:VI} we focus on maps in \(\overline{\Ar_s}\) with finite surface energy. For such maps we prove Theorem \ref{th:main2}. The last section of this paper is devoted to reformulating our candidate for the relaxed energy \(F\) in terms of Cartesian currents.  
We provide four appendices for the comfort of the reader. The first one contains technical lemmas used in the proof of Theorem \ref{th:upper_bound}, the second one describes several geometric quantities in different systems of coordinates, the third one contains a lemma in measure theory used in Section \ref{sec:VI}, and the fourth one computes the surface enery $\mathcal{E} (\vec u)$ of the Conti--De Lellis map.

\section{Notations and definitions}\label{se:notation}

Throughout the paper, we employ the following notation.
\begin{itemize}
\item The open ball of center $\vec x$ and radius $r$ is denoted by $B (\vec x, r)$.
We set $\R_+ = [0, \infty)$.
Given $U \subset \R^n$, its boundary is written as $\p U$, its closure as $\overline{U}$ and its characteristic function as $\chi_U$.
We use the notation $\Subset$ for ``compactly contained''.
\item Vector and matrices are written in bold face.
We recall that the adjugate matrix $\adj \vec A$ of $\vec A \in \R^{3\times 3}$ satisfies $(\det \vec A) \vec{I} = \vec A \adj \vec A$, where $\vec{I}$ denotes the identity matrix.
The transpose of $\adj \vec A$ is the cofactor $\cof \vec A$.
The norm of a vector is the Euclidean norm, and of a matrix the Frobenius norm; we use the notation $|\cdot|$ for both.

\item We use $\wedge$ for the exterior product.
We also make the usual identifications in exterior algebra; for example, a $3$-form in $\R^3$ and a $2$-form in $\R^2$ are identified with a number, while a $2$-form in $\R^3$ is identified with a vector in $\R^3$.
In this way, for instance, $\vec a \wedge \vec b$ is the determinant of $\vec a, \vec b$ whenever $\vec a, \vec b \in \R^2$, while  $\vec a \wedge \vec b$ is the cross product of $\vec a, \vec b$ whenever $\vec a, \vec b \in \R^3$.
\item We use $\Le^N$ for the Lebesgue measure in $\R^N$.
The Hausdorff measure of dimension $d$ is denoted by $\Ha^d$.
We use the abbreviation \emph{a.e.}\ for \emph{almost everywhere} or \emph{almost every}.
It refers to the Lebegue measure, unless otherwise stated.
Given two sets $A, B$ of $\R^N$, we write $A = B$ a.e.\ when $\Le^N (A \setminus B) = \Le^N (B \setminus A) = 0$.
An analogous meaning is given to the expression $\Ha^d$-a.e\@.
\end{itemize}

\subsection{Different coordinates systems and axisymmetry}\label{axisimmetry}

We denote by \( (x_1,x_2,x_3)\in \R^3\) the Cartesian coordinates. We will also use cylindrical coordinates \((r,\theta,x_3)\in \R^+\times [0,2\pi)\times \R\) and spherical coordinates \( (\rho,\theta,\varphi)\in \R^+\times [0,\pi/2)\times [0,\pi]\). The relations between these coordinate systems are
\begin{equation*}
(x_1,x_2,x_3)=(r\cos \theta, r\sin \theta,x_3)=\rho (\cos \theta \sin \varphi, \sin \theta,\sin \varphi, \cos \varphi).
\end{equation*}
A map \( \vec u:\Om \subset \R^3 \to \R^3 \) can be described in the three coordinate systems; we use the notation
\begin{equation*}
\vec u=(u_1,u_2,u_3)=(u_r\cos u_\theta,u_r\sin u_\theta, u_3)=u_\rho( \cos u_\theta \sin u_\varphi, \sin u_\theta \sin u_\varphi, \cos u_\varphi).
\end{equation*}
In Appendix \ref{sec:Appendix_B} we give the different expressions of the differential matrix \( D \vec u\), the cofactor matrix \( \cof D \vec u\), the Jacobian \( \det D \vec u\) and the Dirichlet energy of \(\vec u\) in the different coordinate systems. 

In this paper we mainly work in the axisymmetric setting. We say that the set $\Om \subset \R^3$ is axisymmetric if
\[
 \Om=\bigcup_{\vec x \in \Om} \left( \p B_{\R^2} ((0,0), |(x_1, x_2)|) \times \{ x_3 \} \right).
\]
When we define
\begin{align}\label{def:pi_P}
 \pi : \R^3 & \to [0, \infty) \times \R & \vec P : [0, \infty) \times \R \times \R &\to \R^3 \nonumber \\ 
 \vec x & \mapsto \left( |(x_1, x_2)|, x_3 \right) & (r, \t, x_3) &\mapsto (r \cos \t, r \sin \t, x_3) ,
\end{align}
the axisymmetry of $\Om$ is equivalent to the equality
\begin{equation*}
 \Om = \left\{ \vec P (r, \t, x_3) : \, (r,x_3) \in \pi(\Om), \, \t\in[0,2\pi) \right\} .
\end{equation*}

Given an axisymmetric set $\Om$, we say that $\vec u:\Om\to\R^3$ is axisymmetric if there exists $\vec v : \pi(\Om)\to [0, \infty) \times \R$ such that
\begin{align*}
 & (\vec u \circ \vec P) (r,\t,x_3) = \vec P \left( v_1 (r,x_3), \t , v_2(r,x_3) \right), \text{i.e.}, \\
 & \vec u(r\cos \theta,r\sin \theta,x_3)=v_1(r,x_3)(\cos \theta \vec e_1+\sin \theta \vec e_2) +v_2(r,x_3) \vec e_3
\end{align*}
for all $(r, x_3,\t) \in \pi(\Om)\times[0,2\pi)$.

This $\vec v$ is uniquely determined by $\vec u$. In spherical \( (u_\rho,u_\theta,u_\varphi)\) or cylindrical \( (u_r,u_\theta,u_3)\) coordinates, we remark that for axisymmetric maps we have \( u_\theta=\theta\).

\subsection{Topological images}
We first recall how to define the classical Brouwer degree for continuous functions \cite{Deimling85,FoGa95book}. 
Let \( U \subset \R^3 \) be a bounded open set. If  \(\vec u \in C^1(\overline{U},\R^3)\)  then for every  regular value 
\(\vec y\) of \(\vec u\) with \(\vec y \notin \vec u(\p U)\), we set
\begin{equation}\label{def:degree_1}
\deg(\vec u,  U, \vec y)= \sum_{\vec x \in {\vec u}^{-1}(\vec y) \cap U} \det D {\vec u}(\vec x).
\end{equation}
By definition, a regular value $\vec y$ satisfies that $\det D \vec u (\vec x) \neq 0$ for each $\vec x \in {\vec u}^{-1} (\vec y)$.
Note that the sum in \eqref{def:degree_1} is finite since the pre-image of a regular value consists in isolated points, 
thanks to the inverse function theorem. We can show that the right-hand side of \eqref{def:degree_1} is invariant by 
homotopies. This allows to extend Definition \eqref{def:degree_1} to every \(\vec y \notin \vec u(\p U)\). 
This homotopy invariance can also be used to show that the definition depends only on the boundary values of $\vec u$.
If \(\vec u\) is only in \( C(\p U,\R^3)\), it is again the homotopy invariance which allows to define the degree of \( \vec u\), 
since in this case we may extend $\vec u$ to a continuous map in $\overline U$ by Tietze's theorem and set
\begin{equation*}
\deg(\vec u, U, \cdot)= \deg(\vec v, U ,\cdot),
\end{equation*}
where \(\vec v\) is any map in \(C^1(\overline U, \R^3)\) which is homotopic to the extension of \(\vec u\).

If $U$ is of class $C^1$ and \(\vec u \in C^1(\p U,\R^3)\), by using \eqref{def:degree_1}, Sard's theorem and the divergence identities  
we can make a change of variables and integrate by parts to obtain
\begin{equation}\label{def:degree_integral}
\int_{\R^3} \deg(\vec u ,U,\vec y) \dive \vec g(\vec y) \, \dd \vec y = \int_{\p U} (\vec g \circ \vec u) \cdot \left( \cof D \vec u \, \vec \nu \right) \dd \Ha^{2}.
\end{equation}
This formula can be used as the definition of the degree for maps in \(W^{1,2} \cap L^\infty(\p U,\R^3)\) 
as noticed by Brezis \& Nirenberg \cite{Brezis_Nirenberg_1995}.
For any open set $U$ having a positive distance away from the symmetry axis $\R \vec e_3$
it is possible to use the classical degree since there every map in $\A_s$ has a continuous
representative (cf.\ Lemma 3.1 in \cite{BaHeMoRo_21}). However, for open sets $U$ crossing the axis 
(where maps in $\A_s$ may have singularities) we use the Brezis--Nirenberg degree.

\begin{definition}\label{prop:degree}
Let $U\subset \R^3$ be a bounded open set. For any $\vec u \in C(\partial U, \R^3)$ and 
any $\vec y \in \R^N\setminus \vec u(\partial U)$ we denote by $\deg(\vec u,  U, \vec y)$
the classical topological degree of $\vec u$ with respect to $\vec y$. Suppose now that 
\(U \subset \R^3\) is a \(C^1\) bounded open set and \(\vec u \in W^{1,2}(\p U,\R^3)\cap L^\infty(\p U,\R^3)\). 
Then the degree of \(\vec u\), denoted by \( \deg( \vec u, U, \cdot)\), is defined as the only \(L^1\) function which satisfies
\begin{equation*}
\int_{\R^3} \deg(\vec u, U, \vec y) \dive \vec g(\vec y) \dd \vec y= \int_{\p U} (\vec g \circ \vec u)\cdot \left( \cof D \vec u \, \vec \nu \right) \dd \Ha^{2},
\end{equation*}
for all \( \vec g \in C^\infty(\R^3,\R^3)\).
\end{definition}

To see that this definition makes sense we refer to \cite{Brezis_Nirenberg_1995} or \cite[Remark 3.3]{CoDeLe03}. 
Also, using  \eqref{def:degree_integral} for a sequence of smooth maps approximating $\vec u$ we can see that for any 
\(\vec u \in C(\p U, \R^3)\cap W^{1,2} (\p U,\R^3)\) such that $\Le^3\big ( \vec u(\partial U)\big )=0$
the two definitions are consistent (as stated in \cite[Prop.~2.1.2]{MuSp95}).

Thanks to the degree we can define the topological image of a set through a map.
\begin{definition}\label{def:topim}
Let \(U\subset \R^3\) be a bounded open set and let \(\vec u \in C(\p U,\R^3)\). We define 
\begin{equation*}
\imT(\vec u,U):=\{ \vec y\in \R^3\setminus \vec u(\p U) : \ \deg (\vec u,U,\vec y)\neq 0\}.
\end{equation*}
\end{definition}

We define
\begin{equation}\label{eq:segment_L}
L:=\overline{\Om}\cap \R \vec e_3 .
\end{equation}
Recall from Section \ref{se:intro} that $\vec b:\Om\rightarrow \R^3$ is the given orientation-preserving
diffeomorphism acting as a boundary condition.

\begin{definition}\label{def:imT(u,L)}
Let \(\vec u \in \Ar_s \) and let $\mathcal{U}_{\vec u}^s:=\{U\in\mathcal{U}_{\vec u} \text{ is axisymmetric
and } U\Subset \Om \setminus \R \vec e_3\}$. Here \(\mathcal{U}_{\vec u}\) denotes a family
of ``good open sets'' as defined in \cite[Def.\ 2.12]{BaHeMoRo_21}.
\begin{enumerate}[a)]
\item We define the topological image of $\Om\setminus L$ by $\vec u$ as
\[
\imT (\vec u,\Om\setminus L) 
:=\vec b(\Om') \cup \bigcup_{\substack{U\in \mathcal{U}_{\vec u}^s}} \imT (\vec u, U),
\]
where $\Om'$ is the complement in $\Om$ of the closure of  $\widetilde\Om$.
\item We define the topological image of $L$ by $\vec u$ as
\begin{equation*}
\imT(\vec u,L):=\Om_{\vec b} \setminus \imT (\vec u,\Om\setminus L) .
\end{equation*}
\end{enumerate}
\end{definition}

This definition makes sense because, as explained in Lemma 3.1 in \cite{BaHeMoRo_21}, maps in \(\Ar_s\) are continuous outside the symmetry axis.

Throughout the paper reference is made to condition INV introduced by M\"uller \& Spector \cite{MuSp95}, as a property that is not satisfied by the map of Conti \& De Lellis. 

\begin{definition}
Let $U$ be a bounded open set in $\R^3$.  If $\vec u\in C(U,\R^3)$, we say that $\vec u$ satisfies property INV
in $U$ provided that for every point \(\vec x_0 \in U\) and a.e.\ \(r\in (0,\dist(\vec x_0,\p U))\):
\begin{enumerate}[(a)]
\item \( \vec u(\vec x) \in \imT(\vec u,B(\vec x_0,r))\) for a.e.\ \(\vec x \in B(\vec x_0,r)\)
\item \(\vec u(\vec x) \notin \imT (\vec u,B(\Vec x _0,r))\) for a.e.\ \(\vec x \in \Om \setminus B(\vec x_0,r)\).
\end{enumerate}
\end{definition}

\subsection{The surface energy}
The functional $\E$ was introduced in \cite{HeMo10} to measure the creation of new surface of a deformation.
The formal definition is as follows (see \cite{HeMo10,HeMo12}).

\begin{definition}
	\label{def:measure_mu}
Let $ \vec u\in H^1(\Om,\R^3)$ be such that $\det D \vec u \in L^1(\Om)$.

\begin{enumerate}[a)]
\item For every $\phi \in C^1_c(\Om)$ and $\vec g\in C^1_c(\R^3,\R^3)$ we define
\begin{equation*}
\overline{\E}_{\vec u }(\phi,\vec g)=\int_\Om\left[ \vec g( \vec u(\vec x))\cdot \left( \cof D \vec u(\vec x) D\phi(\vec x) \right) 
+ \phi( \vec x)\dive \vec g ( \vec u(\vec x)) \det D \vec u(\vec x) \right] \dd \vec x .
\end{equation*}

\item For $ \vec g\in C^1_c(\R^3,\R^3)$ we denote by $\overline{\E}_{\vec u}(\cdot, \vec g)$ the distribution on $\Om$ defined by 
\begin{equation*}
\langle \overline{\E}_{\vec u}(\cdot,\vec g),\phi \rangle =\overline{\E}_{\vec u}(\phi,\vec g), \ \forall \phi \in C^1_c(\Om).
\end{equation*}
If for every compact set $K\subset \Om$ we have 
$$\sup \{ \overline{\E}_{ \vec u}(\phi,\vec g) : \phi \in C^1_c(\Om), \supp \phi \subset K, \|\phi\|_{L^\infty}\leq 1 \}<\infty $$
then $\overline{\E}_{\vec u}(\cdot, \vec g)$ is a Radon measure in $\Om$. In this case, if $ U \subset \Om$ is an 
open set then $$\overline{\E}_{\vec u}(U,\vec g)=\lim_{n\rightarrow \infty} \overline{\E}_{\vec u} (\phi_n,\vec g) $$
for all sequences of $\phi_n\in C^1_c(\Om)$ such that $\phi_n\rightarrow\chi_U$ pointwise and \(0\leq \phi_n \leq \chi_{U}\).

\item In the case where $\overline{\E}_{\vec u}(\cdot,\vec g)$ is a Radon measure in $\Om$ 
	for all $\vec g$, we define 
\begin{equation*}
\mu_{\vec u}(E):= \sup \{ \overline{\E}_{\vec u}(E,\vec g) : \vec g \in C^1_c(\R^3,\R^3), \|\vec g\|_{L^\infty}\leq 1 \}
\end{equation*}
for every Borel set $E \subset \Om$.

\item For all $\vec f \in C^1_c(\Om\times \R^3,\R^3)$ we define 
\begin{equation*}
\E_{\vec u}(\vec f)= \int_\Om \left[ D_\vec x\vec f(\vec x,\vec u(\vec x))\cdot \cof D\vec u(\vec x)
+\dive_\vec y \vec f(\vec x,\vec u(\vec x)) \det D \vec u(\vec x) \right] \dd \vec x 
\end{equation*}
and
\begin{equation}\label{eq:def_surface_energy}
 \E(\vec u)= \sup \{ \E_{\vec u}(  \vec f): \, \vec f \in C^1_c(\Om\times \R^3,\R^3) , \, \| \vec f \|_{L^{\infty}} \leq 1 \} .
\end{equation}
\end{enumerate}
\end{definition}

\subsection{Geometric image and area formula}

In this section we state the results for $\R^N$ with arbitrary $N \in \N$.
The definition of approximate differentiability can be found in many places 
(see, e.g., \cite[Sect.\ 3.1.2]{Federer69}, \cite[Def.\ 2.3]{MuSp95} or \cite[Sect.\ 2.3]{HeMo12}).
We recall the area formula (or change of variable formula) of Federer; 
see \cite[Prop.\ 2.6]{MuSp95} or \cite[Th.\ 3.2.5 and Th.\ 3.2.3]{Federer69}.
We will use the notation $\mathcal{N}(\vec u, A,\vec y)$ for the number of preimages of a point $\vec y$ in the set $A$ under $\vec u$.

\begin{proposition}\label{prop:area-formula}
Let $\vec u\in W^{1,1}(\Om,\R^N)$, and denote the set of approximate differentiability points of $\vec u$ by $\Om_d$. Then, for any measurable set $A\subset \Om$ and any measurable function $\varphi:\R^N \rightarrow \R$,
\begin{equation*}
\int_A (\varphi \circ \vec u)|\det D \vec u| \, \dd \vec x =\int_{\R^N} \varphi(\vec y) \, \mathcal{N}(\vec u,\Om_d\cap A,\vec y) \, \dd \vec y
\end{equation*}
whenever either integral exists. Moreover, if a map $\psi:A\rightarrow \R$ is measurable and $\bar{\psi}:\vec u(\Om_d\cap A) \rightarrow \R$ is given by
\begin{equation*}
\bar{\psi}(\vec y):= \sum_{\vec x \in \Om_d \cap A, \vec u(\vec x)= \vec y} \psi(\vec x)
\end{equation*}
then $\bar{\psi}$ is measurable and
\begin{equation}\label{eq:area-formula}
\int_A\psi(\varphi\circ \vec u) |\det D \vec u| \dd \vec x= \int_{\vec u(\Om_d\cap A)} \bar{\psi} \varphi \ \dd \vec y, \ \ \vec y \in \vec u(\Om_d\cap A),
\end{equation}
whenever the integral on the left-hand side of \eqref{eq:area-formula} exists.
\end{proposition}

\begin{definition}\label{def:Om0}
Let $\vec u\in W^{1,1}(\Om,\R^N)$ be such that $\det D \vec u>0$ a.e. We define $\Om_0$ as the set of $\vec x\in \Om$ for which the following are satisfied:
\begin{enumerate}[i)]
\item the approximate differential of $\vec u$ at $\vec x$ exists and equals $D \vec u(\vec x)$.
\item there exist $\vec w\in C^1(\R^N,\R^N)$ and a compact set $K \subset \Om$ of density $1$ at $\vec x$  such that $\vec u|_{K}=\vec w|_{K}$ and $D \vec u|_{K}=D \vec w|_{K}$,
\item $\det D \vec u(\vec x)>0$.
\end{enumerate}
\end{definition}
 We note that the set \( \Om_0\) is a set of full Lebesgue measure in \( \Om \), i.e., \( |\Om \setminus \Om_0|=0\). This follows from Theorem 3.1.8 in \cite{Federer69}, Rademacher's Theorem and Whitney's Theorem.

\begin{definition}\label{def:geometric_image}
For any measurable set $A$ of $\Om$, we define the geometric image of $A$ under $\vec u$ as  
\begin{equation*}
\imG(\vec u,A)=\vec u(A\cap \Om_0)
\end{equation*}
with \(\Om_0\) as in Definition \ref{def:Om0}.
\end{definition}

	\subsection{Change of variables formula for surfaces}

	 Use shall be made of the change of variables formula for surfaces, in the following form.

\begin{proposition}\label{pr:change_of_variables-surfaces}
Let $\vec u \in H^1(\Om, \R^3)\cap L^\infty(\Om, \R^3)$ be injective a.e.\ and such that $\det D\vec u>0$ a.e. 
Then, for every bounded and measurable $\vec g: \R^3 \to \R^3$, every $\vec \xi\in \Om$, and 
a.e.\ $r\in \big (0, \dist(\vec x, \partial \Om)\big )$,
		$$\Ha^2\big ( \partial B(\vec \xi, r)\setminus \Om_0\big ) =0$$
	   and
		$$
			\int_{\partial B(\vecg \xi, r)} \vec g \big ( \vec u(\vec x)\big )\cdot \cof D\vec u(\vec x)\,\vecg\nu(\vec x)\dd\Ha^2(\vec x )
		    = \int_{\imG (\vec u , \partial B(\vec \xi, r) )} \vec g(\vec y) \cdot \tilde{\vecg\nu}_{\vec \xi, r} (\vec y) \dd\Ha^2 (\vec y),
		$$
		where 
		$$
		\tilde{\vecg\nu}_{\vec \xi, r} \big (\vec u(\vec x)\big ) = 
	    \frac{\cof D\vec u(\vec x) \, \vecg\nu(\vec x)}{|\cof D\vec u(\vec x) \, \vecg\nu(\vec x)|},
		\qquad \vec x \in \Om_0\cap \partial B(\vec\xi, r),
		$$
		$\vecg\nu(\vec x)$ being the outward unit normal to $B(\vec\xi,r)$ on $\vec x$.
	 \end{proposition}

	The set $\Om_0$ of Definition \ref{def:Om0} is such that $\vec u|_{\Om_0}$ is injective \cite[Lemma 3]{HeMo11}. 
	Also, $\vec u(\vec x)$ and $D\vec u(\vec x)$ are well defined and $\det D\vec u(\vec x)>0$ at every $\vec x\in \Om_0$. 
	The fact that for almost every radii $r$ the sphere $\partial B(\vec \xi, r)$ is contained in $\Om_0$ except for an $\Ha^2$-null set is a standard consequence of Fubini's theorem and the coarea formula. The change of variables formula, as presented here, is a particular case of the general version of Federer \cite[Cor.~3.2.20]{Federer69} (cf.\ \cite[Prop.~2.7]{MuSp95}, \cite[Prop.~2.9]{HeMo12}).

\section{Towards an upper bound for the relaxed energy}\label{sec:towards_upper_bound}

This section is devoted to the proof of Theorem \ref{th:upper_bound}. We first recall the definition of the limiting map in the Conti--De Lellis example cf.\ \cite[Th.\ 6.1]{CoDeLe03}. Then we present our new optimal approximating sequence and in the last part of this section we check that this approximating sequence is indeed optimal. We denote the canonical basis of \(\R^3\) by $\vec e_1$, $\vec e_2$, $\vec e_3$ and we set
\begin{align*}
 \vec e_r(\theta)  := \cos\theta\, \vec e_1 + \sin\theta \,\vec e_2, \quad 
 \vec e_\theta(\theta) := -\sin \theta\,\vec e_1 + \cos\theta\,\vec e_2 
\qquad \theta\in\R.
\end{align*}

\subsection{Definition of the limit Conti--De Lellis map}
\label{se:LimitMap}

The definition of the limiting map in the Conti--De Lellis example is based on a division of the ball \( B(\vec 0,3)\) into several regions. 
By axisymmetry it suffices to describe \(\vec u\) in the right halfplane.
We describe \(\vec u\) by its spherical coordinates \( (u_\rho, u_{\theta} =\theta, u_\varphi)\).  Note that, when \( \theta=0\), the vector \(\vec e_r \) equals \( \vec e_1\); as a consequence, \( (u_\rho,u_\varphi)\) are the polar coordinates of the map \( \vec u\) restricted to the plane generated by \( (\vec e_1, \vec e_3)\).

\begin{figure}
\centering
\subfigure{
 \begin{overpic}[width=.35\linewidth]{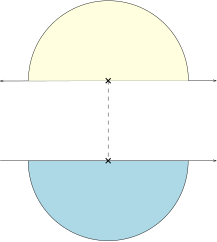}
\put (43,17) {$a$}
\put (75,12) {$b$}
\put (71,48) {$d$}
\put (43,80) {$e$}
\put (75,88) {$f$}
\put (43,27) {$\vec 0$}
\put (43,69) {$\vec 0'$}
\end{overpic}}
\hspace{1cm}             
\subfigure{
\begin{overpic}[width=.41\linewidth]{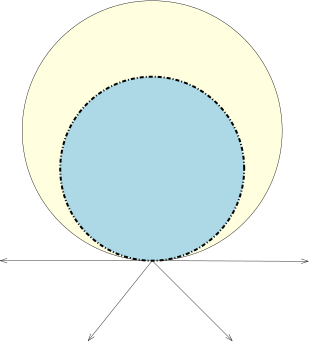}
\put (43,48) {$a$}
\put (44,12) {$b$}
\put (67,16) {$d$}
\put (43,83) {$e$}
\put (70,93) {$f$}
\put (62,73) {$\Gamma$}
\end{overpic}}

\vspace*{1cm}

\caption{The map by Conti--De Lellis is defined differently in regions $a$ to $f$. The reference and deformed configurations appear, respectively, on the left and on the right}
\label{fig:regions}
\end{figure}

We start by describing \(\vec u\) in the region \( a:=\{ \rho \sin \varphi \vec e_r(\theta)+\rho \cos \varphi \vec e_3 : \ 0\leq \rho \leq 1, \frac{\pi}{2} \leq \varphi \leq \pi \}\). In this region we set
\begin{gather*}
    \vec u \big ( \rho \sin\varphi \vec e_r(\theta) +\rho\cos\varphi\vec e_3 \big ) 
    = u_\rho \sin u_\varphi\,\vec e_r(\theta) + u_\rho \cos u_\varphi\, \vec e_3,
    \\
    u_\rho(\rho, \varphi) = (1-\rho)\cos u_\varphi, \qquad
    u_\varphi(\rho,\varphi) = \pi-\varphi.
\end{gather*}
In region \( b:=\{ \rho \sin \varphi \vec e_r(\theta)+\rho \cos \varphi \vec e_3 : 1<\rho\leq 3, \frac{\pi}{2}\leq \varphi \leq \pi\}\) we define \(\vec u \) by
\begin{gather*}
	\vec u\big (\rho\sin\varphi\,\vec e_r(\theta) + \rho\cos\varphi\,\vec e_3\big )= u_\rho \sin u_\varphi\,\vec e_r(\theta) + u_\rho \cos u_\varphi\, \vec e_3,
    \\
    u_\rho(\rho, \varphi) = \rho-1, \qquad
    u_\varphi(\rho,\varphi) = \frac{\varphi+\pi}{2}. 
\end{gather*}
In \(e:=\{ \rho \sin \varphi \vec e_r(\theta) +(1+\rho \cos \varphi) \vec e_3 : 0\leq \rho \leq 1, 0 \leq \varphi \leq \frac{\pi}{2}\}\) we set
	\begin{gather*}
    \vec u\big ( \vec e_3 + \rho\sin\varphi \vec e_r(\theta) + \rho\cos\varphi \vec e_3\big ) := u_\rho  \sin\varphi\, \vec e_r(\theta) + u_\rho \cos \varphi\,\vec e_3,\\
    u_\rho( \rho, \varphi) := (1+\rho)\cos \varphi, \quad  u_\varphi=\varphi.
    \end{gather*}
In \( f:=\{ \vec e_3 + \rho \sin \varphi \vec e_r(\theta)+\rho \cos \varphi \vec e_3 : \rho \geq 1, 0\leq \varphi \leq \frac{\pi}{2}\}\cap B(\vec 0,3)\), in  \cite{CoDeLe03} there is no requirement made on the limiting map \( \vec u\) regarding the region $f$,
 other than that  it be transformed in a bi-Lipschitz manner onto its image, which is contained in
$$\{\vec e_3+\rho \cos\varphi\, \vec e_r(\theta) + \rho\sin\varphi\,\vec e_3: \rho\geq 2\cos\varphi,\  0\leq\varphi\leq \frac{\pi}{2}\}.$$
For the limit map $\vec u$, that means that it can be provided, in spherical coordinates
$$
	\vec u\big (\vec e_3 + \rho\sin\varphi\vec e_r(\theta) + \rho\cos\varphi \vec e_3\big ) = u_\rho\sin u_\varphi \vec e_r(\theta) + u_\rho\cos u_\varphi\vec e_3, 
	\quad \rho\geq 1,\ \varphi \in [0,\frac{\pi}{2}], 
$$
 by any pair of maps $u_\rho=u_\rho(\rho,\varphi)$, $u_\varphi=u_\varphi(\rho,\varphi)$ satisfying 
$$
	u_\varphi(1,\varphi):=\varphi,\quad
	u_\varphi(\rho, \frac{\pi}{2}) = \frac{\pi}{2}, 
	\quad u_\rho(1,\varphi):=2\cos\varphi.
$$ 
Note, in particular, that the interface $\{x_1^2 + x_2^2 \geq 1,\ x_3=1\}\cap B(\vec 0,3)$
between $f$ and $d$ is mapped to a portion of the plane $\{\vec y\in\R^3: y_3=0\}$ via
$$
	\vec e_3 + \rho\vec e_r(\theta)
	\mapsto
	u_\rho(\rho, \frac{\pi}{2}) \vec e_r(\theta),
	\quad \rho\geq 1,\quad \theta \in [0,2\pi]
$$ 
with $u_\rho(1,\frac{\pi}{2})=0$.

It remains to describe \( \vec u\) in region \( d:= \{ 0 \leq x_3\leq 1\}\cap B(\vec 0,3)\). 
Let $\vec g$ be a fixed axisymmetric bi-Lipschitz map from
$$
\{\hat r\vec e_r(\theta) + x_3\;\vec e_3: \hat r\geq 0,\ 0\leq \theta\leq 2\pi,\ 0\leq x_3\leq 1\}
$$
onto
$$
\{s\, \vec e_r(\theta) + z\;\vec e_3: s\geq 0,\ 0\leq \theta\leq 2\pi,\ 0\leq z\leq 3\},
$$
with:
\begin{itemize}
 \item
	\begin{equation}
  \label{eq:rhoRPrime-1}
		s(\hat r,1):=u_\rho(\hat r,\frac{\pi}{2}),\quad z(\hat r,1)\equiv 3, \quad \hat r\geq 1
	\end{equation}
	for $s=s(\hat r,x_3)$ the radial distance of $\vec g \big (\hat r \vec e_r(\theta) + x_3\vec e_3\big )= s\,\vec e_r(\theta)+z\,\vec e_3$, the function $r\mapsto u_\rho(\hat r,\frac{\pi}{2})$ being defined in region $f$;
	\item
	\begin{equation*}
		s(\hat r ,0):=\hat r-1,\quad z(\hat r,0)\equiv 0, \quad \hat r\geq 1;
	\end{equation*}

	\item the points $$A'(\hat r=1, x_3=0),\quad  B'(\hat r=0, x_3=0),\quad  C'(\hat r=0, x_3=1),\quad D'(\hat r=1, x_3=1),$$
being sent, respectively, to 
$$(s=0,z=0),\quad (s=0,z=1),\quad (s=0,z=2),\quad (s=0,z=3);$$ 
\item and $\vec g$ affine in the segments joining those points.
\end{itemize}

	The limit map in region $d$ is given by 
$$
	\vec u \big ( \hat r\vec e_r(\theta) + x_3\,\vec e_3\big ) 
	= s\,\sin \big (\varphi(z)\big)\,\vec e_r(\theta) 
	- s\,\cos \big ( \varphi(z)\big ) \,\vec e_3,
	\quad \hat r\geq 0,\  0\leq x_3\leq 1,
$$
with $\varphi=\varphi(z(\hat r,x_3))$,  $s=s(\hat r,x_3)$, and $z=z(\hat r,x_3)$ defined through the relations
$$
\varphi(z):= \frac{\pi}{4}\bigg ( 1 + \frac{z}{3} \bigg ),\quad 
s\,\vec e_r(\theta) + z\,\vec e_3 = \vec g \big ( \hat r \vec e_r(\theta) + x_3\vec e_3\big).
$$
In particular, note that the polygonal line $A'B'C'D'$ is contracted to a single point ($s=0$) and that the slab $d$ is deformed onto the angular sector $0\leq\varphi\leq \frac{\pi}{4}$. Also, the definition of $\vec u$ matches that of the region $f$ at the interface $\hat r \geq 1$, $x_3=1$, and that of the region $b$ at the interface $\hat r\geq 1$, $x_3=0$.

\subsection{Definition of the new recovery sequence $\vec u_\e$}\label{sec:recovery_sequence}

Once again, the definition will be based on a partition of the domain into several regions, as illustrated in Figure \ref{fig:regions-u_eps}.

\begin{figure}
\centering
\subfigure{
 \begin{overpic}[width=.35\linewidth]{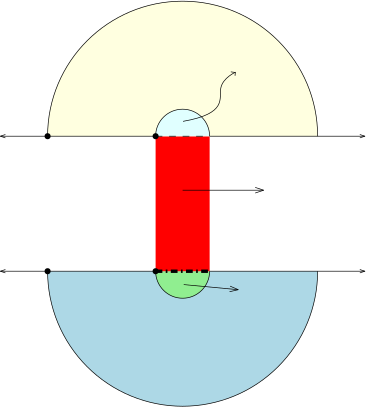}
\put (43,16) {$a_\e$}
\put (59,27) {$a_\e'$}
\put (67,1) {$b$}
\put (66,52) {$c_\e$}
\put (74,47) {$d_\e$}
\put (32,83) {$e_\e$}
\put (59,82) {$e_\e'$}
\put (67,95) {$f$}
\put (9,35) {$A$}
\put (33,35) {$B$}
\put (33,61) {$C$}
\put (9,61) {$D$}
\end{overpic}}
\hspace{1cm}             
\subfigure{
\begin{overpic}[width=.41\linewidth]{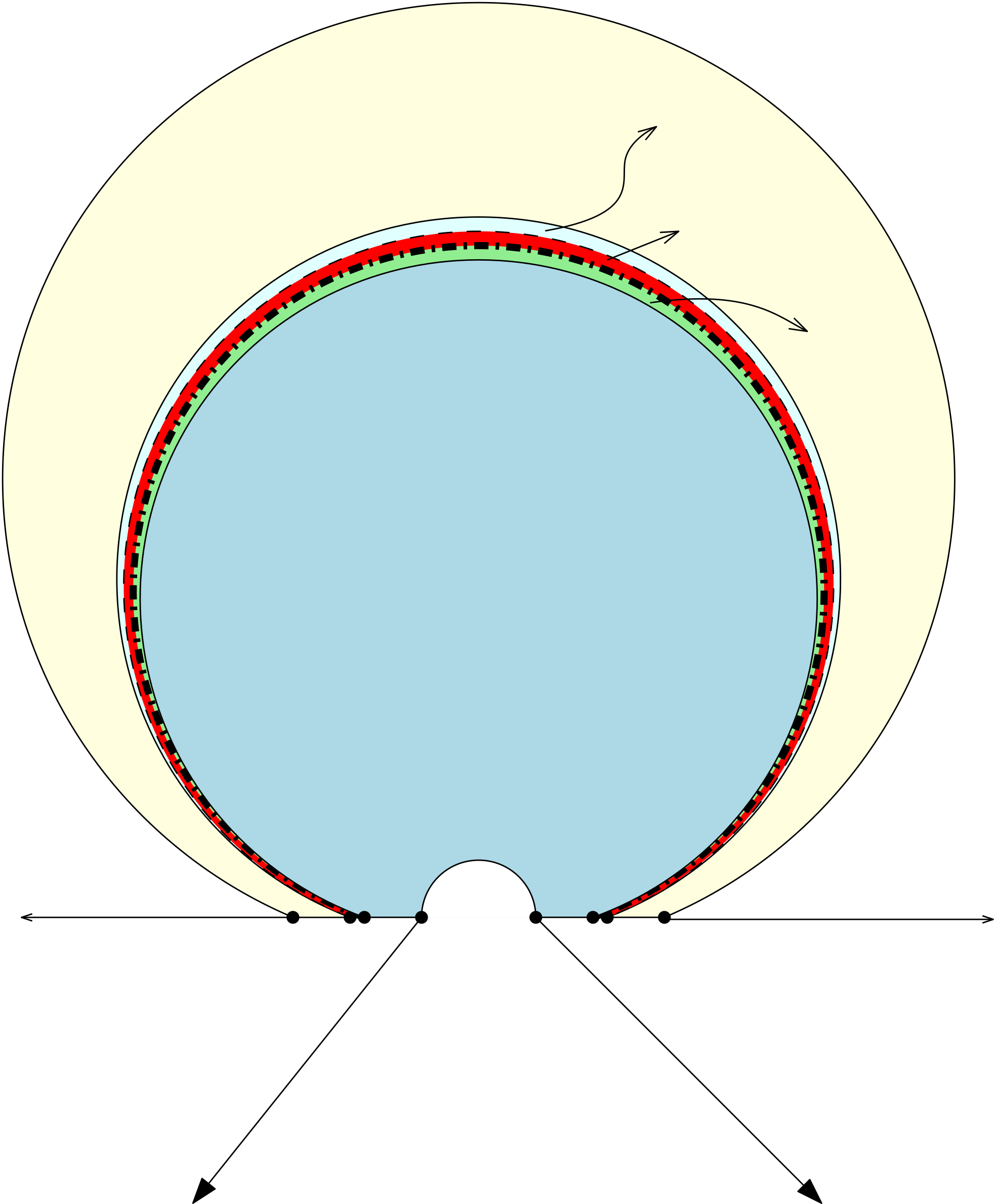}
\put (39,60) {$a_\e$}
\put (68,71) {$a_\e'$}
\put (39,10) {$b$}
\put (58,80) {$c_\e$}
\put (66,12) {$d_\e$}
\put (25,86) {$e_\e$}
\put (56,88) {$e_\e'$}
\put (67,95) {$f$}
\put (34,19) {$A$}
\put (29,19) {$B$}
\put (25,19) {$C$}
\put (20,19) {$D$}
\end{overpic}}

\vspace*{1cm}

\caption{Reference and deformed configurations for the map $\vec u_\e$}
\label{fig:regions-u_eps}
\end{figure}

\subsubsection{Motivation for our choice of recovery sequence}
 
 If \( (\vec u_n)_n \subset \B_s\) is a sequence satisfing INV and \( \vec u_n \rightharpoonup \vec u\) in \(H^1(\Om,\R^3)\) it might be that \(\vec u\) does not satisfy INV because it is not true that \(H^1\) embeds in \(C^0\) for two-dimensional domains and hence the degree is not continuous for the weak \(H^1\) convergence (for condition INV, we look at the degree of \( \vec u\) restricted to $2D$ spheres). The classical example of a sequence showing the non-continuity of the degree for the weak convergence in \(H^1\) is known as the bubbling-off of spheres and is given by 
 \begin{equation}\label{eq:bubble_sphere}
 \BB_n: \Sf^2 \rightarrow \Sf^2, \quad \BB_n(\vec x)=\pi_S^{-1}(n\pi_S(\vec x))
 \end{equation}
where we have denoted by \(\pi_S\) the stereographic projection from the South pole to the plane orthogonal to \(\vec e_3\) and passing through the North pole \( \{ x_3=1\}\). Note that \( \BB_n\) is conformal and \(\BB_n\) does not converge uniformly to \( (0,0,-1)\) or strongly in \(H^1\), and that the Dirichlet energy \(\int_{\Sf^2} |D \BB_n|^2 \dd \vec x\) concentrates near the North pole.

Another way to see conformality appearing in our problem is to observe that a key point to obtain the lower bound on the relaxed energy
in the axisymmetric case (\cite[Th.\ 1.1]{BaHeMoRo_21}) is the following ``area-energy'' inequality
\begin{equation}\label{eq:area_energy_inequality}
|(\cof D \vec u) \vec e_3|\leq \frac12 |D \vec u|^2.
\end{equation}
Let us recall briefly the proof of \eqref{eq:area_energy_inequality}. Since \( \vec e_1 \wedge \vec e_2= \vec e_3\), by using the properties of the cofactor matrix and Cauchy's inequality we find that 
\begin{align*}
|(\cof D \vec u) \vec e_3| &=|(\cof D \vec u) \vec e_1 \wedge \vec e_2|= |(D \vec u)\vec e_1 \wedge (D \vec u) \vec e_2| \\
& \leq |\p_{x_1}\vec u| |\p_{x_2}\vec u| \leq \frac12 (|\p_{x_1}\vec u|^2+|\p_{x_2} \vec u|^2) \leq \frac12 |D \vec u|^2,
\end{align*}
where we recall that we are using the Frobenius norm of a matrix.
From this proof we see that there is equality in \eqref{eq:area_energy_inequality} if and only if 
\[ \p_{x_3} \vec u=0, \quad  \p_{x_1} \vec u \cdot \p_{x_2} \vec u =0, \quad |\p_{x_1} \vec u|=|\p_{x_2} \vec u| .\]
The two last conditions mean that \(\vec u\) restricted to the plane \((\vec e_1,\vec e_2)\) is locally conformal at \( \vec x = (x_1,x_2,x_3)\). 

The two previous paragraphs indicate that to construct a sequence showing the lack of compactness in \(\Ar_s\) with optimal loss of energy for \(E\) we must construct a sequence such that both \( \cof D \vec u_n\) and \(D \vec u_n\) concentrate, and that inequality \eqref{eq:area_energy_inequality} becomes asymptotically an equality, thus involving conformality. Note that, because of the axisymmetry, the only place where the sequence can concentrate is the symmetry axis, as shown in \cite{HeRo18}. To construct our recovery sequence we will use the maps \(\BB_n\) in \eqref{eq:bubble_sphere}. We can use spherical coordinates and see that \( \BB_n\) is given by
\begin{equation*}
(\BB_n)_\rho =1, \quad (\BB_n)_\theta = \theta, \quad (\BB_n)_\varphi=2\arctan\left(n \tan \frac{\varphi}{2} \right).
\end{equation*}
The important information is carried by the zenith angle \(( \BB_n)_\varphi\). If we see the bubble as a map from \(\R^2\) 
to \(\Sf^2\), elementary geometric relations show that \(( \BB_n)_\varphi=2\arctan (nr)\) where \((r,\theta)\) are the polar coordinates. However, we want to construct a bubble with values onto the sphere \( S((0,0,\frac12),\frac12)\) since it is what is done in \cite{CoDeLe03}. In that case, by using the central angle theorem, since the origin of the sphere is at \((0,0,\frac12)\) we can see that we must take a modified sequence
\begin{equation*}
(\tilde{\BB}_n)_\rho=2\cdot \frac12 \cos((\tilde{\BB}_n)_\varphi), \quad (\tilde{\BB}_n)_\theta=\theta, \quad (\tilde{\BB}_n)_\varphi =\arctan(n r).
\end{equation*}

We also observe that the image of a disk of radius \(1/n\) in the plane \( x_3=0\) by \( \pi_S^{-1}(n\cdot)\) is the upper hemisphere of \(\Sf^2\). However, the image of a disk of radius \( \frac1n\) in the plane \(x_3=0\) by \( \pi_S^{-1}(n^2\cdot)\) is almost the all sphere \(\Sf^2\).  We arrive at the conclusion that the map near our set of concentration should look like
\[ (\vec u_n)_\rho\approx \cos (\vec u_n)_\varphi, \quad (\vec u_n)_\theta= \theta, \quad (\vec u_n )_\varphi \approx \arctan(n^2 r).\]
This is to compare with the construction of Conti--De Lellis, where \( (\vec u_n )_\varphi=2\arctan(n r)\). This latter map is not conformal because the sphere we want to use for the bubbling is not centred at the origin. 

Another difference in our construction is that we will require that \(\vec u_n\) is incompressible near the set of concentration. This is because we want to emphasize that the lack of compactness of the problem is due to the Dirichlet part of the neo-Hookean energy and not to the determinant part. Hence our construction is valid for quite a general choice of the convex function \(H\).

For notational simplicity we set \( \e=1/n\) for \(n \in \N\) with $n\geq 1$. From what precedes we can understand that the following function plays the main role in the construction:
\begin{align}
         \label{eq:def_stereo}
    f_\e (r):=\arctan \Big ( \frac{r}{\e^2} \Big ) + \alpha_\e \frac{r}{\e},
        \quad 0\leq r \leq \e,
        \qquad 
    \alpha_\e:=\arctan(\e).
\end{align}

\begin{figure}[hbt!]
\begin{center}
\begin{overpic}[width=.9\textwidth]{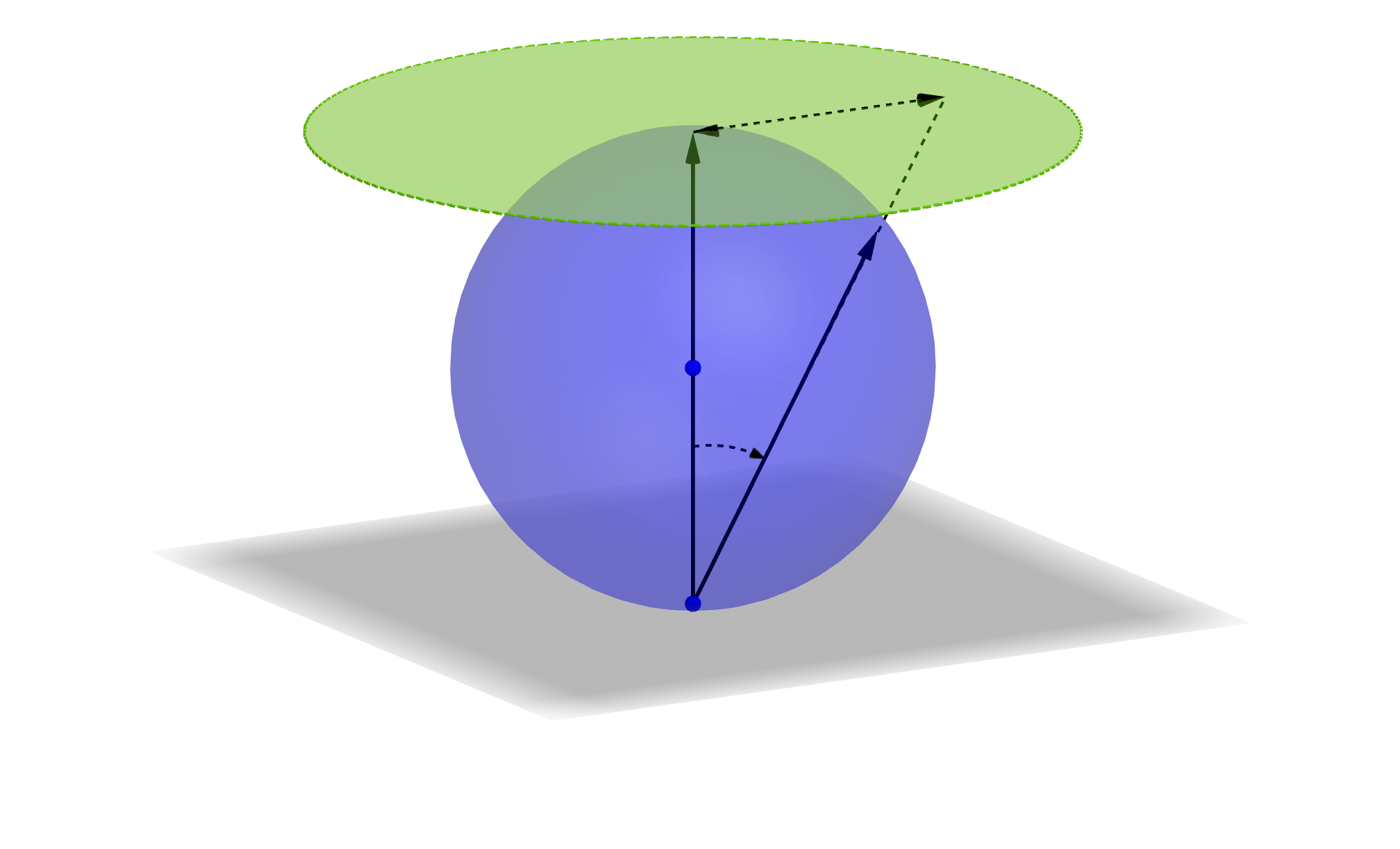}
\put (55,54) {$r/\e^2$}
\put (40,34) {$(0,0,\tfrac{1}{2})$}
\put (51,30) {$\arctan(r/\e^2)$}
\put (46,15) {$(0,0,0)$}
\end{overpic}
\end{center}

\vspace*{-1cm}

\caption{Conformal transformation of an $\e$-disk onto the sphere, via the stereographic projection}
\label{fig:stereo}
\end{figure}

Note that 
\begin{equation*}
    f_\e(0)=0, \quad f_\e(\e)=\frac{\pi}{2},
    \quad \text{and}\quad f_\e'(r)> 0\quad \text{for all }r.
\end{equation*}
The first term in $f_\e(r)$ may be interpreted as a function that 
stretches the disk $\{\vec x\in \R^3:\ x_1^2+x_2^2 < \e^2,\ x_3=0\}$
onto 
the disk $\{\vec y\in\R^3:\ y_1^2 +y_2^2 < \e^{-2},\ y_3=1\},$
so as to subsequently wrap that disk (conformally, via the stereographic projection, see Fig.~\ref{fig:stereo}) 
onto (a very large part of) the sphere \(S((0,0,\frac12),\frac12)\).

The function $f_\e(r)$ will correspond to the angle formed between the positive $y_3$ axis and
the segment
joining the origin with a point on that bubble. 
The correction term
$\alpha_\e \frac{r}{\e}$ in $f_\e(r)$,
which is chosen to be linear for simplicity in the calculations,
has the effect of wrapping the disk onto the whole bubble (all the way up to $f_\e(r)=\frac{\pi}{2}$) and not only to the large part consisting of all points with zenith angle between $0$ and $\frac{\pi}{2} - \alpha_\varepsilon$.
From now on, a choice is made of
\begin{align*}
0<\gamma \leq  \frac{1}{3} \quad \text{a fixed positive exponent.}
\end{align*} 

\subsubsection{Region $c_\e$} 

We start by describing our recovery sequence in region \[c_\e:=\{ x_1^2+x_2^2 < \e^2, 0 <x_3<1 \}.\] This is where concentration of energy will occur. In this region we use cylindrical coordinates in the domain, that is, we write \( \vec x (r, \theta, x_3) = r\vec e_r + x_3 \vec e_3, \ 0\leq r < \e,\ 0\leq \theta \leq \pi, \ 0<x_3<1. \) 
In most of the regions we describe \(\vec u\) via its spherical coordinates, that is, we set
 \begin{equation}   \label{eq:def_u_eps}
 \vec u_\e\big (\vec x(r,\theta,x_3)\big ) = u^\e_\rho \sin u^\e_\varphi\, \vec e_r(\theta)
    +  u^\e_\rho \cos u^\e_\varphi\,\vec e_3.
 \end{equation}
 In region \(c_\e\) we take
 \begin{equation}
 \label{eq:def-u_rho}
 u^\e_\rho(r, x_3) = \Big ( \big (\cos \big (f(r)\big ) + 2\e^\gamma\big )^3 
    + x_3\cdot \frac{3r}{\partial_r \big ( - \cos f(r) \big )} \Big )^{1/3},
  \qquad  u^\e_\varphi(r) = f_\e(r).
\end{equation}
These equations may be regarded as a perturbation of
$u_\rho = \cos u_\varphi$ and \(u_\varphi=\arctan(\frac{r}{\e^2})\), namely, the equations of a bubbling sequence (see Fig.~\ref{fig:bubbling-c})
from a disk of size \(\e\) perpendicular to the symmetry axis to the sphere \(S((0,0,\frac12),\frac12)\).

\begin{figure}[hbt!]
\centering
\begin{minipage}{.48\linewidth}
	\centering
  \includegraphics[width=\linewidth]{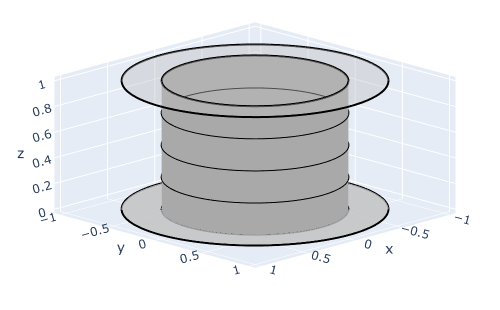}

  {\small a) Reference configuration $c_\e$}
\end{minipage}
\hspace{.02\linewidth}
\begin{minipage}{.48\linewidth}
	\centering
  \begin{overpic}[width=\linewidth,tics=10]{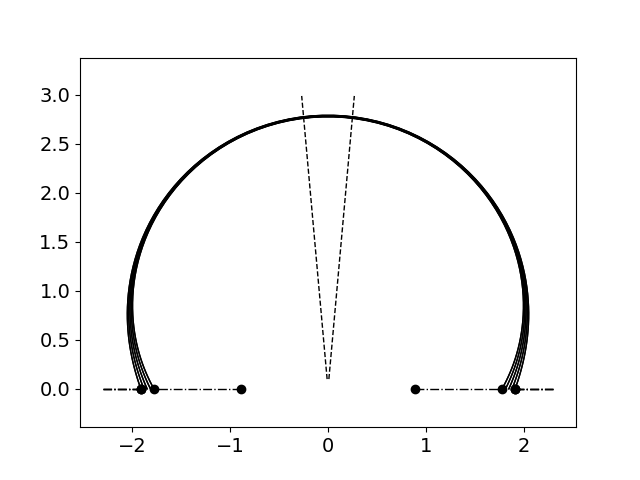}
    \put (36,10) {\footnotesize $A$}
    \put (23,10) {\footnotesize $B$}
    \put (20,10) {\footnotesize $C$}
  \end{overpic}

  {\small b) Axisymmetric deformed configuration}
\end{minipage}


\begin{minipage}{.48\linewidth}
	\centering
  \includegraphics[width=\linewidth]{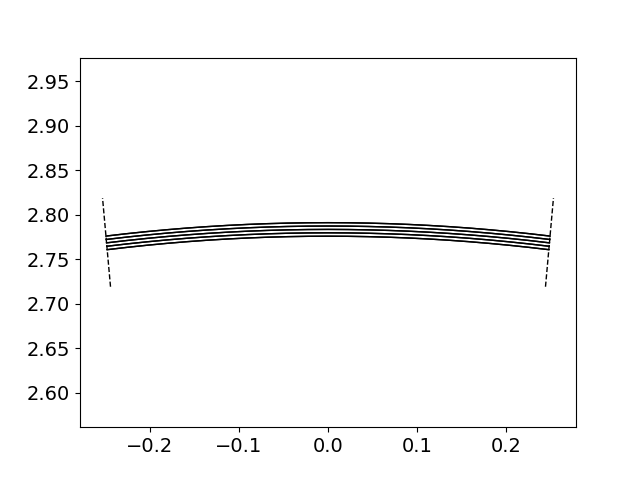}

  {\small c) Close-up view of the stack of spheres}
\end{minipage}

\smallskip

\caption{Illustration of the deformation in the key region $c_\e$, where the singular energy originates. Even for the exaggeratedly large value of $\e=0.7$ used for these plots, the images of the disks $B^2(\vec 0, \e)\times \{x_3\}$, taken at different heights $x_3$ between 0 and 1, are almost indistinguishable. As $\e$ becomes smaller, the polar coordinates $\e^\gamma$ and $2\e^\gamma$ of the deformed points $A$ and $B$ are increasingly small, and the image of
each of the disks resembles more and more the sphere $S\big ( (0,0,\frac{1}{2}), \frac{1}{2}\big )$. The bubbling effect can also begin to be appreciated, since the angular sector $|u_\varphi|<\frac{\pi}{50}(1+ \e\arctan(\e))$ that is zoomed out in Figure c) comes from the much smaller disks
$B^2(\vec 0, \frac{\pi}{50}\e^2)\times \{x_3\}$.
Correspondingly, when $0<x_1^2+x_2^2<\e^{2}$ the huge tangential stretch $\frac{\partial \vec u_\e}{\partial r}$ is of order $\e^{-2}$, and the normal compression $\frac{\partial \vec u_\e}{\partial x_3}$ is of order $\e^4$}
\label{fig:bubbling-c}
\end{figure}

The complicated expression for $u_\rho^\e$ arises as a solution of the incompressibility equation. Indeed, since in this construction $u_\varphi^\e$
is independent of $x_3$, 
from \eqref{eq:det_spherical_cylindrical} it follows that 
\begin{align*}
    \det D\vec u_\e\big ( \vec x (r,\theta,x_3)\big ) 
    = \frac{1}{3}
    \frac{\sin \big (f_\e(r)\big)\,f_\e'(r)}{r} \partial_{x_3} \big ( (u_\rho^\e)^3\big )
    \equiv 1.
\end{align*}
The ``initial condition'' 
\begin{equation*}
    u_\rho^\e(r,x_3) = \cos\big ( f_\e(r)\big)  + 2\e^\gamma
    \quad\text{at}
    \quad
    x_3=0
\end{equation*}
departs from the bubble $u_\rho=\cos u_\varphi$ only because of the 
small term $2\e^{\gamma}$, which is added 
for reasons of technical convenience that will become evident in the sequel (see, e.g., \eqref{eq:use_ofadding_eps_gamma}).

\subsubsection{Region $a_\e'$}
We call $a_\e'$ the region $$a_\e':=\{(x_1,x_2,x_3): x_1^2+x_2^2+x_3^2<\e^2,\, x_3< 0\}.$$  We parametrize this region by
$$
	\vec x(s,\theta,\varphi) = (1-s) g(\varphi)\vec e_r(\theta) + s\big ( \e \sin\varphi \vec e_r(\theta) - \e \cos \varphi \vec e_3\big ),\quad 0\leq s\leq 1,\ 0\leq \theta <2\pi,\ 0\leq \varphi < \frac{\pi}{2},
$$
where 
\begin{equation}\label{eq:def_g_eps}
g_\e:[0,\frac{\pi}{2}]\to [0,\e] \text{ is the inverse of the function } f_\e \text{ defined in } \eqref{eq:def_stereo}.
\end{equation}

\begin{figure}[hbt!]
\begin{center}
\begin{overpic}[width=0.5\textwidth,tics=10]{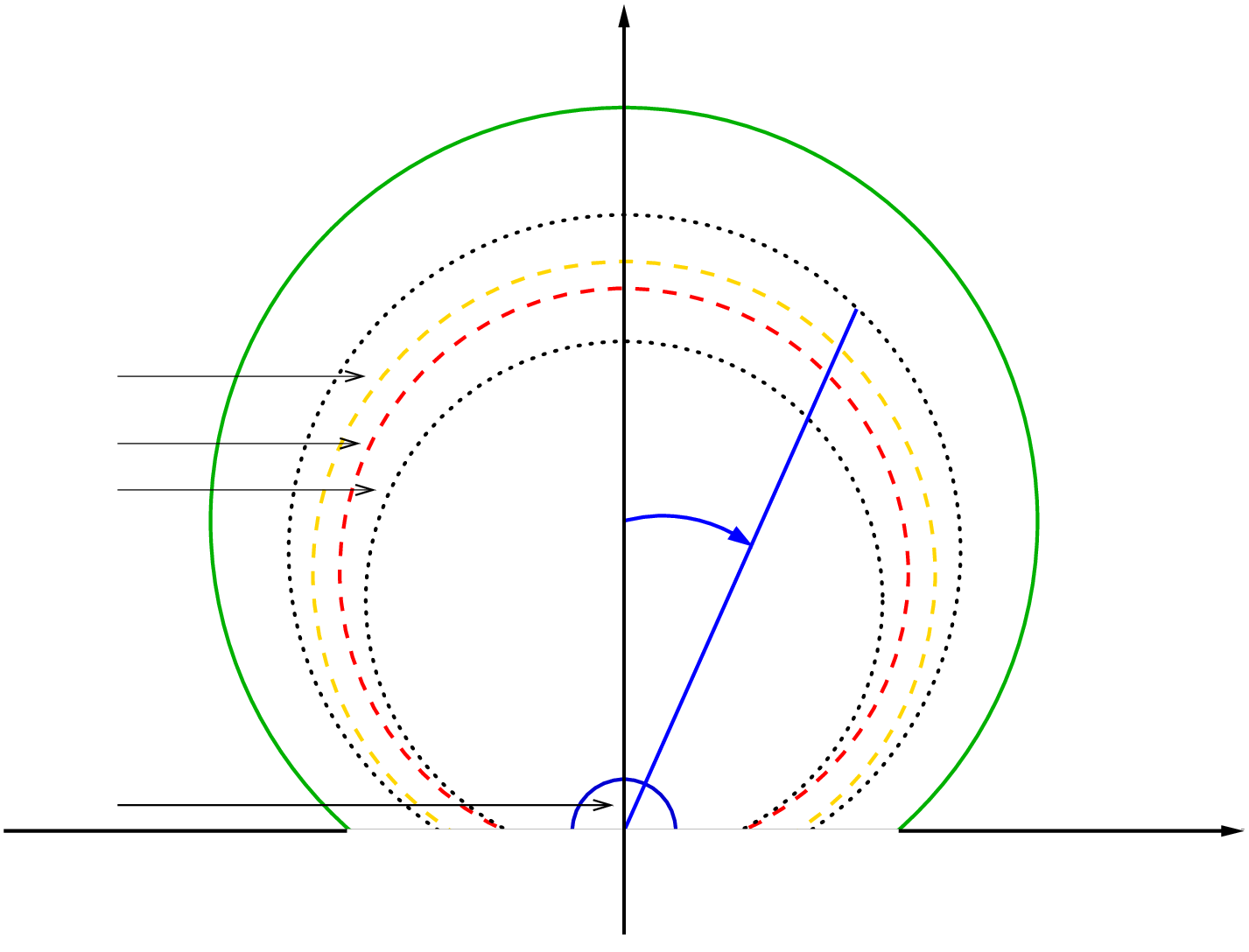}
 \put (54,37) {$\varphi$}
 \put (33, 26) {\small Region}
 \put (39, 22) {\small $a_\e$}
 \put (4, 35) {\small $a_\e'$}
 \put (4, 40) {\small $c_\e$}
 \put (4, 45) {\small $e_\e'$}
 \put (25, 52) {\small $e_\e$}
 \put (16, 67) {\small Region}
 \put (21, 62) {\small $f$}
 \put (6, 11) {\small $b$}
\end{overpic}
\end{center}
\caption{Schematic representation of the image of regions $a_\e$, $a_\e'$, $c_\e$, $e_\e'$, and $f$ after the deformation $\vec u_\e$. The parameter $\varphi$ in the definition of $\vec u_\e$ corresponds to the zenith angle in the deformed configuration}
\label{fig:regions_eps}
\end{figure}

For each fixed $\theta, \varphi$ the parametrization consists of an affine interpolation (with parameter $s$) between the preimage 
on the disk $\{r\leq\e,\, x_3=0\}$ of the point on the bubble with zenith angle $\varphi$ (which is determined by the function $f(r)$ that characterizes the conformal mapping from the stack of horizontal disks in region $c_\e$ to the corresponding stack of ``copies'' of the bubble)
and the point on the sphere $\{r^2+x_3^2=\e^2,\, x_3\leq 0\}$ that (according to the eversion-type map $\vec u_\e$ 
to be defined in region $a_\e$) will be assigned the same zenith polar angle $\varphi$ (see Fig.~ \ref{fig:regions_eps}).

We define $\vec u_\e$ in region $a_\e'$ through its spherical coordinates
\begin{align*}
	u_\rho^\e(s,\varphi) := \Big ( \big ( \cos\varphi + 2\e^\gamma\big ) ^3 - 3\int_{\sigma=0}^s h_\e(\sigma,\varphi)\dd \sigma \Big ) ^{\frac{1}{3}},
	\qquad u_\varphi^\e(s,\varphi) &:= \varphi,
\end{align*}
with
\begin{equation}
 \label{eq:def_h}
	h_\e(s,\varphi):= \e \Big ( (1-s) \frac{g_\e(\varphi)}{\sin\varphi}  +s\e \Big ) \Big ( (1-s) g_\e'(\varphi)\cos\varphi
	+ s (\e-g_\e(\varphi)\sin\varphi)\Big ).
\end{equation}

\underline{Derivatives of $\vec u_\e$:} they  are given by 
\begin{align*}
	\partial_s \vec u_\e = \partial_s u_\rho^\e\, \vec e_\rho(\theta,\varphi),
	\quad
	\partial_\theta \vec u_\e = u_\rho^\e \sin\varphi\, \vec e_\theta(\theta),
	\quad
	\partial_\varphi \vec u_\e = \partial_\varphi u_\rho^\e\,\vec e_\rho(\theta,\varphi) + u_\rho^\e\,\vec e_\varphi(\theta,\varphi),
\end{align*}
with
\begin{equation*}
		\vec e_\rho(\theta,\varphi)= \sin\varphi\, \vec e_r(\theta) +\cos\varphi \vec e_3,
		\quad
		\vec e_\varphi(\theta,\varphi)= \cos\varphi\,\vec e_r(\theta)- \sin\varphi \vec e_3.
\end{equation*}
We now prove that with this definition \( \det D \vec u_\e=1\) in \(a_\e'\).

\underline{Derivatives of the parametrization of the reference domain:}
\begin{align*}
	\partial_s\vec x &= (\e \sin\varphi - g_\e(\varphi))\vec e_r -\e \cos\varphi\,\vec e_3 
	\\
	\partial_\theta \vec x &= \Big ( (1-s) g_\e(\varphi) + s\e \sin\varphi \Big ) \vec e_\theta
	\\
	\partial_\varphi \vec x &= \big ( (1-s) g'_\e(\varphi) + s\e \cos\varphi \big ) \vec e_r
		+s\e \sin\varphi\,\vec e_3.
\end{align*}
Using the formulae
\begin{align*}
	A_{1i} \vec a_i \wedge A_{2j} \vec a_j \wedge A_{3k} \vec a_k = (\det A) \vec a_1\wedge \vec a_2 \wedge \vec a_3,
	\qquad \vec e_r\wedge \vec e_\theta \wedge \vec e_3=1,
\end{align*}
it can be seen that
\begin{align*}	
	\partial_s \vec x \wedge \partial_\theta \vec x \wedge \partial_\varphi \vec x
	 =\Big ( (1-s) g_\e(\varphi) + s\e \sin\varphi \Big ) 
	 \begin{vmatrix} \e\sin\varphi  - g_\e(\varphi) -\e\cos\varphi 
		 \\
	 	(1-s)g'_\e(\varphi) + s\e\cos\varphi s\e\sin\varphi 
	 	\end{vmatrix}
	= \sin \varphi\,h_\e(s,\varphi).	 
\end{align*}

\underline{Incompressibility:}
From the standard relation $\vec e_\varphi\wedge \vec e_\theta \wedge \vec e_\rho=1$ in spherical coordinates, it is clear that 
\begin{align*}
	\partial_s \vec u_\e \wedge \partial_\theta \vec u_\e \wedge \partial_\varphi \vec u_\e 
	&= - (u_\rho^\e)^2 \partial_s u_\rho^\e \sin \varphi = -\frac{1}{3}\partial_s \Big ( (u_\rho^\e)^3\Big ) 
	\sin \varphi
	= h_\e(s,\varphi) \sin \varphi. 
\end{align*}
On the other hand, using that the Jacobian of a composition is the product of the Jacobians, for $\vec F:= D\vec u_\e(\vec x)$, we have that
\begin{equation*}
	\partial_s \vec u_\e \wedge \partial_\theta \vec u_\e \wedge \partial_\varphi \vec u_\e
	=
	\vec F \partial_s \vec x \wedge \vec F \partial_\theta \vec x \wedge \vec F \partial_\varphi \vec x
	= \det D\vec u_\e \cdot \partial_s \vec x \wedge \partial_\theta \vec x \wedge \partial_\varphi \vec x
	= (\det D\vec u_\e) h_\e(s,\varphi) \sin \varphi.
\end{equation*}
Therefore $\det D\vec u_\e(\vec x)=1$ for all $\vec x \in a_\e'$. 
For this, it is necessary to know that $h_\e(s,\varphi)$ is nonzero; this is proved in Lemma \ref{le:positive_h} below.

The formula 
$$
	-\frac{1}{3}\partial_s \Big ((u_\rho^\e)^3\Big ) 
	\sin \varphi = (\det D\vec u_\e) h_\e(s,\varphi) \sin \varphi
$$
and the value of $u_\rho^\e$ at the disk $\{r\leq\e,\, x_3=0\}$
\[
 u_\rho^\e (r, \theta) = \cos f(r) + 2\e^\gamma
\]
(which is prescribed by the construction of $\vec u_\e$ in the critical region $c_\e$)
explain the definition of $u_\rho^\e$ in the region $a_\e'$. 
Although the construction here is more elaborate, it is based on the technique of using direction-preserving deformations developed in \cite{HeSe13,HeMoXu15} to solve the incompressibility constraint.
More precisely, this construction imposes that segments go to segments, which makes the incompressibility equation easily solvable.
Indeed, the incompressibility equation when $\vec u^\e$ is direction-preserving only involves $\partial_s (u_\rho^\e)$, so it can be integrated.
Since we also impose conformality, the component $u_\varphi$ can be recovered and, hence, the equation can be solved explicitly for $\vec u$.

\subsubsection{Region $a_\e$}

We set 
$$
    a_\e:= \{ 
    (x_1, x_2, x_3): 
    \e^2 < 
    x_1^2 + x_2^2 + x_3^2
    <1,\ 
    x_3< 0\}.
$$
In this region we use spherical coordinates \( (\rho,\theta,\varphi)\), \( \e<\rho <1, \ 0\leq \theta <2 \pi, \ \pi/2\leq \varphi \leq \pi\) in the domain and we define
 $\vec u_\e$ in region $a_\e$ through
its spherical coordinates
\begin{equation}
    \label{eq:def-u_eps-region_a}
    \begin{split}
 u_\rho^\e(\rho, \varphi) &:= 
 \frac{1-\rho}{1-\e}
 \Big ( 
 (\cos (\pi-\varphi) + 2\e^\gamma)^3 
 - 3\int_{\sigma=0}^1 h_\e(\sigma, \varphi)\dd\sigma\Big )^{\frac{1}{3}}
 + \e^\gamma \frac{\rho-\e}{1-\e},
 \\
 u_\varphi^\e (\rho,\varphi) &:= \pi-\varphi,
\end{split}
\end{equation}
where $h_\e(\sigma, \varphi)$ is the function defined in  \eqref{eq:def_h}.
 
Regarding $u_\rho^\e$, it is an affine interpolation connecting, on the one hand, the radial distance at the image of the interface between $a_\e'$ and $a_\e$
(which is mapped to a ``copy'' of the bubble lying slightly beneath it),
and, on the other hand,
the radial distance $\e^\gamma$ at which all points on the reference lower unit hemisphere are mapped to. By Lemma \ref{le:hE2Cos} below, the a.e.\ limit of $\vec u_\e$ in region $a_\e$ is the Conti--De Lellis map.

\subsubsection{Region $b$}
We recall that this region is described by 
\[ b=\{ \rho \sin \varphi \vec e_r(\theta) +\rho \cos \varphi \vec e_3 : 1\leq \rho <3, \ \frac{\pi}{2}\leq \varphi \leq \pi\}.
\]

Here a slight modification with respect to the original Conti--De Lellis maps is in order since in our construction we are now changing the scale of the spacings between the image of $A(\rho=1, x_3=0)$, $B(\rho=\e, x_3=0)$, $C(\rho=\e, x_3=1)$, and $D(\rho=1, x_3=1)$ to $\e^\gamma$ with $\gamma\leq 1/3$ instead of just $\e$. 

Working with spherical coordinates $(\rho, \theta, \varphi)$ in the reference configuration and cylindrical coordinates in the deformed configuration, we first define the auxiliary function $\vecg\phi_\e$ by 
\begin{align*}
	\begin{aligned}
		\phi_r^\e (\rho, \varphi) &= (\rho - 1 +\sqrt{2}\e^\gamma ) \sin \left(\frac{\varphi + \pi}{2} \right) , \\
		\phi_3^\e (\rho, \varphi) &= \e^\gamma + (\rho - 1 + \sqrt{2}\e^\gamma) \cos \left( \frac{\varphi + \pi}{2} \right) ,
	\end{aligned}
	\qquad 
	\rho \geq 1,
	\quad
	\frac{\pi}{2}\leq \varphi \leq \pi.
\end{align*}
The factor $\sqrt{2}$ appears because, according to the definition of $\vec u_\e$ in region $a_\e$, the image of $A$ is
$(\e^\gamma, 0)$, whose distance to $(0,\e^\gamma)$ is $\sqrt{2}\e^\gamma$. In region $b$ we define $\vec u_\e$ to be 
$$\vec u_\e (\vec x) = \e^\gamma \vecg\psi \Big (\e^{-\gamma} \vecg \phi_\e(\vec x) \Big ),$$
where $\vecg\psi$ is any axisymmetric bi-Lipschitz bijection from
$$
	\Big \{(x_1, x_2, x_3):\ r^2 + (x_3-1)^2 \geq 2,\ x_3\leq 0,\ r\leq 1 + |x_3|\Big \}
$$
onto
$$
	\Big \{(x_1, x_2, x_3):\ x_3\leq 0,\ r\leq 1 + |x_3|\Big \} \cup B\big ( (0,0,0),\, 1\big )
$$
such that
\begin{enumerate}[i)]
	\item $\vecg\psi(r, x_3)=(r,x_3)$ on the half-line $r=1+|x_3|$, $x_3\leq 0$,
	\item $\vecg\psi \big ( \sqrt{2}\sin (\bar\varphi), 1 + \sqrt{2}\cos (\bar\varphi) \big ) 
		= \Big ( \sin \big ( 2(\pi -\bar\varphi)\big ), \cos \big ( 2(\pi - \bar \varphi)\big ) \Big ) $,
		 $\frac{3\pi}{4}\leq \bar\varphi \leq \pi$,
	\item $\vecg\psi\equiv \id$ in $\{\vec x:\ r^2 +(x_3-1)^2 \geq 8,\ x_3\leq 0,\ r\leq 1+|x_3|\}.$
\end{enumerate}

The first two properties ensure the continuity of $\vec u_\e$ when passing from $b$ to $d_\e$ and to $a_\e$, respectively.
When $\e\to 0$, we can check that the resulting map $\vec u_\e$ converges to the correct limit map.

	\subsubsection{Region $e_\e'$}
The region $e_\e'$ is described by
	\begin{align*}
		e_\e' :=\{(x_1,x_2,x_3): x_1^2 +x_2^2 + (x_3-1)^2 < \e^2,\ x_3>1\}.
	\end{align*}
	We parametrize this region by 
	\begin{multline*}
		\vec x (s,\theta,\varphi)= \vec e_3 + (1-s) g_\e(\varphi) \vec e_r(\theta) + 
			s \big ( \e \sin\varphi \, \vec e_r(\theta)
				+\e\cos\varphi\,\vec e_3\big ),
		\\
			0\leq s\leq 1,\quad 0\leq \theta < 2\pi,\quad 0\leq \varphi < \frac{\pi}{2},
	\end{multline*}
	where, as in region $a_\e'$, the function $g_\e:[0,\frac{\pi}{2}]\to[0,\e]$ is defined by \eqref{eq:def_g_eps}.

	For each $\theta$, $\varphi$ the parametrization consists of an affine interpolation (with parameter $s$) between the pre-image on the disk
	$\{r\leq \e, x_3=1\}$ of the point on the (outer wall of the) bubble with zenith angle $\varphi$ (which is determined by the function $f(r)$, according to the definition of $u_\varphi^\e$ in region $c_\e$) and the point on the hemisphere $\{r^2 + (x_3-1)^2 = \e^2,\ x_3\geq 1\}$ that (according to the definition of the map in region $e_\e$) will be assigned the same zenith polar angle $\varphi$. 

	Define $\vec u_\e$ in region $e_\e'$ through its spherical coordinates
	\begin{align*}
  	 u_\varphi^\e (s, \varphi) &:= \varphi,
	\\
	 u_\rho^\e (s,\varphi)&:= \Bigg ( 
		 ( \cos \varphi + 2\e^\gamma )^3 
		+ \left [ \frac{3r}{\partial_r \big ( -\cos f_\e(r) \big ) }\right]_{r=g(\varphi)}
 + 3\int_{\sigma=0}^s h_\e(\sigma, \varphi)\dd\sigma
	\Bigg )^{\frac{1}{3}}
	\end{align*}
	with $h_\e(s,\varphi)$ defined as in \eqref{eq:def_h}.
	Note that, with this definition, $\vec u_\e$ is continuous 
at the interface $x_3=1$, $0\leq r\leq \e$, between this region and region $c_\e$.

    \subsubsection{Region $e_\e$}
   Region $e_\e$ is given by
   \[ e_\e= \{ \vec e_3+ \rho \sin \varphi \vec e_r(\theta) + \rho \cos \varphi \vec e_3 : \ \e < \rho< 1, \ 0\leq \varphi \leq \frac{\pi}{2}\} . \]
	We define $\vec u_\e$ in region $e_\e$ through its spherical coordinates
	\begin{align*}
&	u_\varphi^\e (\rho, \varphi) := \varphi,
	\qquad
	 u_\rho^\e (\rho ,\varphi):= \frac{1-\rho}{1-\e} u_\rho^\e(\e,\varphi) + \frac{\rho-\e}{1-\e} (2\cos\varphi + 6\e^{\gamma}),
	\\ &	u_\rho^\e(\e,\varphi)=
\Bigg ( 
		 ( \cos \varphi + 2\e^\gamma )^3 
		+ \left [ \frac{3r}{\partial_r \big ( -\cos f(r) \big ) }\right]_{r=g(\varphi)}
 + 3\int_{\sigma=0}^1 h_\e(\sigma, \varphi)\dd\sigma
	\Bigg )^{\frac{1}{3}},
	\end{align*}
	the function
	$h(\sigma,\varphi)$ being defined in \eqref{eq:def_h}.
	Note that, with this definition, $\vec u_\e$ is continuous 
at the interface $x_1^2 +x_2^2 + (x_3-1)^2=\e^2$, $x_3\geq 1$
	with region $e_\e'$.
	Also, from the relation $(a^3 +b^3)^\frac{1}{3}\leq a+b$, valid for \( a,b\geq 0\), and Lemma \ref{le:estimatesUrhoRegionEprime} it can be seen that 
	\begin{align}
  \label{eq:refinedBoundUrhoRegionE}
	 u_\rho^\e(\e,\varphi) \leq (\cos\varphi +2\e^\gamma) + 4\e^\frac{1}{3}
	< \cos\varphi + 6\e^{\gamma},
	\end{align}
	hence $u_\rho^\e(\cdot, \varphi)$ is an affine interpolation between $u_\rho^\e(\e, \varphi)$ and a value (at $\rho=1$) that is strictly larger (even when $\cos\varphi=0$).

  \subsubsection{Region $f$}
This region can be described by 
\[f=\{ \vec e_3 + \rho \sin \varphi \vec e_r(\theta)+ \rho \cos \varphi \vec e_3 : 1\leq \rho, \ 0 \leq \varphi < \frac{\pi}{2}\}\cap B(\vec 0,3). \]

Let $u_\rho(\rho, \varphi)$, $u_\varphi(\rho,\varphi)$ be the spherical coordinates for the limit Conti--De Lellis map. Recall that in region $f$ this map is not uniquely determined; nothing is imposed on the sequence producing $\vec u$ apart from the image region, the continuity across the interface with region $e$, and the bi-Lipschitz regularity.
Our aim is to prove the upper bound of Theorem \ref{th:upper_bound}
regardless of the specific definition chosen for the limit map $\vec u$ in region $f$.

Define $\vec u_\e$ through its spherical coordinates
	\begin{align*}
	u_\varphi^\e (\rho, \varphi) := u_\varphi(\rho, \varphi),
	\qquad
	 u_\rho^\e (\rho ,\varphi):= u_\rho(\rho,\varphi) + 6\e{\gamma}.
	\end{align*}
	Note that when $\rho=1$ the radial coordinate $$u_\rho^\e(1,\varphi)=u_\rho(1,\varphi) + 6\e^{\gamma}
	= 2\cos\varphi + 6\e^{\gamma}$$
	coincides with
	the definition given in region $e_\e$.

\subsubsection{Region $d_\e$}
Region $d_\e$ is given by 
$$
    d_\e =\{ x_1^2 +x_2^2 >\e^2,\quad 0<x_3<1\}.
$$
We use cylindrical coordinates in the domain in this region. The definition of $\vec u_\e$ must match the definition
already provided in regions $a_\e$, $b$, $c_\e$,  $e_\e$, and $f$.
\begin{itemize}
\item On the interface between $d_\e$ and $b$, the deformation $\vec u_\e$
is already prescribed as:
    $$
    r\geq 1\quad \Rightarrow\quad \vec u_\e\big ( r\,\vec e_r(\theta) \big )
    = \e^\gamma\vec e_r(\theta)+ (r-1) \frac{\vec e_r(\theta) - \vec e_3}{\sqrt{2}}.
        $$
    \item On the interface between $a_\e$ and $d_\e$, the deformation is prescribed as:
    $$
    \e \leq r\leq 1\quad \Rightarrow\quad 
    \vec u_\e \big (r\,\vec e_r(\theta)\big ) = \left( \frac{1-r}{1-\e} (2\e^\gamma) 
        + \e^\gamma \frac{r-\e}{1-\e}\right) \vec e_r(\theta).
    $$
    \item On the interface between $c_\e$ and $d_\e$, the deformation is prescribed as: 
    $$
    0\leq x_3\leq 1
    \quad \Rightarrow\quad
    \vec u_\e\big ( \e \,\vec e_r(\theta) + x_3\;\vec e_3\big) = \left( (2\e^\gamma)^3 
            + x_3\cdot \frac{3\e}{f'(\e)}\right)^{1/3} \vec e_r(\theta),
    $$
    where
    $$
    f'(\e)= 2 - \frac{\e^4}{\e^4 +\e^2} - \frac{\e -\alpha_\e}{\e} = 2+O(\e).
    $$
    \item On the interface between $e_\e$ and $d_\e$, the deformation is: 
	$$
		\e\leq r\leq 1\ \Rightarrow\ 
		\vec u_\e\big (\vec e_3 + r\,\vec e_r(\theta)\big ) = \frac{1-r}{1-\e} \eta_\e + \frac{r-\e}{1-\e} \cdot 6\e^{\gamma},
	$$
	with 
	\begin{equation*}
\eta_\e:=u_\rho^\e(\e,\frac{\pi}{2})= \Bigg ( (2\e^\gamma)^3  + \frac{3r}{\sin f_\e(r)f'_\e(r)}\Bigg |_{r=\e}\Bigg )^{\frac{1}{3}}
	 = \Bigg ( 
		(2\e^\gamma)^3
	 + \frac{3\e}{f'_\e(\e)}
	 \Bigg )^{\frac{1}{3}}.
	\end{equation*}
	This relevant quantity $\eta_\e$ is of order $\e^{\gamma}$ and we recall that we chose \(\gamma\leq 1/3\).

	\item On the interface with $f$, the deformation is prescribed to be
	$$
	 r\geq 1
	 \quad \Rightarrow \quad
	 \vec u_\e\big (\vec e_3 + r \,\vec e_r(\theta)\big )
		= \big (u_\rho(r, \frac{\pi}{2}) + 6\e^{\gamma}\big )\,\vec e_r(\theta).
	$$
\end{itemize}

The map shall be defined by composing three different auxiliary transformations. The first one is given by 
	\begin{equation}
 \label{eq:rPrimeRegionD}
		\hat r=\begin{cases} \displaystyle \frac{(r-\e)r}{\e^{2\gamma}-\e}, & \e < r\leq \e^{2\gamma},\\ r, & r\geq \e^{2\gamma}.
		\end{cases}
	\end{equation}
After this transformation the radial distance lies in the same interval $(0,+\infty)$ for all $\e$. Note also that $\hat r=\e^{2\gamma}$ when $r=\e^{2\gamma}$, so that $\hat r$ is continuous as a function of $r$. The decisions to take it quadratic in $r$, and to separately study the range $\e\leq r\leq \e^{2\gamma}$,  are to control the determinants, as will be seen in Section \ref{se:energyD}. 
    The second transformation
	is the fixed axisymmetric bi-Lipschitz map $\vec g$ used in section \(d\); see Section \ref{se:LimitMap}.
    The third transformation is 
	\begin{align}
 \label{eq:wEps}
		\vec w_\e\big ( s\,\vec e_r(\theta) + z\,\vec e_3\big ) = w_r^\e(s,z)\,\vec e_r(\theta) + w_3^\e(s,z)\, \vec e_3
	\end{align}
	given by
	\begin{equation}
	 \begin{aligned}
		w_r^\e(s,z) &= \omega_\e(z) + s\,\sin \big ( \varphi(z)\big ),
	\label{eq:defWr}
		\\
		w_3^\e(s, z) &= -s\,\cos \big ( \varphi(z)\big ),
	 \end{aligned}
	\qquad \text{with}\quad
\varphi(z):= \frac{\pi}{4} (1+\frac{z}{3}).	
	\end{equation}
Set 
    \begin{equation*}
    \vec u_\e \big ( r\vec e_r(\theta) + x_3 \vec e_3\big ) =
        \vec w_\e \Big ( \vec g(\hat r\vec e_r(\theta) + x_3\,\vec e_3)\Big )
		=w_r^\e(s,z) \vec e_r(\theta) + w_3^\e(s,z)\, \vec e_3,
    \end{equation*}
	where $s=s(\hat r,x_3)$ and $z=z(\hat r,x_3)$
	are the cylindrical coordinates of $\vec g(\hat r\vec e_r(\theta) + x_3\vec e_3)$
	and $\hat r=\hat r(r)$ is that of \eqref{eq:rPrimeRegionD}. 
The function $\omega_\e$ in \eqref{eq:defWr} is chosen such that $\vec u_\e$ takes the prescribed values on the interfaces with $b$, $a_\e$, $c_\e$, $e_\e$, and $f$.
        To be precise, since $\vec g$ is affine in the segments joining the points $A'(\hat r=1, x_3=0)$, $B'(\hat r=0, x_3=0)$, $C'(\hat r=0, x_3=1)$, $D'(\hat r=1, x_3=1)$, then 
        $$\omega_\e(0)=\e^\gamma,\quad \omega_\e(1)=2\e^\gamma,
        \quad \omega_\e(2)=\eta_\e,
        \quad \omega_\e(3)=6\e^{\gamma},$$
        and
        $$
            \omega_\e (\xi) = \begin{cases}
            \xi (2\e^\gamma) + (1-\xi)\e^\gamma, & 0\leq \xi\leq 1,
            \\
            \left( (2\e^\gamma)^3 + (\xi-1)\cdot \displaystyle \frac{3\e}{f'(\e)} \right)^{1/3}, & 
            1\leq \xi\leq 2,
            \\
            (\xi-2)\eta_\e + (3-\e)\cdot (6\e^{\gamma}),& 2\leq \xi\leq 3.
	\end{cases}
        $$
	Note also that, since $s(\hat r, 1)=u_\rho(\hat r,\frac{\pi}{2})$ by \eqref{eq:rhoRPrime-1}
	and $\hat r=r$ when $r\geq 1$ by \eqref{eq:rPrimeRegionD},
	then 
	\begin{align*}
	 \vec u_\e(r\,\vec e_r(\theta)+x_3\vec e_3)= (6\e^{\gamma} + u_\rho(r,\frac{\pi}{2}))\vec e_r(\theta),
\quad r\geq 1, \quad x_3=1 ,
	\end{align*}
	as desired (so that $\vec u$ is continuous at the interface with $f$).
	Analogously,
	\begin{equation*}
	 \vec u_\e(r\,\vec e_r(\theta)+x_3\vec e_3)= (\e^{\gamma} + \frac{r-1}{\sqrt{2}})\vec e_r(\theta) - \frac{r-1}{\sqrt{2}}\vec e_3,
\quad r\geq 1, \quad x_3=0,
	\end{equation*}
	as desired for continuity across the interface with $b$.

        The deformation $\vec u_\e$ depends on $\e$ through $\omega_\e(\xi)$ and 
        through $\hat r=\frac{(r-\e)r}{\e^{2\gamma}-\e}$ for $\e\leq  r\leq \e^{2\gamma}$.

\subsection{Computing the limit of the energies of the approximating sequence}

In this section we compute the limit of the energy of our approximating sequence. We divide our analysis into the several different regions.

\subsubsection{Extra energy in $c_\e$ is $2\pi$.}

The key estimate is that of the energy in this region, since it is where the singular term of the energy completely originates.
Since $\vec u_\e$ maps all discs $B^2(\vec 0,\e)\times \{x_3\}$ onto essentially (see Fig.~\ref{fig:bubbling-c}) the sphere $S\big ( (0,0,\frac{1}{2}), \frac{1}{2}\big )$, it follows that for every fixed $0<x_3<1$,
\[
 \int_{B^2(\vec 0,\e)\times \{x_3\}} |\cof D\vec u_\e| \dd\mathcal H^2 \approx \mathcal H^2\Big ( S\big ( (0,0,\frac{1}{2}), \frac{1}{2}\big ) \Big ) = 4\pi \cdot (\frac{1}{2})^2 = 2\pi.
\]
By integrating with respect to $x_3$, recalling that the normal contraction $\frac{\partial \vec u_\e}{x_3}$ is negligible (see Fig.~\ref{fig:bubbling-c}),
and using that
\[
 |D\vec u_\e|^2 = \Big |\frac{\partial \vec u_\e}{\partial x_1}\Big |^2
 +
 \Big |\frac{\partial \vec u_\e}{\partial x_2}\Big |^2
 +
 \Big |\frac{\partial \vec u_\e}{\partial x_3}\Big |^2
 \approx
 \Big |\frac{\partial \vec u_\e}{\partial x_1}\Big |^2
 +
 \Big |\frac{\partial \vec u_\e}{\partial x_2}\Big |^2
 \geq 2\Big | \frac{\partial \vec u_\e}{\partial x_1} \wedge
 \frac{\partial \vec u_\e}{\partial x_2} \Big |,
\]
it follows that the Dirichlet energy $\int |D\vec u_\e|^2 \dd\vec x$ cannot be small in $c_\e$.

\begin{proposition}
    Let $\vec u_\e$ be as in \eqref{eq:def_u_eps}. Then
    \[
            \lim_{\e\to 0} \int_{c_\e} [| D \vec u_\e|^2+H(\det D \vec u_\e) ] \dd \vec x = 2\pi.
    \]
\end{proposition}

\begin{proof}
First we notice that in the region \(c_\e\) we have that \( \det D\vec u_\e=1\). Hence \( \int_{c_\e} H(\det \vec u_\e )=H(1)|c_\e| \rightarrow 0\) as \(\e\) tends to zero. Then, using \eqref{eq:Dirichlet_energy_cyl_coord},we  observe that the above is equivalent
to proving that
\begin{equation}
        \label{eq:strategy-c}
   \lim_{\e \rightarrow 0} \int_{x_3=0}^1 \int_{r=0}^{\varepsilon} \Big [ 
    \underbrace{r|\partial_r u^\e_\rho|^2 + r|u^\e_\rho\partial_r u^\e_\varphi|^2}_{:=I}  
    + \underbrace{\frac{1}{r}|u^\e_\rho \sin u^\e_\varphi|^2}_{:=II}
        + \underbrace{r|\partial_{x_3} u^\e_\rho|^2}_{:=III} \Big ] \dd r \dd x_3
        =1.
\end{equation}

We remark that  our map satisfies  \(\p_{x_3} \vec u^\e \approx 0\) and \(|\p_r u_\rho^\e|^2+|u^\e_\rho\p_ru^\e_\varphi|^2\approx|u^\e_\rho \sin u^\e_\varphi|^2/r\) which is due to the conformality of \(\vec u_\e\) in the planes orthogonal to \(\vec e_3\). Half of the energy will come from $I$ and half from $II$. The contribution of $III$, as will be shown, is negligible.  More precisely, the claim follows by combining \eqref{eq:strategy-c}
with Lemmas \ref{le:partI},
\ref{le:partII}, and 
\ref{le:partIII}.
\end{proof}

\begin{lemma}
        \label{le:partI}
  Let $u^{\e}_\rho$ and $u^{\e}_\varphi$ be defined as in \eqref{eq:def-u_rho}. Then 
  \begin{align*}
    \lim_{\e\to0}
    \int_{x_3=0}^1 \int_{r=0}^{\varepsilon} r|\partial_r u^{\e}_\rho|^2 + r|u^{\e}_\rho\partial_r u^{\e}_\varphi|^2\dd r \dd x_3
        =\frac{1}{2}.
  \end{align*}
\end{lemma}

The proof consists in establishing $
    \partial_r u^{\e}_\rho
    \approx \partial_r \big (\cos f_\e(r)\big ), \ 
    u_\rho \partial_r u_\varphi^\e  
    \approx \big (\cos f_\e(r) \big ) f'_\e(r),
$
where \(f_\e \) is defined in \eqref{eq:def_stereo}, and in using the following result.

\begin{lemma}
    \label{le:energy_stereo}
    Let $f_\e$ be defined as in \eqref{eq:def_stereo}. Then 
 \[
  \lim_{\e\to0}
    \int_{r=0}^{\varepsilon} rf_\e'(r)^2 \dd r = \frac{1}{2}.
 \]
\end{lemma}

\begin{proof}
  By using that $\alpha_\e=\arctan(\e) \leq \e$ we can write that
    \begin{align*}
    \int_{r=0}^{\varepsilon} rf'_\e(r)^2 \dd r
    &= \int_{r=0}^{\e} \left[\frac{1}{1+\frac{r^2}{\e^4}} \e^{-2} + \frac{\alpha_\e}{\e} \right]^2 r\dd r 
    = \int_{u=1}^{1+\e^{-2}} \left[ \frac{1}{u}\e^{-2} 
        + \frac{\alpha_\e}{\e}\right]^2 \frac{\e^4}{2} \dd u
    \\ &= \frac{1}{2}\left[ \int_{u=1}^{1+\e^{-2}}
    \frac{1}{u^2} + 2\frac{\e \alpha_\e}{u}  + \alpha_\e^2\e^2 \dd u \right] = \frac{1}{2} + O(\e^2 \ln|\e|) .
  \end{align*}

\end{proof}

For notational simplicity we will drop the subscript and superscript \(\e\) in the proofs of the following results. In the proof of Lemma \ref{le:partI}, use shall be made of the following expressions. First,
\begin{equation*}
 \partial_r (u_\rho)^3 = 3(\cos f(r) + 2\e^\gamma)^2 \partial_r \big ( \cos f(r)\big )
 + x_3 \partial_r \left[ \frac{3r}{\partial_r (-\cos f(r))}\right] .
\end{equation*}
Second, since \(
 \partial_r u_\rho = \frac{\partial_r (u_\rho^3)}{3u_\rho^2},
\)
it follows that
\begin{align}
    \label{eq:p_r-u_rho1}
 \partial_r u_\rho = \frac{(\cos f(r) + 2\e^\gamma)^2}{u_\rho^2} \partial_r \big ( \cos f(r)\big)  + x_3 \frac{
 \displaystyle \partial_r \left[ \frac{r}{\partial_r
  (-\cos f(r))}\right]}{u_\rho^2}.
\end{align}

\begin{proof}[Proof of Lemma \ref{le:partI}]
 
 \phantom{,}
 
 \underline{Claim 1:} 
  \begin{align*}
    \lim_{\e\to0}
    \int_{x_3=0}^1 \int_{r=0}^{\varepsilon} r|\partial_r u_\rho|^2 \,\dd r \dd x_3
    =
     \lim_{\e\to0}
    \int_{x_3=0}^1 \int_{r=0}^{\varepsilon} r(\sin^2 f(r))(f'(r))^2\dd r \dd x_3 .
  \end{align*}
  In order to prove this claim, we begin by applying the relation $a^2-b^2 = 2b(a-b) + (a-b)^2$:
  \begin{multline}
   \int_{x_3=0}^1 \int_{r=0}^{\varepsilon} r|\partial_r u_\rho|^2
   - r(\sin^2 f(r))(f'(r))^2\,\dd r \dd x_3
   \\
   = 2\int_{x_3=0}^1 \int_{r=0}^{\varepsilon} r\partial_r\big (\cos f(r)\big )  \Big ( 
   \partial_r u_\rho -\partial_r\big ( \cos f(r) \big ) \Big ) \dd r \dd x_3
   \\
   + \int_{x_3=0}^1 \int_{r=0}^{\varepsilon} r\Big ( \partial_r u_\rho -\partial_r\big ( \cos f(r)\big ) \Big )^2 \dd r \dd x_3.
\label{eq:int_difference1}
  \end{multline}
  
  From \eqref{eq:p_r-u_rho1} it follows that
    \begin{equation}
   \partial_ru_\rho -\partial_r\big ( \cos f(r)\big )
   = \partial_r \big ( - \cos f(r)\big ) \left[ 1- \frac{(\cos f(r) + 2\e^\gamma)^2}{u_\rho^2} \right]
    + x_3 u_\rho^{-2} \partial_r \left( \frac{r}{\partial_r (-\cos f(r))}\right).
\label{eq:error-p_r-u_rho}
  \end{equation}

  The first term can be estimated via the relation
  $$
    a^2-b^2 = (a+b)(a-b) = (a+b)\frac{a^3-b^3}{a^2+ab+b^2},
  $$
yielding
  \begin{align}
\label{eq:u_rho2-cos2}
   u_\rho^2 - (\cos f(r) + 2\e^\gamma)^2 
   =\Big  (u_\rho + (\cos f(r) + 2\e^\gamma)\Big )
   \frac{ x_3 \cdot \frac{3r}{\partial_r (-\cos f(r))}}{
   u_\rho^2 + u_\rho (\cos f(r)+2\e^\gamma)+ (\cos f(r)+2\e^\gamma)^2}.
  \end{align}
  A first conclusion is that the expression on the left-hand side of \eqref{eq:u_rho2-cos2}
  is positive since $0\leq f(r)<\frac{\pi}{2}$ and $-\cos f(r)$ is increasing.
  Second, since \( \p_r(-\cos f(r)) >0\), it holds that
  \begin{align}
        \label{eq:u_rho-larger-eps_gamma}
    \cos f(r) + 2\e^\gamma \leq u_\rho.
  \end{align}
  Therefore, bounding $u_\rho^2 + u_\rho (\cos f(r)+2\e^\gamma) +(\cos f(r)+2\e^\gamma)^2$  from below by $u_\rho^2$, 
  it can be seen that
  \begin{multline}    
  |\partial_r(-\cos f(r))|
   \left |1- \frac{(\cos f(r) + 2\e^\gamma)^2}{u_\rho^2}\right |
   \\ \leq 
   \partial_r(-\cos f(r))\cdot u_\rho^{-2}\cdot (2u_\rho)\cdot u_\rho^{-2}\cdot \frac{3r}{\partial_r (-\cos f(r))}
   \leq 6ru_\rho^{-3}.
\label{eq:p_r-u_rho-cos-2}
  \end{multline}

Putting together \eqref{eq:error-p_r-u_rho}, 
\eqref{eq:p_r-u_rho-cos-2},
and Part \ref{it:p_r-complicated}) of
Lemma \ref{le:f_and_derivatives},
the following bound is obtained
(for small enough $\e$):
\begin{align*}
 |  \partial_ru_\rho -\partial_r\big ( \cos f(r)\big )|
 \leq 6ru_\rho^{-3} + 64\sqrt{2} u_\rho^{-2}
    = u_\rho^{-2} (6 r u_\rho^{-1} + 64\sqrt{2}).
\end{align*}
Thanks to \eqref{eq:u_rho-larger-eps_gamma} and by observing that
\begin{align}\label{eq:use_ofadding_eps_gamma}
 u_\rho \geq 2\e^\gamma \geq 2\e > 2r
\end{align}
we find \(
  |  \partial_ru_\rho -\partial_r\big ( \cos f(r)\big )|
 \leq 25 \e^{-2\gamma}.
\) Plugging that into 
\eqref{eq:int_difference1},
and using Part \ref{it:int-p_r-cos_f}) of Lemma \ref{le:f_and_derivatives},
the conclusion is that
\begin{equation*}
 \left | \int_{x_3=0}^1 \int_{r=0}^{\varepsilon} r|\partial_r u^{\e}_\rho|^2
   - r(\sin^2 f(r))(f'(r))^2\,\dd r \dd x_3
   \right |
\\
   \leq 
   25 \e^{-2\gamma} \cdot 12 \e^2 |\ln \e|
    + 25^2 \e^{-4\gamma} \int_0^\e r\dd r,
\end{equation*}
which vanishes as $\e\to0$ since $0<\gamma \leq \frac{1}{3} <\frac{1}{2}$.

 \underline{Claim 2:} 
  \begin{align*}
    \lim_{\e\to0}
    \int_{x_3=0}^1 \int_{r=0}^{\varepsilon} r|u_\rho^{\e}\partial_r u^{\e}_\varphi|^2 \,\dd r \dd x_3
    =
     \lim_{\e\to0}
    \int_{x_3=0}^1 \int_{r=0}^{\varepsilon} r(\cos^2 f(r))(f'(r))^2\dd r \dd x_3 .
  \end{align*}
  In order to prove this claim,
   we first observe that 
  \begin{align*}
   &
   \int_{x_3=0}^1 \int_{r=0}^{\varepsilon} r|u_\rho\partial_r u_\varphi|^2 
   - r(\cos^2 f(r))(f'(r))^2\dd r \dd x_3
   \\
   &\qquad \qquad \qquad 
   = \int_{x_3=0}^1 \int_{r=0}^{\varepsilon} rf'(r)^2(|u_\rho|^2 - \cos^2 f(r))\dd r\dd x_3
   \\
   &\qquad \qquad \qquad 
    = \int_{x_3=0}^1 \int_{r=0}^{\varepsilon} rf'(r)^2(|u_\rho|^2 - (\cos f(r)+2\e^\gamma)^2)\dd r\dd x_3
\\
   &\qquad \qquad \qquad 
    \qquad \qquad \qquad 
       + 4\e^\gamma  \int_{r=0}^{\varepsilon} rf'(r)^2 \cos f(r)\dd r 
    +4 \e^{2\gamma}
     \int_{r=0}^{\varepsilon} rf'(r)^2 \dd r.
  \end{align*}
    By Lemma \ref{le:energy_stereo},
    and considering that $0\leq \cos f(r)\leq 1$, 
    the last two terms
    are bounded by $4\e^\gamma
    +4\e^{2\gamma}$ and
    vanish as $\e$ goes to zero.
    As for the first term,
    by    
    \eqref{eq:u_rho2-cos2},
    \eqref{eq:u_rho-larger-eps_gamma},
    and Part \ref{it:p_r-cos_f})
    of Lemma \ref{le:f_and_derivatives},
    \begin{align*}
     u_\rho^2 -(\cos f(r) + 2\e^\gamma)^2 
     &\leq \frac{2u_\rho \cdot x_3\cdot 
     \displaystyle \frac{3r}{\partial_r\big(\cos f(r)\big)}}{u_\rho^2+0+0}
     \leq 6 u_\rho^{-1}\cdot 2\e^{-2}(\e^4+r^2)^{3/2}
\\
     &\leq 
     6\e^{-\gamma}\cdot \e^{-2} \cdot (2\e^2)^{3/2} = 12\sqrt{2}\e^{1-\gamma}.
    \end{align*}
    The claim follows by combining the above with Lemma \ref{le:energy_stereo}.
\medskip

    \underline{Conclusion:}
    By Claims 1 and 2, together with Lemma \ref{le:energy_stereo},
    \begin{multline*}
        \lim_{\e\to0}
    \int_{x_3=0}^1 \int_{r=0}^{\varepsilon} r|\partial_r u^{\e}_\rho|^2 + r|u^{\e}_\rho\partial_r u^{\e}_\varphi|^2\dd r \dd x_3
    \\
        =    
        \lim_{\e\to0}
    \int_{x_3=0}^1 \int_{r=0}^{\varepsilon} 
    r\sin^2 f(r)\,f'(r)^2 + 
    r\cos^2 f(r)\,f'(r)^2 \dd r = \frac{1}{2}.
    \end{multline*}

\end{proof}

\begin{lemma}
    \label{le:partII}
  Let $u^{\e}_\rho$ and $u^{\e}_\varphi$ be defined as in \eqref{eq:def-u_rho}. Then 
  \begin{align*}
    \lim_{\e\to0}
    \int_{x_3=0}^1 \int_{r=0}^{\varepsilon} \frac{1}{r}|u^{\e}_\rho \sin u^{\e}_\varphi|^2 \dd r \dd x_3
        =\frac{1}{2}.
  \end{align*}
\end{lemma}

\begin{proof}
 Writing $u_\rho$ as $\cos f(r) + (u_\rho-\cos f(r))$, the integral expression becomes
 \begin{multline*}
    \int_{x_3=0}^1 \int_{r=0}^{\varepsilon} \frac{1}{r}|u_\rho \sin u_\varphi|^2 
    \dd r \dd x_3
    =
    \int_{r=0}^{\varepsilon}
    \frac{1}{r}\cos^2 f(r) \sin^2 f(r)
    \dd r 
\\
    +
    \int_{x_3=0}^1 \int_{r=0}^{\varepsilon}
    \frac{1}{r}
    \Big [2\cos f(r)\,(u_\rho-\cos f(r))
    + 
    (u_\rho-\cos f(r))^2
    \Big ]\sin ^2 f(r)
    \dd r \dd x_3 .
 \end{multline*}
 By \eqref{eq:sin_A} the first integral is given by
 \begin{align*}
  \int_{r=0}^{\varepsilon}
    \frac{1}{r}\cos^2 f(r) \sin^2 f(r)
    \dd r 
    &= \int_0^\e \frac{\e^4 r\dd r }{(\e^4+r^2)^2}
    = \frac{\e^4}{2(\e^4+r^2)}\Big |_{r=\e}^0
    \overset{\e\to0}{\longrightarrow}\frac{1}{2}.
 \end{align*}
    Regarding the second integral,
     we use the relation $(a-b)(a^2+ab+b^2)=a^3-b^3$:
  \[
  u_\rho-\cos f(r) = 2\e^\gamma + u_\rho - (\cos f(r) + 2\e^\gamma) 
  =2\e^\gamma + \frac{x_3\cdot \frac{3r}{\partial_r(-\cos f(r))}}{u_\rho^2 + u_\rho (\cos f(r) + 2\e^\gamma) + (\cos f(r) + 2\e^\gamma)^2}.
\]
 Part \ref{it:p_r-cos_f})
 of Lemma \ref{le:f_and_derivatives}
 and \eqref{eq:u_rho-larger-eps_gamma}
 yield
 \begin{align*}
  u_\rho-\cos f(r)
    \leq 2\e^\gamma + (2\e^\gamma)^{-2}\e^{-2}(\e^4+r^2)^{3/2}
    = 2\e^\gamma + \frac{\sqrt{2}}{2}\e^{1-2\gamma}.
 \end{align*}
 Hence, by using that \( \gamma \leq 1/3 <1/2\), we arrive at
 \begin{multline*}
  \left| \int_{x_3=0}^1 \int_{r=0}^{\varepsilon}
    \frac{1}{r}
    \Big [2\cos f(r)\,(u_\rho-\cos f(r))
    + 
    (u_\rho-\cos f(r))^2
    \Big ]\sin ^2 f(r)
    \dd r \dd x_3
        \right |
\\
    \leq 
    \Big [ 2\cdot \big (
    2\e^\gamma + \frac{\sqrt{2}}{2}\e^{1-2\gamma}
    \big ) 
    + 
    \big (
    2\e^\gamma + \frac{\sqrt{2}}{2}\e^{1-2\gamma}
    \big )^2
    \Big ]
    \int_{r=0}^{\varepsilon}
    \frac{1}{r}\sin ^2 f(r)
    \dd r
\\
    = \big [ O(\e^\gamma) + O(\e^{1-2\gamma})\big ] 
    \underbrace{\int_{r=0}^{\e}
    \frac{2r}{\e^4+r^2}\dd r 
    }_{=\ln (1+\e^{-2})}
    \overset{\e\to0}{\longrightarrow} 0.
 \end{multline*}
\end{proof}

\begin{lemma}
    \label{le:partIII}
  Let $u^{(\e)}_\rho$ and $u^{(\e)}_\varphi$ be defined as in \eqref{eq:def-u_rho}. Then 
  \begin{align*}
    \lim_{\e\to0}
    \int_{x_3=0}^1 \int_{r=0}^{\varepsilon} r|\partial_{x_3} u^{\e}_\rho|^2 \dd r \dd x_3
        =0.
  \end{align*}
\end{lemma}

\begin{proof}
We start by observing that \(
  \partial_{x_3}u_\rho = 
  \frac{1}{3} u_\rho^{-2} \partial_{x_3} \big ( u_\rho^3\big )
  =u_\rho^{-2} \frac{r}{\partial_r\big(-\cos f(r)\big )}.
\)
 By \eqref{eq:u_rho-larger-eps_gamma} and Part \ref{it:p_r-cos_f}) of Lemma \ref{le:f_and_derivatives},
 \begin{align*}
   \int_{x_3=0}^1 \int_{r=0}^{\varepsilon} r|\partial_{x_3} u^{\e}_\rho|^2 \dd r \dd x_3
   \leq \int_{r=0}^{\varepsilon} r |(2\e^\gamma)^{-2}\cdot 2\e^{-2}(
   \underbrace{\e^4 +r^2}_{\leq 2\e^2})^{3/2}|^2 \dd r \leq
   \e^{4(1-\gamma)}
    \overset{\e\to 0}{\longrightarrow}0.
 \end{align*}
\end{proof}

In the rest of this section we will prove the following

\begin{proposition}\label{prop:convergence_energy}
Let \(\vec u_\e\) be the recovery sequence defined in Section \ref{sec:recovery_sequence}. Then
\begin{equation}\label{eq:energy_vanishes_a'_e'}
\lim_{\e \rightarrow 0} \int_{a_\e'\cup e_\e'} \left[ |D \vec u_\e|^2+H(\det D \vec u_\e) \right] \dd \vec x =0
\end{equation}
and
\begin{equation}\label{eq:convergence_energy_a_b_d_f}
\lim_{\e \rightarrow 0} \int_{a_\e\cup b_\e \cup d_\e \cup f} \left[ |D \vec u_\e|^2+H(\det D \vec u_\e) \right] \dd \vec x =\int_{a \cup b \cup d \cup f }\left[ |D \vec u|^2+H(\det D \vec u) \right] \dd \vec x.
\end{equation}
\end{proposition}

We remark that all the regions involved in the previous proposition are disjoint. Thus we will work separately in each of these regions.

\subsubsection{Extra energy in $a_\e'$ is negligible}

\begin{proof}[Proof of \eqref{eq:energy_vanishes_a'_e'} in \(a_\e'\)]
	We first observe that, since in this region \( \det D\vec u_\e=1\) we have that \( \int_{a_\e'} H(\det D\vec u_\e) \dd \vec x \rightarrow 0\) as \(\e\) tends to zero. The proof is then  obtained by dealing with the Dirichlet energy and by combining
	inequality \eqref{eq:DirichletAPrime} below,
	the bounds for the partial derivatives in Lemma \ref{le:partialUAPrime}
	and the integral estimates of 
	Lemma \ref{le:gradVarphiAPrime}.
\end{proof}

\underline{Inverse of the parametrization of the reference configuration}

Using Lemma \ref{le:positive_h} and, e.g., Ball's global invertibility theorem 
\cite{Ball81}
(considering that $(s, \cdot , \varphi)\mapsto \vec x(s, \cdot, \varphi)$,
seen as a map to $\R^2$, is one-to-one on the boundary of $[0,1]\times [0, \frac{\pi}{2}-\delta]$ for every small $\delta$)
we obtain that the parametrization of the reference domain (excluding $\varphi=\frac{\pi}{2}$, which collapses to the circle $\{r=\e,\, x_3=0\}$)
is a diffeomorphism and that $s$, $\theta$, and $\varphi$ can be obtained as functions of $\vec x$ in the interior of region $a_\e'$. 
Inverting the coefficient matrix for $\partial_s\vec x$, $\partial_\theta \vec x$, $\partial_\varphi \vec x$
in the basis $(\vec e_r, \vec e_\theta, \vec e_3)$ we find that
\begin{align}
	\nabla s &= \frac{ s\sin \varphi\,\vec e_r - \big ( (1-s)\e^{-1} g'(\varphi) + s\cos \varphi\big ) \vec e_3}{
		(1-s) g'(\varphi)\cos\varphi
	+ s (\e-g(\varphi)\sin\varphi)}, \nonumber
	\\
	\nabla \theta &= \frac{1}{(1-s) g(\varphi)
	+ s \e\sin\varphi} \vec e_\theta, \nonumber
	\\
			\label{eq:nablaVarphi}
	\nabla \varphi &= \frac{\cos \varphi \,\vec e_r + (\sin \varphi - \e^{-1}g(\varphi))\vec e_3}{
	(1-s) g'(\varphi)\cos\varphi
	+ s (\e-g(\varphi)\sin\varphi)}.
\end{align}

\underline{The Dirichlet energy}:
Based on the representation
\begin{align*}
	D\vec u_\e = \partial_s\vec u_\e \otimes \nabla s + \partial_\theta \vec u_\e \otimes \nabla \theta 
		+ \partial_\varphi \vec u_\e \otimes \nabla \varphi,
		\qquad |D\vec u_\e|^2 = \tr D\vec u_\e^T D\vec u_\e,
\end{align*}
it follows that 
\begin{align*}
	|D\vec u_\e|^2 &= |\nabla s|^2 |\partial_s\vec u_\e|^2 + 2(\nabla s\cdot \nabla \varphi)(\partial_s\vec u_\e \cdot \partial_\varphi \vec u_\e) + |\nabla \theta|^2 |\partial_\theta \vec u_\e|^2 + |\nabla \varphi|^2 |\partial_\varphi \vec u_\e|^2 
	\\
	& \! \! \! = |\nabla s|^2|\partial_s u_\rho^\e|^2
	+ 2(\nabla s\cdot \nabla \varphi)
	\partial_s u_\rho^\e \partial_\varphi u_\rho^\e
	+ |\nabla \theta|^2 |u_\rho^\e|^2 \sin^2\varphi 
	+ |\nabla \varphi|^2 (|u_\rho^\e|^2
	+ |\partial_\varphi u_\rho^\e|^2)
	\\ & \! \! \! \leq 
	|\nabla s|^2|\partial_s u_\rho^\e|^2
	+ 2(|\nabla s||\partial_s u_\rho^\e|)(|\nabla \varphi|
	 |\partial_\varphi u_\rho^\e|)
	+ |\nabla \theta|^2 |u_\rho^\e|^2 \sin^2\varphi 
	+ |\nabla \varphi|^2 (|u_\rho^\e|^2
	+ |\partial_\varphi u_\rho^\e|^2).
	\end{align*}
	Cauchy's inequality then yields that
	\begin{align}
		\label{eq:DirichletAPrime}
	|D\vec u_\e|^2
	&\leq 
	2|\nabla s|^2|\partial_s u_\rho^\e|^2
	+ |\nabla \theta|^2 |u_\rho^\e|^2 \sin^2\varphi 
	+ |\nabla \varphi|^2 (|u_\rho^\e|^2
	+ 2|\partial_\varphi u_\rho^\e|^2).
	\end{align}
Note also that
$$
	\int_{a_\e'} |D\vec u_\e|^2 \dd\vec x 
	= 2\pi \int_{s=0}^1 \int_{\varphi=0}^{\frac{\pi}{2}} |D\vec u_\e|^2 h_\e(s,\varphi) \sin\varphi \dd\varphi \dd s.
$$

\underline{Estimates for $g_\e$ and $h_\e$}

In order to estimate the partial derivatives of $ u_\rho^\e$, it is important to control first the derivatives of the functions $g_\e$ and $h_\e$ that appear in its definition. This is the object of Lemma \ref{le:positive_h} in the appendix.

\underline{Estimates for $u_\rho^\e$}

\begin{lemma}
		\label{le:partialUAPrime}
	For all $\varphi \in [0, \frac{\pi}{2}]$, all $s\in [0,1]$, 
	and all positive $\e$ such that 
	$\e^{2-2\gamma}
	< \frac{7}{9\pi\sqrt{2}}$,
	\begin{align*}
		\frac{1}{4}(\cos \varphi + 2\e^\gamma)\leq  u_\rho^\e(s,\varphi) \leq \cos \varphi + 2\e^\gamma\leq 2 ,
		\quad
		|\partial_s u_\rho^\e|\leq C \e^{2-2\gamma}\cos\varphi,
		\quad 
		|\partial_\varphi u_\rho^\e|=O(1).
	\end{align*}
\end{lemma}

\begin{proof}

Since, by Lemma \ref{le:hE2Cos},
$$
	3\left | \int_{\sigma=0}^1 h(\sigma,\varphi)\dd\sigma \right | \leq \frac{9\pi\sqrt{2}}{2} \e^2 \cos\varphi,
$$
then
\begin{equation*}
	(\cos\varphi +2\e^\gamma )^3\geq u_\rho^\e(s,\varphi)^3
	\geq
	(\cos\varphi +2\e^\gamma )^3
	- \frac{9\pi\sqrt{2}}{2}\e^2\cos\varphi.
\end{equation*}

\underline{If $2\e^\gamma\leq \cos\varphi$}, 
then
$$
    \frac{7}{8}\cos^2\varphi
    \geq \frac{7}{2}\e^{2\gamma}
    \geq \frac{9\pi\sqrt{2}}{2}\e^2
$$
and
\begin{equation*}
	u_\rho^\e(s,\varphi)^3 \geq \cos^3\varphi - 
	\frac{7}{8}\cos^2 \varphi \cdot \cos\varphi  \geq \left( \frac{1}{2}\cos\varphi\right)^3.
\end{equation*}
Consequently,
\begin{equation*}
	u_\rho^\e(s,\varphi)\geq \frac{1}{2}\cos\varphi = \frac{1}{2} (\cos\varphi + 2\e^\gamma) \frac{1}{1 + \frac{2\e^\gamma}{\cos\varphi}}
	 \geq \frac{1}{4}(\cos\varphi + 2\e^\gamma).
\end{equation*}

\underline{If $2\e^\gamma> \cos\varphi$,}
\begin{align*}
	u_\rho^\e(s,\varphi)^3 \geq 8\e^{3\gamma} - \frac{9\pi\sqrt{2}}{2}\e^2 \cdot 2\e^\gamma
	=
	 8\e^{3/2} \left( 1- \frac{9\pi\sqrt{2}}{8}\e^{2-2\gamma} \right)
	\geq  
	\left((2\sqrt{\e})\cdot \frac{1}{2}\right)^3.
\end{align*}
Therefore,
\begin{equation*}
	u_\rho^\e(s,\varphi) \geq  \frac{1}{2} (\cos\varphi + 2\e^\gamma) \frac{1}{\frac{\cos\varphi}{2\e^\gamma}+1}
	\geq \frac{1}{4}(\cos\varphi + 2\e^\gamma).
\end{equation*}

Regarding $\partial_\varphi u_\rho^\e$,
\begin{align*}
		3\partial_\varphi u_\rho^\e =  (u_\rho^\e)^{-2}\partial_\varphi \Big ( (u_\rho^\e)^3\Big ) =-3
		\underbrace{\left(\frac{\cos\varphi + 2\e^\gamma}{u_\rho^\e}\right)^2}_{\leq 4^2}\sin\varphi -3(u_\rho^\e)^{-2} \underbrace{\int_{\sigma=0}^s \partial_\varphi h(\sigma, \varphi) \dd\sigma}_{=O(\e)}.
	\end{align*}
	Since 
	$$
		(u_\rho^\e)^2 \geq \frac{1}{4^2}(\cos \varphi + 2\e^\gamma)^2 \geq \frac{1}{4}\e^{2\gamma},
	$$
	and $\gamma \leq \frac{1}{3} < \frac{1}{2}$, 
	the derivative $\partial_\varphi u_\rho^\e$ is bounded uniformly with respect to $\e$. 
	
	The estimate for the derivative with respect to $s$ is now straightforward:
	\begin{align*}
		3|\partial_s u_\rho^\e|
		= \left | (u_\rho^\e)^{-2}\partial_s \Big ( (u_\rho^\e)^3\Big ) \right |
		\leq 3\cdot \frac{ \frac{3\pi\sqrt{2}}{2} \e^2\cos \varphi}{\frac{1}{4}\e^{2\gamma}} \leq C\e^{2-2\gamma}\cos\varphi.
	\end{align*}
\end{proof}

\underline{Integral estimates for $\nabla s$, $\nabla \theta$, and $\nabla \varphi$}

\begin{lemma}
	\label{le:gradVarphiAPrime}
\begin{align*}
 & \int_{a_\e'}
	 |\nabla \varphi|^2\,
		\dd\vec x 
		= O (\e^2|\ln \e|),
	\qquad 	\int_{a_\e'}(\e \cos\varphi)^2|\nabla s|^2\,\dd\vec x = O(\e|\ln \e|),\\
	& \text{and}\quad 
	\int_{a_\e'} |\nabla \theta|^2 \sin^2\! \varphi\,\dd\vec x 
		= O (\e|\ln \e|^2).
	\end{align*}
 \end{lemma}

\begin{proof}
	By \eqref{eq:1menosSinCos} and Lemma \ref{le:lowerEMinusG},
	the modulus of the numerator in \eqref{eq:nablaVarphi} can be estimated by 
	\begin{align*}
		\left| \cos\varphi \right| + \left| \sin \varphi -\e^{-1}g(\varphi) \right|
		&= \left| \cos\varphi \right| + \left| \e^{-1} (\e-g(\varphi)) - (1-\sin\varphi) \right|
		\leq \e^{-1} (\e-g(\varphi)) +2\cos \varphi
		\\
		&\leq \e^{-1} (\e-g(\varphi)) + 
		6\frac{\e-g(\varphi)}{\max\{\e, \frac{g(\varphi)}{\e}\}} 
		\leq 7\e^{-1}(\e-g(\varphi)).
	\end{align*}
	Therefore, by Lemma  \ref{le:hE2Cos}, for all $\varphi \in [0,\frac{\pi}{2}]$ and $s\in [0,1]$,
		\begin{align*}
		|\nabla \varphi|^2 h(s,\varphi)
		&\leq  \frac{49\e^{-2}(\e-g(\varphi))^2 }{ |(1-s)g'(\varphi)\cos \varphi + s(\e - g(\varphi)\sin\varphi) |^2}
		\cdot 
		\\ & \quad \quad \cdot 
		\e \Big ( \underbrace{(1-s) \frac{g(\varphi)}{\sin \varphi}+ s\e}_{\leq \sqrt{2} \e} \Big ) \Big ( (1-s)g'(\varphi)\cos \varphi + s(\e - g(\varphi)\sin\varphi) \Big )
		\\ & \leq
		\frac{C(\e-g(\varphi))^2}{(1-s)g'(\varphi)\cos \varphi + s(\e - g(\varphi)\sin\varphi)}.
	\end{align*}
	Hence, by Lemma \ref{le:logNablaVarphi},
	\begin{align*}
		\int_{\varphi=0}^{\frac{\pi}{2}}\int_{s=0}^1 |\nabla \varphi|^2 h(s,\varphi) \sin\varphi\,\dd s \dd\varphi
		& \leq 
		\int_{\varphi=0}^{\frac{\pi}{2}}
		 C\underbrace{(\e-g(\varphi))^2}_{\leq \e^2}|\ln \e| \dd\varphi =O(\e^2|\ln \e|).
	\end{align*}

	As for $\nabla \theta$, 
	observe first that 
	\begin{align*}
		|\nabla \theta|^2 \sin^2 \varphi \, h(s,\varphi) \sin \varphi &= \frac{\e \Big ( (1-s)g'(\varphi)\cos \varphi + s(\e - g(\varphi)\sin\varphi) \Big )}{(1-s) \frac{g(\varphi)}{\sin\varphi}+ s\e } \cdot\sin\varphi .
	\end{align*}
	By Lemma \ref{le:lowerEMinusG},
	\begin{equation*}
		 |\nabla \theta|^2 \sin^2 \varphi \, h(s,\varphi) \sin \varphi \leq \frac{ C\e \min \{\frac{1}{\e}, \frac{\e}{g(\varphi)}\}g'(\varphi)\cos\varphi}{(1-s) \frac{g(\varphi)}{\sin\varphi}+ s\e }.
	\end{equation*}
	
	\underline{Case when $g(\varphi)\leq \e^2$:}
	calling $r:=g(\varphi)$, it is easy to see that
	$$
		g'(\varphi) = \frac{1}{ \frac{\e^2}{r^2 +\e^4} + \frac{\alpha_\e}{\e} } \leq 2\e^2
		\quad
		\text{and}\quad 
		g(\varphi)= \int_{t=0}^\varphi g'(t)\dd t \geq \int_{t=0}^\varphi \frac{\e^2}{2} \dd t \geq \frac{\e^2}{2} \varphi \geq \frac{\e^2}{2}\sin\varphi.
	$$
	Hence
	\begin{align*}
		\int_{s=0}^1 |\nabla \theta|^2 \sin^2 \varphi \, h(s,\varphi) \sin \varphi&\leq \int_{s=0}^1\frac{ C\e^2}{(1-s) \e^2+ s\e }
		\leq  C\frac{\e^2}{\e-\e^2}|\ln \e|
		= O(\e|\ln \e|).
	\end{align*}
	
	\underline{Case when $g(\varphi)\geq \e^2$:}
	here $g'(\varphi)\leq 2\frac{r^2}{\e^2}$. Also,  applying Lemma \ref{le:logEpsilon} with
	$$
		a= \frac{g(\varphi)}{\sin \varphi},\quad
		b= \e,\quad \lambda= \e ,
		$$	
	it can be seen that
	\begin{align*}
		\int_{s=0}^1
		|\nabla \theta|^2 \sin^2 \varphi \, h(s,\varphi) \sin \varphi \,\dd s 
		&\leq C\e \frac{\e}{r}\cdot 2 \frac{r^2}{\e^2}\cos\varphi \cdot \e^{-1}\frac{1}{1-\e} |\ln \e|
		= C\frac{r}{\e}|\ln \e|.
	\end{align*}
	Integrating now over $\varphi$, changing variables to $t= \frac{g(\varphi)}{\e}$, $\varphi = f \Big ( \e t\Big )$:
	\begin{align*}
		&\int_{\varphi=f(\e^2)}^{\frac{\pi}{2}}
		\int_{s=0}^1
		|\nabla \theta|^2 \sin^2 \varphi \, h(s,\varphi)\sin \varphi \,\dd s \dd\varphi 
		\leq C|\ln \e| \int_{t=\e}^1 t \underbrace{f'\Big ( \e t\Big )}_{\leq 2t^{-2}} \e \dd t
		\leq C\e |\ln \e|^2.
		&
	\end{align*}

	Finally, for the result for $\nabla s$ note that
		\begin{align*}
		(\e\cos\varphi)^2 |\nabla s|^2 h(s,\varphi)\sin\varphi
		 & \leq \frac{ C\e^2 \Big (\sin^2 \varphi + \cos^2\varphi + \e^{-2}|g'(\varphi)|^2\Big )\cdot \e  \Big ( (1-s) \frac{g(\varphi)}{\sin \varphi}+ s\e \Big )}{(1-s)g'(\varphi)\cos \varphi + s(\e - g(\varphi)\sin\varphi)}\cos^2 \varphi
		\\
		& \leq \frac{C\e^2\cos^2 \varphi}{(1-s)g'(\varphi)\cos \varphi + s(\e - g(\varphi)\sin\varphi)}.
	\end{align*}

	By Lemma \ref{le:logNablaVarphi},
	\begin{align*}
		& \int_{\varphi=0}^{\frac{\pi}{2}}\int_{s=0}^1 (\e\cos\varphi)^2|\nabla s|^2 h(s,\varphi) \sin\varphi\,\dd s \dd\varphi
		 \leq C\e^2 
		\int_{\varphi=0}^{\frac{\pi}{2}}
		 \frac{\cos^2\varphi}{\e-g(\varphi)}|\ln \e| \dd\varphi 
		 \\ &\leq 
		 C\e^2|\ln \e| \int_{\varphi=0}^{\frac{\pi}{2}} \frac{\e - g(\varphi)}{\max\{\e, \frac{g(\varphi)}{\e}\}^2 } \dd\varphi 
		 \leq C\e^3|\ln \e| \left( \int_0^{f(\e^2)}\e^{-2}\dd\varphi + 
		 \int_{f(\e^2)}^{\frac{\pi}{2}} \frac{\e^2}{g(\varphi)^2} \dd\varphi \right).
	\end{align*}
	For the last integral, change variables to $t=\frac{g(\varphi)}{\e}$, $\varphi= f(\e t)$:
	\begin{align*}
	\int_{f(\e^2)}^{\frac{\pi}{2}} \frac{\e^2}{g(\varphi)^2} \dd\varphi &= 
		\int_\e^1 t^{-2} f'(\e t) \e \dd t 
		\leq 2\e \int_\e^1 t^{-4}\dd t 
		\leq \frac{2}{3}\e^{-2}
	\end{align*}
	and the conclusion follows.
\end{proof}

\subsubsection{Extra energy in $a_\e$ is negligible}\label{subsubse:aenegligible}

In this section we prove the part of Proposition \ref{prop:convergence_energy} relative to the region \(a_\e\).

In order to prove this proposition we first deduce an integrability property of the function \(H\) implied by the assumption that \( E(\vec u)<+\infty\). Indeed, using \eqref{eq:det_spherical_spherical}, the following expression is obtained for the Jacobian of the limit map:
 \begin{align}
    \label{eq:detContiDeLellisRegionA}
\det D\vec u = \frac{u_\rho^2 \sin u_\varphi \cdot \matrizzDet{ \partial_\rho u_\rho & \partial_\varphi u_\rho  \\ \partial_\rho u_\varphi & \partial_\varphi u_\varphi } }{
\rho^2 \sin\varphi }
    = (1-\rho)^2\rho^{-2}\cos^3\varphi.
\end{align}

	\begin{lemma}
  \label{le:integrabilityHregionA}
	Let $H$ be as in the statement of Theorem \ref{th:upper_bound}. 
		Then,
			$$
				\int_{s=1}^\infty H(s) s^{-5/2} \dd s < \infty.
			$$
	\end{lemma}

\begin{proof}
First note that the convexity assumption of $H$ together with the growth \eqref{eq:growth_H} implies that there exists $\delta>0$ such that 
            \begin{align}
    \label{eq:HdecreasingIncreasing}
                H \text{ is decreasing in } (0,\delta)
                \text{ and increasing in } (\frac{1}{\delta}, +\infty).
            \end{align}

Let \(\vec u\) be the Conti--De Lellis map as in Theorem \ref{th:upper_bound}. By assumption \eqref{eq:finiteHu},
        \begin{align*}
    \infty &> \int_{a} H(\det D\vec u) \dd \vec x > 
    \int\limits_{\{\cos \varphi > \frac{1}{2}\ \wedge\ \rho < \frac{1}{2} \}} H(\det D\vec u)\dd\vec x
    \\
    & = 2\pi \int_{\rho=0}^{\frac{1}{2}}\int_{\cos\varphi=\frac{1}{2}}^{\cos\varphi=1}
    H \big ( (1-\rho)^2 \rho^{-2}\cos^3\varphi \big ) \rho^2 \dd \big ( \cos\varphi\big ) \dd\rho
    \\
    & = 2\pi \int_{\rho=0}^{\frac{1}{2}} \int_{t=\frac{1}{2}}^1 
    H \big ( (1-\rho)^2 \rho^{-2} t^3 \big ) \rho^2 \dd t \dd \rho.
    \end{align*}

    At this point we observe that \( (1-\rho)^2 \rho^{-2} t^3 \geq 2^{-5} \rho^{-2}\). 
   In  the above integral, we keep only those values of $\rho$ such that 
    $2^{-5}\rho^{-2} \geq \frac{1}{\delta}$,
    i.e., $\rho \leq \sqrt{\delta/32}$.
    Since, by \eqref{eq:HdecreasingIncreasing}, $H$ is increasing in $(\frac{1}{\delta}, +\infty)$,
    it follows that 
    \begin{align*}
    \infty &> 2\pi \int_{\rho=0}^{\sqrt{\delta/32}}
    \int_{t=\frac{1}{2}}^1 H(2^{-5}\rho^{-2}) \rho^2 \dd t \dd \rho
    =
    \pi \int_{\rho=0}^{\sqrt{\delta/32}} H(2^{-5}\rho^{-2})\rho^2 \dd \rho.
    \end{align*}

    Changing the integration variable to $s=2^{-5}\rho^{-2}$ yields
    \begin{align*}
    \infty > \int_{s=\frac{1}{\delta}}^\infty H(s)
    \frac{2^{-5/2}}{s} s^{-3/2} \dd s.
    \end{align*}
    This finishes the proof since in $[1, \frac{1}{\delta}]$ the function $H$ is continuous (hence bounded).
\end{proof}

\begin{lemma}
\label{le:auxDCregion_a}
 Let us suppose that $\e$ is sufficiently small so that $\e^{2(1-\gamma)} \leq \frac{\sqrt{2}}{\pi}. $
Then,  for all $0\leq \varphi < \frac{\pi}{2}$,
\begin{equation*}
    \left(
 (\cos \varphi + 2\e^\gamma)^3 
 - 3\int_{\sigma=0}^1 h(\sigma, \varphi)\dd\sigma
 \right)^{\frac{1}{3}}
 \geq \cos\varphi + \e^{\gamma}.
\end{equation*}
\end{lemma}

\begin{proof}
    By Lemma \ref{le:hE2Cos},
    \begin{multline*}
    (\cos \varphi + 2\e^\gamma)^3 
    - 3\int_{\sigma=0}^1 h(\sigma, \varphi)\dd\sigma
    \geq 
    (\cos \varphi + 2\e^\gamma)^3 
    - \frac{9\pi\sqrt{2}}{2}\e^2\cos\varphi
    \\
        \geq
        \big (\cos^3\varphi + 6\e^{\gamma}\cos^2\varphi+12\e^{2\gamma}\cos\varphi
              +8\e^{3\gamma}\big )
    - \frac{9\pi\sqrt{2}}{2}\e^2\cos\varphi
    \\
    = (\cos \varphi + \e^{\gamma})^3
        + \underbrace{3\e^{\gamma}\cos^2\varphi}_{\geq 0}
        + \frac{9\pi\sqrt{2}}{2}\e^{2\gamma}
        \underbrace{\big ( \frac{\sqrt{2}}{\pi} 
    - \e^{2(1-\gamma)}\big )}_{\geq 0}\cos\varphi  + \underbrace{7\e^{3\gamma}}_{\geq 0}.
    \end{multline*}

\end{proof}

\begin{lemma}
 Let \(\vec u_\e\) be the recovery sequence of Section \ref{sec:recovery_sequence} and \(\vec u\) the Conti--De Lellis map. Suppose that $\e$ is sufficiently small so that 
$
    \e^{2(1-\gamma)} \leq \frac{\sqrt{2}}{\pi}.
$ Then 
$$
\det D\vec u^\e (\vec x)
\geq \det D\vec u(\vec x) \quad 
\text{for all } \vec x \text{ in region } a_\e.
$$
\end{lemma}

\begin{proof} 
    The same calculation as in \eqref{eq:detContiDeLellisRegionA} gives
    \begin{align}
    \label{eq:1detDuEpsdetDuRegionA}
        \det D\vec u_\e = \rho^{-2}(u_\rho^\e)^2\partial_\rho u_\rho^\e.
            \end{align}
            By Lemma \ref{le:auxDCregion_a},
            \begin{align}
    \label{eq:2detDuEpsdetDuRegionA}
            \partial_\rho u_\rho^\e
            &= \underbrace{\frac{1}{1-\e}}_{\geq 1} \Bigg ( 
                   \Big (
 (\cos (\pi-\varphi) + 2\e^\gamma)^3 
 - 3\int_{\sigma=0}^1 h(\sigma, \varphi)\dd\sigma
 \Big )^{\frac{1}{3}}
                   -\e^\gamma \Bigg )
            \geq -\cos \varphi.
            \end{align}
            Regarding $(u_\rho^\e)^2$, using again Lemma \ref{le:auxDCregion_a},
    \begin{align}
\label{eq:3detDuEpsdetDuRegionA}
    u_\rho^\e 
    \geq 
    \underbrace{\frac{1-\rho}{1-\e}}_{\geq (1-\rho)} (\cos (\pi\varphi) + \e^\gamma) 
    + \underbrace{\e^\gamma \frac{\rho-\e}{1-\e}}_{\geq 0}
    \geq -(1-\rho)\cos\varphi. 
    \end{align}
    Combining \eqref{eq:1detDuEpsdetDuRegionA}, \eqref{eq:2detDuEpsdetDuRegionA}, and
    \eqref{eq:3detDuEpsdetDuRegionA} it can be seen that
    \begin{equation*}
     \det D\vec u_\e \geq -(1-\rho)^2 \rho^{-2}\cos^3 \varphi = \det D\vec u.
    \end{equation*}
\end{proof}

Now bounds from above are needed for $\det D\vec u_\e$.

    \begin{lemma}
    Suppose that $\e$ is small enough so that $\frac{(1+2\e^\gamma)^3}{1-\e} < 2.$
    Then, for all $\rho\in (\e,1)$, $\theta\in [0,2\pi]$, and $\varphi \in [\frac{\pi}{2},\pi]$,
        $$\det D\vec u_{\e} \big ( \vec x (\rho, \theta, \varphi)\big ) \leq 2\rho^{-2}.$$
    \end{lemma}

    \begin{proof}
    From \eqref{eq:def-u_eps-region_a} it can be seen (since $h\geq 0$) that
    $$
        |u_\rho^\e(\rho,\varphi)|\leq \frac{1-\rho}{1-\e} (1+2\e^\gamma) + \e^\gamma \frac{\rho-\e}{1-\e} \leq (1+2\e^\gamma).
    $$
    On the other hand,
    \begin{align*}
        \partial_\rho u_\rho^\e
            &= \frac{1}{1-\e} \Bigg ( 
                   \Big (
 (\underbrace{\cos (\pi-\varphi)}_{\leq 1} + 2\e^\gamma)^3 
 - 3\int_{\sigma=0}^1 h(\sigma, \varphi)\dd\sigma
 \Big )^{\frac{1}{3}}
                   -\e^\gamma \Bigg )
            \leq \frac{1+2\e^\gamma}{1-\e}.
         \end{align*}
    Plugging both into \eqref{eq:1detDuEpsdetDuRegionA} the desired upper bound for $\det D\vec u_\e$ is obtained. 
    \end{proof}

We are now ready to prove the part of Proposition \ref{prop:convergence_energy} relative to the region \(a_\e\).
    \begin{proof}[Proof of Proposition \ref{prop:convergence_energy} in \(a_\e\)]
We start with the Dirichlet part of the energy. Since \(\vec u\) is given in spherical coordinates in this region we use \eqref{eq:Dirichlet_energy_spherical_sperical}. By direct computation we can check that \( \chi_{a_\e}D\vec u_\e \rightarrow\chi_a D\vec u\) pointwise. Furthermore using that \( h(s,\varphi)=O(\e^2)\) from Lemma \ref{le:hE2Cos} we can see that \( |\p_\rho u_\rho^\e|, |u_\rho^\e \p_\varphi u_\varphi^\e|, |u_\rho^\e \sin u_\varphi^\e|\) are uniformly bounded in the region \(a_\e\). Now, using the same computation as in the last item of Lemma \ref{le:partialUAPrime} we also obtain that \( | \p_\varphi u_\rho^\e|\) is uniformly bounded. We can thus applied the dominated convergence theorem to obtain the convergence of the Dirichlet energy in this region.

    We then deal with the determinant part of in the energy. Combining the bounds for $\det D\vec u_\e$ from below and above one obtains
    \begin{align*}
    H\big (\det D\vec u_{\e}(\vec x)\big ) \leq C + H\big (\det D\vec u(\vec x)\big )
    + H(2\rho^{-2}),
    \end{align*}
    the constant $C$ being $C:=\max\{ H(J): \delta \leq J\leq 1/\delta\}$.
    In order to use the dominated convergence theorem it suffices, then, to show that $H(2\rho^{-2})$ is integrable.
We just note that
    \begin{align*}
    \int_a H(2\rho^{-2})\dd \vec x 
    &= 2\pi \int_{\rho=0}^1\int_{\varphi=0}^\pi H(2\rho^{-2}) \rho^2 \dd (-\cos \varphi)\dd\rho
    = 2\pi \int_{\rho=0}^1 H(2\rho^{-2}) \rho^2\dd\rho.
    \end{align*}
    Changing the integration variable to $t:=2\rho^{-2}$ and using Lemma \ref{le:integrabilityHregionA},
    \begin{align*}
    \int_a H(2\rho^{-2})\dd \vec x 
    &=2\sqrt{2}\pi \int_{t=2}^\infty H(t)t^{-5/2} \dd t < \infty.
    \end{align*}
\end{proof}

\subsubsection{Extra energy in $b$ is negligible}
Here we prove the part of Proposition \ref{prop:convergence_energy} concerning region~\(b\).

\begin{proof}[Proof of Proposition \ref{prop:convergence_energy} in \(b\)]

We start with the Dirichlet energy. Since in this region \(\vec u_\e= \e^\gamma \vec \psi \left(\e^\gamma \vec \phi_\e (\vec x) \right)\) we find that \( D \vec u_\e (\vec x)= D \vec \psi (\e^{-\gamma}\vec \phi_\e (\vec x))D\vec \phi_\e(\vec x))\). From this and the expressions of \( \vec \phi_\e\) and \( \vec \psi\) we can see that \( D \vec u_\e \rightarrow D \vec u\) almost everywhere in \( b\).  Since \( \vec \psi \) is bi-Lipschitz and \( |D \vec \phi_\e|\leq C\) for \(\e\) small enough we can apply the dominated convergence theorem to get the convergence of the Dirichlet energy.

For the determinant part, as in Subsection \ref{subsubse:aenegligible}, we start by deriving an integrability property of the function \(H\) in Theorem \ref{th:upper_bound}. We can compute that the Jacobian of the limit map is 
$$
	\det D\vec u (\rho, \varphi) = \frac{ (\rho-1)^2 \sin \Big (\frac{\varphi+\pi}{2} \Big )}{2\rho^2 \sin\varphi}
	= \frac{(\rho-1)^2}{4\rho^2 \sin \frac{\varphi}{2}}.
$$
We know from \eqref{eq:HdecreasingIncreasing} that $H$ is decreasing in some interval $(0,\delta)$. Since
$$
	\rho\geq 1,\ \frac{\pi}{2}\leq \varphi\leq \pi\quad\Rightarrow\quad 
	\det D\vec u \leq \frac{(\rho-1)^2}{2\sqrt{2}},
$$
it follows that 
$$
	\frac{(\rho-1)^2}{2\sqrt{2}}\leq\delta\ \Rightarrow\ H(\det D\vec u) \geq H\Big ( \frac{(\rho-1)^2}{2\sqrt{2}} \Big ).
$$
Hence
\begin{align*}
	\infty &> 
	\int_b H(\det D\vec u)\dd\vec x \geq 2\pi \int_{\rho=1}^3 \int_{\varphi=\frac{\pi}{2}}^{\pi} 
		H(\det D\vec u) \cdot 1^2 \cdot \sin\varphi\dd\varphi \dd\rho 
		\\ & \geq 
	2\pi \int_{\{ \frac{(\rho-1)^2}{2\sqrt{2}}\leq\delta \}}
	\int_{\varphi=\frac{\pi}{2}}^\pi H\Big ( \frac{(\rho-1)^2}{2\sqrt{2}} \Big )\sin\varphi\dd\varphi\dd\rho 
	= 2\pi \int_{t=0}^{\sqrt{\delta}} H(t^2)\cdot \sqrt{2\sqrt{2}} \dd t,
\end{align*}
which yields that the function $t\mapsto H(t^2)$ is $L^1\big ((0,1)\big )$. This is helpful in what follows. 

Now we estimate $H(\det D\vec u_\e)$. In the region $\rho\geq 1+\sqrt{2}\e^\gamma$ we have that 
\[
         \vec u_\e(\vec x) =\e^\gamma \id \Big ( \e^{-\gamma} \vecg\phi_\e (\vec x) \Big )
        = \e^\gamma \vec e_3 + (\rho -1 +\sqrt{2}\e^\gamma ) \Big ( \Big ( \sin \frac{\varphi+\pi}{2}\Big ) \vec e_r + \Big (  \cos \frac{\varphi+\pi}{2}\Big ) \,\vec e_\theta\Big ).
\]
Hence
$$	
	\det D\vec u_\e = \frac{(\rho-1+\sqrt{2}\e^\gamma)^2}{4\rho^2 \sin \frac{\varphi}{2}}.
$$
On the one hand, $\det D\vec u_\e(\vec x)\geq \det D\vec u(\vec x)$ for all $\vec x$ in that region. 
On the other hand,
$$
	1\leq \rho\leq 3,\ \frac{\pi}{2}\leq \varphi\leq \pi\quad\Rightarrow\quad 
	\det D\vec u_\e \leq \frac{(2+\sqrt{2}\e^\gamma)^2}{4\cdot 1^2 \cdot \frac{\sqrt{2}}{2}} \leq 2
$$
for small $\e$. Therefore,
$$
	H(\det D\vec u_\e(\vec x)) \leq \max\Big \{ H(\det D\vec u(\vec x)), \max_{\delta\leq J\leq 2} H(J)\Big \}
$$
and, by dominated convergence, 
$$
	\int_{b\cap \{\rho \geq 1+\sqrt{2}\e^\gamma\}} H (\det D\vec u_\e) \dd\vec x \overset{\e\to0}{\longrightarrow} \int_b H(\det D\vec u) \dd \vec x.
$$

In the region $1\leq \rho \leq 1+\sqrt{2}\e^\gamma$: 
$$
	\det D\vec u_\e(\vec x) = \Big ( \det (D\vecg\psi)\big (\e^{-\gamma}\vecg\phi_\e(\vec x)\big ) \Big ) \cdot \frac{(\rho-1+\sqrt{2}\e^\gamma)^2}{4\rho^2 \sin \frac{\varphi}{2}},
$$
so
$$
	(\min \det D\vecg\psi)\cdot \frac{(\rho -1+\sqrt{2}\e^{\gamma})^2}{4(1+\sqrt{2}\e^\gamma)^2\cdot 1}
	\leq \det D\vec u_\e(\vec x) \leq (\max \det D\vecg\psi)\frac{(\rho -1+\sqrt{2}\e^{\gamma})^2}{4\cdot 1^2 \cdot \frac{\sqrt{2}}{2}}.
$$
Since $\rho-1+\sqrt{2}\e^\gamma\leq 2\sqrt{2}\e^\gamma$, for $\e$ small enough, $\det D\vec u_\e(\vec x)$ lies in the region where $H$ is decreasing and 
$$
	H(\det D\vec u_\e(\vec x)) \leq H\big ( C^2 (\rho-1 + \sqrt{2}\e^\gamma)^2 \big ),\quad \text{with} \quad C^2= \frac{\min \det D\vecg\psi}{5}.
$$
Consequently, 
\begin{align*}
	\int_{\{ 1\leq \rho \leq 1+\sqrt{2}\e^\gamma \}} H(\det D\vec u_\e) \dd\vec x 
	&\leq
	2\pi 
	\int_{\{ 1\leq \rho \leq 1+\sqrt{2}\e^\gamma \}} H\big ( C^2 (\rho-1 + \sqrt{2}\e^\gamma)^2 \big )
	\int_{\varphi=\frac{\pi}{2}}^\pi 3^2 \sin\varphi \dd\varphi\dd\rho
	\\
	&=
	18\pi \int_{t=0}^{2\sqrt{2}C\e^\gamma} H(t^2)\dd t ,
\end{align*}
which vanishes as $\e\to 0$ because $t\mapsto H(t^2)$ is integrable in $(0,1)$. 
\end{proof}

	\subsubsection{Extra energy in $e_\e'$ is negligible}
	
	We prove here the part of Proposition \ref{prop:convergence_energy} concerning region \(e_\e'\) and in particular \eqref{eq:energy_vanishes_a'_e'}.
	
We start with the following.
	\begin{lemma}
  \label{le:estimatesUrhoRegionEprime}
		For all $\varphi\in [0,\frac{\pi}{2}]$, all $s\in[0,1]$, and all $\e$ sufficiently small,
		\begin{align}
	\label{eq:bndUregionEprime}
		\cos \varphi + 2\e^\gamma \leq u_\rho^\e
		\leq \Big (  ( \cos \varphi + 2\e^\gamma )^3 + 12\sqrt{2}\e 
			+ \frac{9\pi \sqrt{2}}{2}\e^2 \cos\varphi \Big )^\frac{1}{3}
			\leq 2,
		\end{align}
		\begin{align}
	\label{eq:bndUregionEprime2}
			|\partial_s u_\rho^\e|\leq C\e^{2-2\gamma}\cos\varphi
\quad \text{and}\quad |\partial_\varphi u_\rho^\e|\leq C \e^{-2\gamma}.
		\end{align}

	\end{lemma}

	\begin{proof}
		The bounds for $u_\rho^\e$ can be obtained using Lemmas \ref{le:f_and_derivatives}.\ref{it:p_r-cos_f})
		and \ref{le:hE2Cos}.
		The bound for $\partial_s u_\rho^\e$ is proved exactly as in Lemma \ref{le:hE2Cos}. Regarding $\partial_\varphi u_\rho^\e$,
		\begin{multline*}
			3\partial_\varphi u_\rho^\e = (u_\rho^\e)^{-2} \partial_\varphi \big ( (u_\rho^\e)^3\big ) 
			= -3 \underbrace{\left( \frac{\cos \varphi +2\e^\gamma}{u_\rho^\e}\right)^2}_{\leq 1 \text{ by \eqref{eq:bndUregionEprime}}} \sin \varphi 
			\\
			+ (u_\rho^\e)^{-2} \left [ \partial_r \left ( \frac{3r}{\partial_r \big ( -\cos f(r) \big ) }\right)\right]_{r=g(\varphi)}g'(\varphi)
			+ 3(u_\rho^\e)^{-2} 
			\underbrace{\int_{\sigma=0}^s \partial_\varphi h(\sigma, \varphi)\dd\sigma}_{=O(\e) \text{ by Lemma \ref{le:hE2Cos}}}.
		\end{multline*}
		The claim \eqref{eq:bndUregionEprime2} follows from
		$u_\rho^\e\geq 2\e^\gamma$ (see  \eqref{eq:bndUregionEprime}), 
		Lemma \ref{le:f_and_derivatives}.\ref{it:p_r-complicated}),
		and Lemma \ref{le:gPrime}.
	\end{proof}

	\begin{proof}[Proof of Proposition \ref{prop:convergence_energy} in \(e_\e'\)]
	We note that in this region \( \det D \vec u_\e =1\) and thus
	\[ \int_{e_\e'} H(\det D \vec u_\e) \dd \vec x=O(\e^3) .\]
We deal with the Dirichlet energy. 
	First, the parametrization of the region in the reference configuration is equivalent up to a reflection (changing the sign of $\nabla \varphi$ but not its magnitude) to the parametrization of the region \(a_\e'\).
		By \eqref{eq:DirichletAPrime} and Lemma \ref{le:estimatesUrhoRegionEprime},
		\begin{align*}
		 \int_{e_\e'} |D\vec u_\e|^2 \dd\vec x 
			\leq \underbrace{C\e^{4(\frac{1}{2}-\gamma)}}_{\leq 1}\int_{e_\e'} (\e\cos\varphi)^2|\nabla s|^2 
			+ 4\int_{e_\e'} |\nabla \theta|^2 \sin^2\varphi\, \dd\vec x
			+ C\e^{-4\gamma}\int_{e_\e'} |\nabla \varphi|^2\,\dd\vec x. 
		\end{align*}
		The estimate for $\int |D\vec u_\e|^2$ in this region now follows from Lemma \ref{le:gradVarphiAPrime}, which can be applied in region \(e_\e'\) thanks to a reflection argument.
	\end{proof}

    \subsubsection{Extra energy in $e_\e$ is negligible}
    We show here that the neo-Hookean energy of the recovery sequence in region \(e_\e\) converges to the neo-Hookean energy of the limit map in region \(e\), cf.\ \eqref{eq:convergence_energy_a_b_d_f} in Proposition \ref{prop:convergence_energy}.

Before doing the proof of this fact we start by stating some integrability properties of the function \(H\) in Theorem \ref{th:upper_bound}.

	\begin{lemma}
	 Let  $H$ be as in the statement of Theorem \ref{th:upper_bound}. Then,
		\begin{align*}
        \int_0^1 H(s^3)\dd s <\infty
        \quad \text{and}\quad 
        \int_1^\infty s^{-5/2}H(s)\dd s < \infty.
        \end{align*}
	\end{lemma}

	\begin{proof}
       We remark the expression of the Jacobian of the limit map:
    \begin{align*}
    \det D\vec u = \frac{u_\rho^2\partial_\rho u_\rho}{\rho^2}
    = (1+\rho)^2 \rho^{-2}\cos^3 \varphi.
    \end{align*}
     It vanishes as $\varphi\to \frac{\pi}{2}^-$ and goes to infinity as $\rho\to 0^+$ 
    (and $\cos\varphi$ remains away from zero). The neo-Hookean energy of the limit map is being assumed to be finite, so
        \begin{align*}
        \infty > \int_{e} H(\det D\vec u)\, \dd\vec x 
        &= 2\pi \int_{\varphi=0}^1\int_{\varphi=0}^{\pi/2}
        H\Big ( (1+\rho)^2\rho^{-2}\cos^3 \varphi \Big ) \rho^2 \sin\varphi\,\dd\varphi\, \dd\rho
        \\
        &\underset{t=\cos\varphi}{=}\int_{\rho=0}^1\int_{t=0}^1 H\Big ( (1+\rho)^2\rho^{-2}t^3\Big)\rho^2 \,\dd t \,\dd\rho.
        \end{align*}
        Note that
    \begin{equation*}
    \int_{\rho=0}^1\int_{t=0}^1 H\Big ( (1+\rho)^2\rho^{-2}t^3\Big)\rho^2 \,\dd t \,\dd\rho
    > \int_{\rho=\frac{1}{2}}^1 \int_{t=0}^{(\delta/16)^{1/3}} H(16t^3)\,\dd t\,\dd \rho
    \end{equation*}
    because if $\frac{1}{2}<\rho< 1$ and $t <(\delta/16)^{1/3}$ then 
    \begin{align*}
        (1+\rho)^2\rho^{-2}t^3 < (1+1)^2 (1/2)^{-2}t^3 < \delta
    \end{align*}
    and in $(0,\delta)$ the function $H$ decreases (see \eqref{eq:HdecreasingIncreasing}).
        The integrand is now independent of $\rho$ and changing variables to $s=16^{1/3}t$
        yields the integrability of $H(s^3)$ near zero.

    Regarding the integrability at infinity, it was already established in Lemma \ref{le:integrabilityHregionA}.
	\end{proof}

\begin{proof}[Proof of Proposition \ref{prop:convergence_energy} in \(e_\e\)]
For the Dirichlet part of the energy we use spherical coordinates and formula \eqref{eq:Dirichlet_energy_spherical_sperical}. We note that 
\begin{align*}
\left[\frac{3r}{\p_r(-\cos f(r))}\right]_{r=g(\varphi)}= \frac{3 g(\varphi)}{f'(g(\varphi))\sin \varphi}.
\end{align*}
We can compute \(D \vec u_\e\) and use that \( h(s,\varphi)=O(\e^2)\) and Lemmas \ref{le:gPrime}--\ref{le:hE2Cos} to prove that \(  D\vec u_\e \rightarrow D \vec u\) pointwise and that \( | D\vec u_\e| \leq C\) in \(e_\e\). Hence, from the dominated convergence theorem we obtain the convergence of the Dirichlet energy.

For the other part of the energy, using \eqref{eq:det_spherical_spherical}, one obtains the Jacobian determinant for $\vec u_\e$:
	\begin{align}
  \label{eq:boundDetRegionE}
	 \det D\vec u_\e = \rho^{-2}(u_\rho^\e)^2\partial_\rho u_\rho^\e 
		= \frac{1}{1-\e}\rho^{-2}(u_\rho^\e)^2 \big ( 2\cos\varphi + 6\e^{\gamma} - u_\rho^\e (\e,\varphi)\big ).
	\end{align}
	Note that
	\begin{equation*}
		(1+\e)\cos\varphi \leq \cos\varphi + \e 
		<\cos \varphi + 2\e^\gamma \leq u_\rho^\e(\e, \varphi),
	\end{equation*}
	so
	\begin{equation*}
		u_\rho(\rho,\varphi)=
	 (1+\rho)\cos\varphi  = \frac{1-\rho}{1-\e}\big ((1+\e)\cos\varphi\big) 
		+ \frac{\rho-\e}{1-\e} (2\cos\varphi)
		\leq u_\rho^\e(\rho,\varphi).
	\end{equation*}
	Using now \eqref{eq:refinedBoundUrhoRegionE},
	\begin{equation*}
	 \det D\vec u_\e 
\geq \frac{1}{1-\e} \rho^{-2} (u_\rho)^2 \cos\varphi = \frac{1}{1-\e}\det D\vec u\geq \det D\vec u.
	\end{equation*}

	Now a bound from above is needed for $\det D\vec u_\e$. 
	Plugging \eqref{eq:refinedBoundUrhoRegionE} along with the lower bound $u_\rho^\e(\e,\varphi)\geq \cos \varphi$ into \eqref{eq:boundDetRegionE} yields
	\begin{align*}
	 \det D\vec u\big (\vec x(\rho, \theta, \varphi)\big ) \leq \frac{1}{1-\e}\rho^{-2}\cdot \big (\cos\varphi + 2\e^\gamma +4\e^{\frac{1}{3}}\big )^2 \big ( \cos\varphi + 6\e^{\gamma}\big )
	\leq 2\rho^{-2}
	\end{align*}
	for every $\rho$ and $\varphi$ (provided $\e$ is sufficiently small). 

	These bounds, with exactly the same proof of as the part of Proposition \ref{prop:convergence_energy} relative to region $a_\e$, gives the desired conclusion:
    $$\lim_{\e\to 0}  
    \int_{e_{\e}} H(\det D\vec u_{\e}) \dd\vec x 
    = \int_{e} H(\det D\vec u)\dd \vec x.
        $$
\end{proof}

  \subsubsection{Extra energy in  $f$ is negligible}

We now deal with region \(f\) and prove Proposition \ref{prop:convergence_energy} for this region.
\begin{proof}[Proof of Proposition \ref{prop:convergence_energy} in \(f\)]
It is direct to see that, in region \(f\), we have \( D \vec u_\e \rightarrow D \vec u\) almost everywhere and \( |D\vec u_\e|\leq C(|D \vec u|^2+1)\). Then by dominated convergence theorem we obtain the convergence of the Dirichlet energy.
For the other part of the energy. First note (using \eqref{eq:det_spherical_spherical}) that 
	\begin{equation*}
	 \det D\vec u_\e = \frac{(u_\rho + 6\e^{\gamma})^2 \sin u_\varphi}{\rho^2\sin \varphi} \matrizzDet{\partial_\rho u_\rho & \partial_\varphi u_\rho \\
	\partial_\rho u_\varphi & \partial_\varphi u_\varphi},
\quad
	\det D\vec u =
	\frac{u_\rho^2 \sin u_\varphi}{\rho^2\sin \varphi} \matrizzDet{\partial_\rho u_\rho & \partial_\varphi u_\rho \\
	\partial_\rho u_\varphi & \partial_\varphi u_\varphi} .
	\end{equation*}
	On the one hand, this shows that $\det D\vec u_\e \geq \det D\vec u$. It remains to bound the Jacobian from above. Here we decompose $f$ in the part where $0\leq \varphi<\frac{\pi}{4}$ and that where $\frac{\pi}{4}\leq \varphi\leq\frac{\pi}{2}$. In the first part use that
	$$
		u_\rho(\rho,\varphi)\geq u_\rho(1,\varphi)=2\cos\varphi \geq \sqrt{2}\geq 6\e^{\gamma},
	$$
	which holds (at least close to $\rho=1$, where the regularization of constructing the sequence $\vec u_\e$ is required)
	since $\vec u$ is orientation-preserving and injective, so increasing in $\rho$. Then
	\begin{align*}
	 \det D\vec u_\e \leq 
\frac{(u_\rho + u_\rho)^2 \sin u_\varphi}{\rho^2\sin \varphi} \matrizzDet{\partial_\rho u_\rho & \partial_\varphi u_\rho \\
	\partial_\rho u_\varphi & \partial_\varphi u_\varphi}
	=4\det D\vec u.
	\end{align*}
	In $\{\varphi\in[\frac{\pi}{4},\frac{\pi}{2}]\}$ just consider that 
	\begin{align*}
	 \det D\vec u_\e \leq \frac{(u_\rho + 6\e^{\gamma})^2 \sin u_\varphi}{1^2 \cdot \sin \frac{\pi}{4}} \matrizzDet{\partial_\rho u_\rho & \partial_\varphi u_\rho \\
	\partial_\rho u_\varphi & \partial_\varphi u_\varphi}
	\end{align*}
	and that the right-hand side is uniformly bounded since the limit map $\vec u$ is Lipschitz. With these bounds at hand and the hypothesis on $H$ in Theorem \ref{th:upper_bound}, it is possible to obtain the conclusion invoking the dominated convergence theorem.
\end{proof}

	\subsubsection{Extra energy in $d_\e$ is negligible}
 \label{se:energyD}

We conclude this section by proving the convergence of the neo-Hookean energy of the recovery sequence towards the energy of the limit map in region~\(d_\e\).

\begin{proof}[Proof of Proposition \ref{prop:convergence_energy} in \(d_\e\)]
We recall that in region \( d_\e\), the map \(\vec u_\e\) is the composition of three maps: \(\vec u_\e=\vec w_\e 
\circ \vec g \circ \hat{\vec R}\) where we denoted by \( \hat{ \vec R}\) the map which is described in spherical coordinates by \( (r,
\theta,\varphi) \mapsto (\hat{r},\theta,\varphi)\) with \( \hat{r}\) defined in \eqref{eq:rPrimeRegionD}. We thus have \( D \vec u_\e=D \vec 
w_\e (\vec g\circ \hat{\vec R}) D \vec g( \hat{\vec R})D \hat{\vec R}\). From this and the expressions of \( \vec w_\e \) and \( \vec R\), it is 
easily seen that \( \chi_{d_\e}D \vec u_\e \rightarrow \chi_d D \vec u\) almost everywhere and that \( |D 
\vec u_\e |\leq C\) for some \(C >0\). Hence by dominated convergence we obtain the convergence of the Dirichlet energy.

We now examine the convergence of the determinant part of the energy. We first examine the determinant of the differential of the limit map to understand the integrability property of \(H\). Let $\vec g$ be the bi-Lipschitz map from $\{\hat r\vec e_r+ x_3 \vec e_3: \hat r\geq 0, x_3\in[0,1]\}$ onto $\{s\vec e_r + z\vec e_3: s\geq 0, z\in[0,3]\}$ of Section \ref{se:LimitMap}.
 Since $\vec g$ is Lipschitz and $s(\hat r,x_3)=0$ when $\hat r=0$ (because that corresponds to the values of $\vec g$ on the segment $B'C'$),
	\begin{align}
 \label{eq:sRhat}
		\left |\frac{s(\hat r,x_3)}{\hat r} \right |
	= \left | \fint_0^{\hat r} \frac{\partial s}{\partial \hat r} (t,x_3)\dd t \right | 
	= \left | \fint_0^{\hat r} \Big (\big (D\vec g(t\vec e_r+x_3\vec e_3)\big ) \vec e_r\Big ) \cdot \vec e_r\dd t \right |
	\leq \|D\vec g\|_\infty.
	\end{align}
	Since $\vec g^{-1}$ is Lipschitz,
	\begin{align}
 \label{eq:LwrBndMinor2by2}
		0< c \leq \det D\vec g = \frac{s(\hat r,x_3)}{\hat r} \frac{\partial(s,z)}{\partial(\hat r,x_3)}
		\leq \|D\vec g\|_\infty \frac{\partial(s,z)}{\partial(\hat r,x_3)}.
	\end{align}
	It can be seen that
	\begin{equation}
 \label{eq:detLimitD}
	\det D\vec u \big ( \hat r\vec e_r(\theta) + x_3\vec e_3\big ) = \frac{\pi}{12} \frac{s^2\sin\varphi}{\hat r} \frac{\partial (s,z)}{\partial (\hat r,x_3)}.
	\end{equation}
	In addition,
	\begin{equation}
 \label{eq:2by2minorSz}
	\frac{\partial (s,z)}{\partial (\hat r,x_3)} = \left | \frac{\partial \vec g}{\partial \hat r}\wedge \frac{\partial \vec g}{\partial x_3}\right |\leq \|D\vec g\|_\infty^2.
	\end{equation}
	From the above it follows that
	\begin{align*}
	\frac{c}{\|D\vec g\|_\infty}\frac{\pi}{12} \frac{s^2\cdot \sin \frac{\pi}{4}}{\hat r}
	\leq \det D\vec u \big ( \hat r\vec e_r(\theta) + x_3\vec e_\theta\big ) 
	\leq \frac{\pi}{12} \frac{s^2 \cdot 1}{\hat r}\|D\vec g\|_\infty^2.
	\end{align*}
	This amounts to saying that $\det D\vec u \sim s^2/\hat r$ and that 
	$$
		\infty > \int_{ d} H(s^2/\hat r)\dd\vec x 
		\geq 2\pi \int_{\hat r=0}^{3 \sin(\arccos(1/3))} \int_{x_3=0}^1 H \Big ( \frac{s^2(\hat r,x_3)}{\hat r}\Big ) \hat r\dd x_3 \dd \hat r.
	$$
	Since $s(\hat r,x_3)=0$ in the segments  $A'B'$ and $C'D'$, where $\hat r>0$, the determinant vanishes as $\vec x$ tends to those segments from region $d$. In contrast, the determinant does not blow up to infinity since the singularity $\hat r=0$ in the denominator is removable (recall that $s/\hat r\leq \|D\vec g\|_\infty$).

	Now we consider the determinant of \( D \vec u_\e\). Let $\vec w_\e$ be the auxiliary transformation of \eqref{eq:wEps}. We observe that 
	\begin{equation}\label{eq:det3w}
	\begin{split}
	 &\frac{\partial \vec w_\e}{\partial s} \wedge \frac{\partial \vec w_\e}{\partial \theta} \wedge \frac{\partial \vec w_\e}{\partial z}
	\\&= (\sin\varphi\,\vec e_r -\cos \varphi\,\vec e_3) \wedge (w_r^\e \vec e_\theta) \wedge \Big (  \omega_\e'(z)\vec e_r+\frac{\pi}{12} s\,\cos \varphi\,\vec e_r + \frac{\pi}{12} s\,\sin\varphi\,\vec e_3\Big )
	\\
	&= \begin{vmatrix} \sin \varphi & 0 & \omega_\e' + \frac{\pi}{12}s\cos\varphi \\  0 & w_r^\e & 0 \\ -\cos\varphi & 0 & \frac{\pi}{12}s\sin\varphi \end{vmatrix}
	= w_r^\e \cdot \Big (\frac{\pi}{12}s +\omega_\e'(z)\cos\varphi \Big ).
	\end{split}
	\end{equation}
	Since $\vec u_\e$ is defined with cylindrical coordinates,
	\begin{align}
	& \det D\vec u_\e \big ( r\vec e_r(\theta) + x_3 \vec e_3 \big ) 
	= \frac{1}{r} \frac{\partial \vec u_\e}{\partial r}\wedge \frac{\partial \vec u_\e}{\partial \theta} \wedge \frac{\partial \vec u_\e}{\partial x_3} \nonumber
	\\
	&= \frac{1}{r} \Big ( \frac{\partial \vec w_\e}{\partial s} \frac{\partial s}{\partial \hat r}\frac{\partial \hat r}{\partial r} + \frac{\partial \vec w_\e}{\partial z}\frac{\partial z}{\partial \hat r} \frac{\partial \hat r}{\partial r} \Big ) \wedge \frac{\partial \vec w_\e}{\partial \theta}\wedge \Big ( \frac{\partial \vec w_\e}{\partial s}\frac{\partial s}{\partial x_3} + \frac{\partial \vec w_\e}{\partial z}\frac{\partial z}{\partial x_3}\Big ) \nonumber
	\\ 
 \label{eq:detDuDw}
	& =\frac{1}{r} \Big ( \frac{\partial \vec w_\e}{\partial s} \wedge \frac{\partial \vec w_\e}{\partial \theta} \wedge \frac{\partial \vec w_\e}{\partial z}\Big ) \frac{\partial \hat r}{\partial r} \Big ( \frac{\partial s}{\partial \hat r}\frac{\partial z}{\partial x_3} - \frac{\partial s}{\partial x_3}\frac{\partial z}{\partial \hat r}\Big ) .
	\end{align}

	Consider first the range $\e\leq r\leq \e^{2\gamma}$:
	\begin{align}
 \label{eq:detBetweenEpsAndEpsGamma}
	\det D\vec u_\e \big ( r\vec e_r(\theta) + x_3 \vec e_3 \big ) 
	= \frac{w_r^\e (2r-\e) }{r(\e^{2\gamma}-\e)} \Big (\frac{\pi}{12}s +\omega_\e'(z)\cos\varphi(z) \Big )\frac{\partial (s,z)}{\partial (\hat r,x_3)}.
	\end{align}
	By
	\eqref{eq:2by2minorSz}, and since $\e< \frac{1}{2}\e^{2\gamma}$ for small $\e$,
	\begin{align*}
	 \det D\vec u_\e \big ( r\vec e_r(\theta) + x_3\vec e_3\big ) 
\leq \frac{w_r^\e \cdot (2r)}{r (\frac{1}{2}\e^{2\gamma})}\Big (\frac{\pi}{12}s +\omega_\e'(z)\cos\varphi(z)
 \Big ) \|D\vec g\|_\infty^2.
	\end{align*}
	By \eqref{eq:sRhat},
	\[
	 w_r^\e \big ( s(\hat r, x_3), z(\hat r, x_3) \big ) = \omega_\e(z) + s(\hat r, x_3)\sin (\varphi(z))
 \leq \omega_\e(3) + \|D\vec g\|_\infty \hat r 
 \leq 6\e^{\gamma} + \|D\vec g\|_\infty \frac{(r-\e)r}{\e^{2\gamma}-\e}.
\]
	Since $r\leq \e^{2\gamma}$, it follows that $w_r^\e \leq C\e^{\gamma}$. By direct computation we can check that \( |\omega_\e'(z)|\leq C \e^\gamma \). Thus we find
	\begin{equation*}
	 \frac{\pi}{12}s(\hat r, x_3) +\omega_\e'\big ( z(\hat r, x_3)\big ) \cos\varphi\big(z(r,x_3)\big ) \leq C\e^{\gamma}.
	\end{equation*}
	Therefore,
	\begin{equation*}
		\sup \Big \{
		\det D\vec u_\e \big (r\vec e_r(\theta)+x_3\e\big ): \e\leq r\leq \e^{2\gamma}\Big \} < \infty.
	\end{equation*}
Here we recall that we chose \( \gamma \leq 1/3\).
The radial distance $r$ is greater than $\e$, hence $r>2r-\e$. Also, $w_r^\e(r)\geq 2\e^\gamma$ and 
$\omega_\e'(z)\geq C\e^\gamma$. Using these bounds  in \eqref{eq:detBetweenEpsAndEpsGamma}
yields 
	\begin{align*}
	  \det D\vec u_\e \big ( r\vec e_r(\theta) + x_3 \vec e_3 \big )
		\geq \frac{2\e^\gamma}{\e^{2\gamma}-\e} \big ( 0 + C\e^\gamma\cos\varphi\big ) \frac{\partial(s,z)}{\partial(\hat r, x_3)}
	\geq C\cos\varphi,
	\end{align*}
	where in the last inequality also  \eqref{eq:LwrBndMinor2by2} was used.
At this point let us observe that $\cos\varphi$ is bounded below away from zero when $\e \leq r\leq \e^{2\gamma}$. Indeed, when $\hat r=0$, the point $(\hat r, x_3)$ lies on the segment $BC$ and $z(\hat r, x_3)$ lies between $z=1$ and $z=2$. Therefore, when $\hat r$ is close to zero, $z(\hat r, x_3)$ must be near the interval $[1,2]$, hence $\cos\varphi$ should be bounded below by a value close to $\cos\varphi(2)=\cos(5\pi/12)$. This is made rigorous by noting that 
	\begin{align*}
	  z(\hat r,x_3)&= z(0,x_3) + \int_0^{\hat r} \frac{\partial z}{\partial r} (t,x_3)\dd t \leq 2 + 
	 \int_0^{\hat r} \big ( ( D\vec g)\vec e_r\big ) \cdot \vec e_3 \dd t 
	 \leq 2 + \|D\vec g\|_\infty \hat r.
	\end{align*}
	Since $\hat r \leq \e^{2\gamma}$ when $r\leq \e^{2\gamma}$, then, for $\e$ sufficiently small, $\cos\varphi$ is greater than any value smaller than $\cos (5\pi/12)$. All in all, in the region $\e\leq r\leq \e^{2\gamma}$ the determinants are controlled from below and above, and
	\begin{align*}
	 \lim_{\e\to0}
	 2\pi
	 \int_0^1  \int_\e^{\e^{2\gamma}} H(\det D\vec u_\e)\, r\,\dd r\,\dd x_3 = 0.
	\end{align*}

	Consider now the range $\e^{2\gamma}\leq r\leq 3$. By the definition of $\omega_\e$ and  \eqref{eq:sRhat}, 
	\begin{align*}
	 w_r^\e = \omega_\e(z(r,x_3)) + s(r,x_3)\sin \varphi \leq 6\e^{\gamma}  + \|D\vec g\|_\infty r
	\leq (6+\|D\vec g\|_\infty)\max\{r, \e^{\gamma}\}.
	\end{align*}
	Analogously,
	\begin{align*}
		\frac{\pi}{12}s(r,x_3) + \omega_\e'(z(r,x_3))\cos \varphi\big ( z (r,x_3)\big ) \leq C\max\{r,\e^{\gamma}\}.
	\end{align*}
	Hence,
	using \eqref{eq:det3w} and \eqref{eq:detDuDw}
	(recall that $\hat r=r$ for $r\geq \e^{2\gamma}$),
	it follows that
	\begin{align*}
	 \det D\vec u_\e (r\,\vec e_r(\theta) + x_3\vec e_3\big ) \leq C\frac{\max\{r, \e^{\gamma}\}^2}{r}
	\leq \begin{cases}
	      C\e^{2\gamma}/r, & \text{if } \e^{2\gamma}\leq r \leq \e^{\gamma}
\\
	Cr^2/r, & \text{if } \e^{\gamma}\leq r\leq 3.
	     \end{cases}
	\end{align*}
	In both cases, the expression at the right-hand side is clearly bounded. Therefore, $\det D\vec u_\e$ is bounded from above in $d_\e$ uniformly with respect to $\e$. 

	Regarding the lower bound, it suffices to see, from \eqref{eq:det3w} and \eqref{eq:detDuDw}, that when $r\geq \e^{2\gamma}$,
	\begin{align*}
	 \det D\vec u_\e \big (r\,\vec e_r(\theta)+ x_3\,\vec e_3\big ) \geq \frac{1}{r} \cdot \big (s\,\sin \varphi \big)\Big (\frac{\pi}{12}s \Big )\cdot 1 \cdot \frac{\partial (s,z)}{(r,x_3)},
	\end{align*}
	where $s=s(r,x_3)$ and $\varphi=\varphi\big (z(r,x_e)\big)$ (recall that $\hat r\equiv r$ in the range of values of $r$ being considered in this part of the proof). Comparing with  \eqref{eq:detLimitD}, the conclusion is obtained that 
	$$
		\det D\vec u_\e \big (r\,\vec e_r(\theta)+ x_3\,\vec e_3\big ) \geq \det D\vec u \big (r\,\vec e_r(\theta)+ x_3\,\vec e_3\big ) 
\qquad \forall\, r\geq \e^{2\gamma}, \ \forall\, x_3\in[0,1].
	$$
	With this estimate at hand, it is possible to find a dominating function to prove the desired convergence of the energies in region $d_\e$. 
\end{proof}

\section{Geometric and topological description of singularities as dipoles}\label{sec:VI}

This section shows that the singularities that have to be dealt with in the neo-Hookean problem share the same structure as the elastic dipole constructed by Conti \& De Lellis \cite[Th.~6.3]{CoDeLe03}, at least in the axisymmetric class and under the additional hypothesis that the functional $\E(\vec u)$ (defined in \eqref{eq:def_surface_energy}, which measures \cite{HeMo11} the deformed area of the surfaces created by $\vec u$) is finite.

\begin{proof}[Proof of Theorem \ref{th:main2}]
	It is a consequence of Propositions \ref{prop:previous}, \ref{pr:muproperties}, \ref{prop:proprety_Delta}, \ref{prop:DeltaNonZero}, \ref{prop:surface}, and \ref{prop:sumDegrees} below.
\end{proof}

The map by Conti \& De Lellis has two point singularities (in the reference configuration), namely, $\vec 0=(0,0,0)$ and 
$\vec 0'=(0,0,1)$, which are paired together. 
They are both cavitation points, but the cavities they open have ``opposite signs''. A bubble (the sphere in the deformed configuration centred at $(0,0,\tfrac{1}{2})$ and of radius $\tfrac{1}{2}$; denote it by $\Gamma$) is created from (horizontal $\e$-disks centred at) both points (in the sense of the bubbling-off of spheres in harmonic map theory, as explained in Section \ref{sec:recovery_sequence}). However, the material points surrounding $\vec 0'$ (those in region $e$) are sent to the outer side of $\Gamma$, whereas the points surrounding $\vec 0$ (those in region $a$) are sent to the radius $\tfrac{1}{2}$ ball $\hat B$ enclosed by $\Gamma$. Thus, the ``positive'' cavity opened 
at $\vec 0'$ is filled with the portion of the body next to the ``negative'' cavitation point $\vec 0$. 

Cavitation singularities affect the topology of the elastic body: first, the domain is punctured. 
Then, in the case of a positive cavitation point,  a hole is observed after the body is deformed. 
In particular, the boundary of the image $\vec u(\Set\setminus \{\vec 0\})$ of any smooth region $\Set$ 
containing $\vec 0$ but not $\vec 0'$ will have the additional connected component $\Gamma$ apart from 
the image of the boundary $\vec u(\partial \Set)$. The situation around the negative cavitation point $\vec 0'$ 
is analogous. As has been pointed out since \cite[Example 6.1]{Ball77}, these changes in topology are detected 
by the distributional determinant $\Det D\vec u$ of \eqref{eq:defDet}, which in turn is related to the 
Brouwer (and the Brezis--Nirenberg) degree. 
In particular, 
$$
	\Det D\vec u = (\det D\vec u)\Le^3 + \frac{\pi}{6}\delta_{(0,0,1)} - \frac{\pi}{6} \delta_{(0,0,0)}.
$$
The $\frac{\pi}{6}$ in front of the Dirac masses is the volume enclosed by the bubble $\Gamma$, and the negative sign in front of the Dirac at $\vec x=\vec 0$ is indicative of the reversed orientation in which the bubble is being reached as $\vec x\to \vec 0$ compared to the orientation obtained of $\Gamma$ when $\vec x\to \vec 0'$.

The positive cavity opened at $\vec 0'$ can be seized by taking small $r$ and looking at the set 
$$
	\{\vec y\in \R^3: \deg (\vec u, B_r(\vec 0'),\vec y)=1\ \text{ and }\ \vec y \notin \imG\big ( \vec u, B_r(\vec 0')\big )\}
$$
(that is, the points that are enclosed by $\vec u\big (\partial B_r(\vec 0')\big )$ but are not material points). This corresponds to the notion of the topological image of $\vec 0'$, introduced by \v{S}ver\'ak \cite{Sverak88}. For maps satisfying condition INV (maps that are monotone in the sense of sending inside what was inside, and outside what was outside) and with geometric images $\imG(\vec u, \Om)$ of finite perimeter, M\"uller \& Spector \cite[Th.\ 8.4]{MuSp95} established that the cavitation points can be identified as points $\vec \xi$ for which 
$$
	\Le^3\big ( \imT(\vec u, \vec \xi) \big ) >0,
$$
and that there is an at most countable set $C(\vec u)\subset \Om$ of them. 
In the $H^1$ setting, where, as shown by Conti \& De Lellis \cite{CoDeLe03}, condition INV is unstable and negative cavitations may occur, the natural analogue for identifying the singularities is to fix a point $\vec x\in \Om$, to define the map 
$$
	\Delta_{\vec x, r}(\vec y):= \deg \big (\vec u, B(\vec x, r), \vec y\big ) - \chi_{\imG\left(\vec u, B(\vec x, r)\right)}(\vec y),
$$
and to let $r\to 0^+$. For $\vec x=\vec 0'$ and $\vec x =\vec 0$ we obtain the maps
$$
\Delta_{\vec 0'} = \chi_{\hat B} , 
\qquad \qquad
\Delta_{\vec 0}(\vec y) = - \chi_{\hat B} .
$$
For any other $\vec x$, the map $\Delta_{\vec x}$ is identically zero. 

Our starting point is the result by Mucci \cite{Mucci05,Mucci10a,Mucci10b}, \cite[Th.\ 6.2]{HeMo12} that for $H^1$ maps 
with $\E(\vec u)<\infty$ the singular part of the distributional determinant consists always only of Dirac masses (at most a countable number of them), even when INV is not satisfied. Developing further his arguments, as well as those in \cite{MuSp95} and \cite{HeMo12}, we arrive at Theorem \ref{th:main2}, stated in the Introduction.
To prove the theorem, we begin by mentioning
that the bilinear form
\[
 \overline {\E}_{\vec u}(\phi, \vec g), \quad \text{with } \phi\in C_c(\Om) \text{ and }\vec g\in C^1_c(\R^3,\R^3) ,
\]
the surface energy $\E(\vec u)$, and the measure
\begin{equation*}
\mu_{\vec u}(E):= \sup \{ \overline{\E}_{\vec u}(E,\vec g) : \vec g \in C^1_c(\R^3,\R^3), \|\vec g\|_{L^\infty}\leq 1 \},
	\quad  \text{on Borel subsets $E$ of }\Om
\end{equation*}
(Definition \ref{def:measure_mu})
shall play an important auxiliary role.

In this section, we will use standard notation about measure theory and functions of bounded variation: push-forward $\sharp$ of a measure \cite[Def.\ 1.70]{AmFuPa00}, restriction $\res$ of a measure \cite[Def.\ 1.65]{AmFuPa00}, set $\At$ of atoms of a measure \cite[Sect.\ 6]{HeMo12}, total variation $|\cdot|$ of a measure \cite[Def.\ 1.4]{AmFuPa00}, reduced boundary $\p^*$ of a finite-perimeter set \cite[Def.\ 3.60]{AmFuPa00}, space $SBV$ \cite[Ch.\ 4]{AmFuPa00}, set $J_{\vec u^{-1}}$ of jumps of $\vec u^{-1}$ \cite[Def.\ 3.67]{AmFuPa00}, traces $(\vec u^{-1})^{\pm}$ and normal $\vec \nu$ of $\vec u^{-1}$ at a jump point \cite[Prop.\ 3.69]{AmFuPa00}, jump $[\cdot]$ of a $BV$ function \cite[Sect.\ 5]{BaHeMoRo_21}, singular part $D^s$, jump part $D^j$ and Cantor part $D^c$ of the derivative of a $BV$ function \cite[Def.\ 3.91]{AmFuPa00}.
Finally, $\M (\R^3)$ denotes the family of the finite Radon measures on $\R^3$, while
$\|\cdot\| = |\cdot| (\R^3)$.

We recall also Definition \eqref{eq:defAs} of the class $\overline{\Ar_s}$, and gather some results from previous papers. 

\begin{proposition}\label{prop:previous}
Let $\vec u \in \overline{\Ar_s}$ satisfy $\E(\vec u) < \infty$.
Then
\begin{enumerate}[a)]
\item\label{item:previousA} $\vec u^{-1}\in SBV(\Om_{\vec b},\R^3)$.

\item \label{item:previousDet} The distributional determinant $\Det D\vec u$ is a measure and 
\begin{equation*}
\Det D\vec u= (\det D\vec u)\Le^3+\sum_{\vec a\in \At(\mu_\vec u)} \Det D\vec u(\{ \vec a \})\delta_{\vec a}.
\end{equation*}

\item\label{item:previousB} For every $\vec g\in C^1_c(\R^3,\R^3)$ the measure $\overline{\E}_{\vec u}(\cdot,\vec g)$ is purely atomic and $$\At(\overline{\E}_{\vec u}(\cdot,\vec g)) \subset \At (\mu_\vec u).$$
\end{enumerate}
\end{proposition}

Item \ref{item:previousA}) of the previous proposition follows from \cite[Th.\ 2]{HeMo11}. Items \ref{item:previousDet}) and  \ref{item:previousB}) are the result of \cite{Mucci05,Mucci10a,Mucci10b}, \cite[Th.\ 6.2]{HeMo12}. This formula generalizes an analogous formula for the distributional Jacobian due to M\"uller and Spector \cite[Th.\ 8.4]{MuSp95} and to Conti and De Lellis \cite[Th.\ 4.2]{CoDeLe03}. The points \( \vec a\) in \(\At (\mu_{\vec u})\) can be seen as ``generalized'' cavitation points. Contrarily to the classical cavitation setting, \(\Det D\vec u(\{ \vec a \})\) is not necessarily positive. Nevertheless, its absolute value can be thought of as the volume of the generalized cavity.

\begin{definition}
Let \(\vec u \in \overline{\Ar_s }\) be such that \( \E(\vec u)<\infty \) and let $\vec \xi, \vec \xi' \in \R^3$.
Fix a Borel orientation $\vec \nu$ of $J_{{\vec u}^{-1}}$.
We define \( \Gamma_{\vecg\xi}^+\), \( \Gamma_{\vecg\xi}^-\), $\Gamma_{\vecg\xi}$ and $\Gamma_{\vecg\xi, \vecg\xi'}$ as
\[
 \Gamma_{\vecg\xi}^{\pm} := \{\vec y\in J_{{\vec u}^{-1}}: ({\vec u}^{-1})^{\pm}(\vec y)= \vecg\xi\} , \qquad \Gamma_{\vecg\xi}:= \Gamma_{\vecg\xi}^- \cup \Gamma_{\vecg\xi}^+, \qquad \Gamma_{\vecg\xi, \vecg\xi'} := \Gamma_{\vecg\xi}^- \cap \Gamma_{\vecg\xi'}^+ .
\]
The map $\vec \nu_{\vec \xi} : \Gamma_{\vec \xi} \to \R^3$ is defined $\Ha^2$-a.e.\ as
\[
 \vec \nu_{\vec \xi} = \begin{cases}
 \vec \nu & \text{in } \Gamma_{\vecg\xi}^-, \\
 -\vec \nu & \text{in } \Gamma_{\vecg\xi}^+ .
  \end{cases}
\]
\end{definition}
One side of $\Gamma_{\vecg\xi}$ consists of material points that were located near  $\vecg\xi$ in the reference configuration, whereas the other side consists of a portion of the body coming from a different cavitation point. Note that $\vecg\nu_{\vecg\xi}$ points towards the latter side.

In the sequel use shall be made of the notation 
\( \vec f \bowtie \vec g : \R^3 \rightarrow \R^3 \times \R^3\), 
$$\vec f \bowtie \vec g (\vec y)=(\vec f(\vec y),\vec g (\vec y)).$$
Recall also Definition \eqref{eq:segment_L} of the segment $L$ where the singularities are confined.

\begin{proposition}\label{pr:muproperties}
Let $\vec u \in \overline{\Ar_s}$ satisfy $\E(\vec u) < \infty$.
Then:
  \begin{enumerate}[i)]
\item\label{it:Lambda}  For all $\vec f\in C_c(\Om \times \R^3, \R^3)$,
\begin{equation}\label{eq:E=Lambda}
 \E_{\vec u}(\vec f) =  \int_{\Om \times \R^3} \vec f \cdot \dd \vec \Lambda
\end{equation}
for the Radon measure
\[
    \vecg\Lambda = \big ( ({\vec u}^{-1})^-\bowtie \id \big ) _{\#}\vecg\nu \Ha^2\res {J_{\vec u^{-1}}}
      - \big ( ({\vec u}^{-1})^+\bowtie \id \big ) _{\#}\vecg\nu \Ha^2\res {J_{\vec u^{-1}}}.
\]
\item \label{it:E_structure} For all $\vec f\in C_c(\Om \times \R^3, \R^3)$,
\[
	  \E_{\vec u}(\vec f) =
	  \sum_{\vecg \xi \in\At(\mu_{\vec u})} \int_{\Gamma_{\vecg\xi}} \vec f( \vecg \xi, \vec y) \cdot \vecg\nu_{\vecg\xi} (\vec y) \dd \Ha^2(\vec y).
\]
      
      \item \label{it:mu_atomic}
$\displaystyle \mu_{\vec u} = \sum_{\vecg\xi \in\At(\mu_{\vec u})}  \Ha^2\big (\Gamma_{\vecg\xi}\big )\,  \delta_{\vecg\xi}$.

       \item  \label{it:jump_atoms}
       $(\vec u^{-1})^{\pm} (\vec y) \in \At(\mu_{\vec u})$ for \(\Ha^2 \)-a.e.\ $\vec y\in J_{\vec u^{-1}}$.
       
       \item \label{it:sym-axis}
$\At (\mu_{\vec u}) \subset L$.
       
       \item \label{it:dipoles}
       $ \displaystyle
              \|D^s \vec u^{-1} \|
              = \sum_{\vecg \xi, \vecg\xi' \in\At(\mu_{\vec u})}
		  \left| \vecg\xi-\vecg\xi' \right| \Ha^2 (\Gamma_{\vecg\xi, \vecg\xi'})$.

\item\label{it:JG} $ \displaystyle \Ha^2 (J_{\vec u^{-1}}) = \! \sum_{\vecg\xi \in\At(\mu_{\vec u})} \! \Ha^2 (\Gamma_{\vec \xi}^+) = \!\sum_{\vecg\xi \in\At(\mu_{\vec u})} \! \Ha^2 (\Gamma_{\vec \xi}^-) = \frac{1}{2} \! \sum_{\vecg\xi \in\At(\mu_{\vec u})} \! \Ha^2 (\Gamma_{\vec \xi}) = \! \sum_{\vecg \xi, \vecg\xi' \in\At(\mu_{\vec u})} \!\Ha^2 (\Gamma_{\vecg\xi, \vecg\xi'})$.
  \end{enumerate}
\end{proposition}

\begin{proof}
We prove each item.

\smallskip

\emph{Proof of \ref{it:Lambda}).}
By \cite[Th.\ 3]{HeMo11}, equality \eqref{eq:E=Lambda} holds for the Radon measure
\[
 \vec \Lambda = \big ( ({\vec u}^{-1})^-\bowtie \id \big ) _{\#}\vecg\nu \Ha^2\res {\Gamma_V (\vec u)}
 + \left( \big ( ({\vec u}^{-1})^-\bowtie \id \big ) - \big ( ({\vec u}^{-1})^+\bowtie \id \big ) \right)_{\#} \vecg\nu \Ha^2\res {\Gamma_I (\vec u)} ,
\]
where $\Gamma_V (\vec u)$ and $\Gamma_I (\vec u)$ are the visible and invisible surfaces, respectively, defined as follows:
\[
 \Gamma_I(\vec u):= \left\{ \vec y \in J_{\vec u^{-1}}:\ (\vec u^{-1})^+ (\vec y) \in\Om \ \textrm{and}
 \ (\vec u^{-1})^-(\vec y) \in \Om \right\}.
\]
The visible surface is denoted by $\Gamma_V (\vec u)$, as the set of points $\vec y_0\in \R^n$ for which there exists $\boldsymbol \nu \in \Sf^{n-1}$ satisfying the following conditions:
\begin{enumerate}[i)]
\item $D(\imG(\vec u, \Om) \cap H^-(\vec y_0, \boldsymbol \nu), \vec y_0) = \frac{1}{2}$.

\item The lateral trace
\[
 (\vec u^{-1})^-(\vec y_0) = \operatorname*{ap\,lim}_{\substack{\vec y \to \vec y_0 \\ \vec y \in H^-(\vec y_0, \boldsymbol \nu)}} \vec u^{-1}(\vec y)
\]
exists and lies inside $\Om$.

\item $D(\imG(\vec u, U)\cap H^+(\vec y_0, \boldsymbol \nu), \vec y_0)=0$ for every open set $U \Subset \Om$.
\end{enumerate}
Here $D$ stands for the density of a set at a point, the definition of which can be found, e.g., in \cite[Def.\ 2.2]{MuSp95}, \cite[Sect.\ 2]{HeMo10}, \cite[Sect.\ 2.1]{BaHeMoRo_21}.
The definition of approximate limit $\operatorname*{ap\,lim}$ is also standard \cite[Def.\ 2.2]{MuSp95}, \cite[Def.\ 1]{HeMo11}.
Finally, $H^{\pm}(\vec y_0, \boldsymbol \nu)$ is the half-space with inequation $\pm (\vec y - \vec y_0) \cdot  \boldsymbol \nu \geq 0$.

The proof of \emph{\ref{it:Lambda})} will be finished as soon as we show that $\Gamma_V (\vec u) = \varnothing$ and $\Gamma_I (\vec u) = J_{\vec u^{-1}}$ $\Ha^2$-a.e.

Assume, for a contradiction, that there exists $\vec y_0 \in \Gamma_V (\vec u)$. In particular, this implies that $(\vec u^{-1})^-(\vec y_0) \in \Om$.
Now take any $U \Subset \Om$ such that $(\vec u^{-1})^-(\vec y_0) \in U$ and $\Om \subset U$.
By \cite[Lemma 5.\emph{ii).c)}]{HeMo11}, $\vec y_0 \in \p^* \imG (\vec u, U)$.
Now, 
by \cite[Prop.~5.1]{BaHeMoRo_21},
$\imG (\vec u, \Om) = \Om_{\vec b}$ a.e., so $\imG (\vec u, U) = \vec b (U)$ a.e\@.
Therefore, $\p^* \imG (\vec u, U) = \p^* \vec b (U)$.
Now, $\p^* \vec b (U) \subset \p \vec b (U) = \vec b (\p U)$.
Thus, $\vec y_0 \in \vec b (\p U)$.
Now we take two open sets $U_1$ and $U_2$ with the same properties as $U$ before and such that, additionally, $\p U_1 \cap \p U_2 = \varnothing$.
Then $\vec y_0 \in \vec b (\p U_1) \cap \vec b (\p U_2)$, which is a contradiction since $\vec b$ is injective.

Now, to prove $\Gamma_I (\vec u) = J_{\vec u^{-1}}$ $\Ha^2$-a.e.\
it suffices to show that the set of points $\vec y \in J_{\vec u^{-1}}$ for which one of the traces $(\vec u^{-1})^{\pm} (\vec y)$ belongs to $\partial \Om$ has zero $\Ha^2$ measure.
But this is a consequence of \cite[Prop.~5.3]{BaHeMoRo_21}

\smallskip

  \emph{Proof of \ref{it:E_structure}).}
We use that $\bar {\E} (\phi, \vec g) = \E_{\vec u}(\phi \, \vec g)$ and that $\bar{\E}_ \vec u(\cdot, \vec g)$ is supported on the atoms of~$\mu_\vec u$.
What we find is that
  \begin{equation} \label{eq:aux_phi_g}
    \E_{\vec u} (\phi \, \vec g) =
    \bar{\E}_\vec u(\phi, \vec g) = \sum_{\vecg \xi\in\At(\mu_\vec u)}
      \bar{\E}_{\vec u} (\{\vecg\xi\}, \vec g) \, \phi(\vecg\xi).
  \end{equation}
Now, given $\vec \xi \in \At(\mu_\vec u)$, we choose a decreasing sequence $\{ \phi_j \}_{j \in \N}$ in $C^1_c (\Om)$ 
such that $\phi_j \to \chi_{\{ \vec \xi \}}$ pointwise.
By dominated convergence and part \emph{\ref{it:Lambda})},
  \begin{equation} \label{eq:aux2_phi_g}
  \begin{split}
 \bar{\E}_{\vec u} (\{\vecg\xi\}, \vec g) & = \lim_{j \to \infty} \bar{\E}_{\vec u} (\phi_j, \vec g) = \lim_{j \to \infty} \int \phi_j (\vec x) \, \vec g(\vec y) \cdot \dd \vecg \Lambda (\vec x, \vec y) \\
 & = \int \chi_{\{\vecg\xi\}}(\vec x) \, \vec g(\vec y) \cdot \dd \vecg \Lambda (\vec x, \vec y) = \left ( \int_{\Gamma_{\vecg\xi}^-} - \int_{\Gamma_{\vecg\xi}^+} \right ) \vec g \cdot \vecg \nu \dd\Ha^2.
  \end{split}
  \end{equation}
As a consequence of \eqref{eq:E=Lambda}, \eqref{eq:aux_phi_g}, \eqref{eq:aux2_phi_g} and 
the definition of $\vec \nu_{\vec \xi}$,
  \begin{equation} \label{eq:Lambda_structure}
  \begin{split}
   \int \vec f \cdot \dd\vecg\Lambda 
   & = \sum_{\vecg \xi\in\At(\mu_\vec u)} \left[ 
   \int_{\Gamma_{\vecg\xi}^-} \vec f (\vecg \xi, \vec y)\cdot \vecg\nu(\vec y)\dd\Ha^2 (\vec y)
    - \int_{\Gamma_{\vecg\xi}^+} \vec f (\vecg \xi, \vec y)\cdot \vecg\nu(\vec y)\dd\Ha^2 (\vec y) \right] \\
    & = \sum_{\vecg \xi\in\At(\mu_\vec u)} \int_{\Gamma_{\vecg\xi}} \vec f (\vecg \xi, \vec y)\cdot \vecg\nu_{\vec \xi} (\vec y) \dd\Ha^2 (\vec y)
  \end{split}
  \end{equation}
for all functions $\vec f$ of separated variables. By the density of the span of the set of functions of separated variables, 
the above holds also for all $\vec f\in C_c(\Om \times \R^3;\R^3)$ and, hence, for all $\vec f$ Borel.
  
  \smallskip
  
  \emph{Proof of \ref{it:mu_atomic}).}
  Equation \eqref{eq:Lambda_structure} may be rewritten as 
\[
\vecg\Lambda = \sum_{\vecg\xi\in\At(\mu_{\vec u})} (\vecg\xi \bowtie \id )_{\#} \bigl( \vecg\nu_{\vec \xi} \Ha^2\res \Gamma_{\vecg\xi} \bigr),
\]
  where $\vecg\xi \bowtie \id $ is the map $(\vecg\xi \bowtie \id )(\vec y)= (\vecg \xi, \vec y)$. 
  
By Lemma \ref{le:var_sum} and \cite[Lemma 1.3]{Ambrosio95} (considering that $\vecg\xi\bowtie \id$ is injective),
we have that 
\begin{align*}
    |\vecg\Lambda| &= \sum_{\vecg\xi\in\At(\mu_{\vec u})} \Big |(\vecg\xi \bowtie \id )_{\#} \bigl( \vecg\nu_{\vec \xi} \Ha^2\res \Gamma_{\vecg\xi} \bigr) \Big |
	  =
	  \sum_{\vecg\xi\in\At(\mu_{\vec u})} (\vecg\xi \bowtie \id )_{\#} \Big | \vecg\nu_{\vec \xi} \Ha^2\res \Gamma_{\vecg\xi}  \Big |
	 \\
	 & = \sum_{\vecg\xi\in\At(\mu_{\vec u})} (\vecg\xi \bowtie \id )_{\#} \Ha^2\res \Gamma_{\vecg\xi}.
\end{align*}

As a consequence of \eqref{eq:E=Lambda} and Riesz' theorem, 
\[
 |\vec \Lambda| (E) = \sup \{ \E_{\vec u} (\vec f) : \, \vec f \in C^1_c(\Om\times \R^3,\R^3) , \, |\vec f| \leq \chi_E \} . 
\]
In particular,  
\begin{equation} \label{le:muu}
 \mu_{\vec u}(E)=|\vec \Lambda| (E\times \R^3) 
	\quad \text{for every Borel set }E\subset \Om.
\end{equation} 
Hence, for any Borel $E\subset\Om$,
  \begin{align*}
   \mu_{\vec u}(E) 
   = \sum_{\vecg\xi\in\At(\mu_{\vec u})}
	\left (\Ha^2 \res \Gamma_{\vecg\xi} \right )
	\big ( \{\vec y : (\vecg \xi, \vec y ) \in E\times \R^3 \} \big ) = \sum_{\vecg\xi \in\At(\mu_{\vec u}) \cap E}  \Ha^2\big (\Gamma_{\vecg\xi}\big ),
  \end{align*}
so \emph{\ref{it:mu_atomic})} is proved.
  
  \smallskip
  
  \emph{Proof of \ref{it:jump_atoms}).}
From \eqref{le:muu}, \emph{\ref{it:Lambda})} and Lemma \ref{le:var_sum}, it follows that for any Borel $E\subset\Om$,
  \begin{align*}
    \mu_{\vec u} (E) 
    &= |\vecg \Lambda| (E\times \R^3)
    \\ &= |\big ( ({\vec u}^{-1})^-\bowtie \id \big ) _{\#}\vecg\nu \Ha^2\res {J_{\vec u^{-1}}}|
	(E\times \R^3)
       + |\big ( ({\vec u}^{-1})^+\bowtie \id \big ) _{\#}\vecg\nu \Ha^2\res {J_{\vec u^{-1}}}| (E\times \R^3)
    \\ &= \Ha^2 \big ( \{ \vec y \in J_{\vec u^{-1}}:
      (\vec u^{-1})^-(\vec y) \in E\}\big )
      + \Ha^2 \big ( \{ \vec y \in J_{\vec u^{-1}}:
      (\vec u^{-1})^+(\vec y) \in E\}\big ).
  \end{align*}
When applied to $E = \Om \setminus \At(\mu_{\vec u})$, it says that $\mu_\vec u\big (\Om\setminus \At(\mu _{\vec u})\big )$ 
is the $\Ha^2$-measure of the set of points in $J_{\vec u^{-1}}$ where at least one of the traces of $\vec u^{-1}$ lies
outside $\At(\mu_{\vec u})$. Because of \emph{\ref{it:mu_atomic})}, that measure is zero.
  
  \smallskip
  
  \emph{Proof of \ref{it:sym-axis}).}
Assume, for a contradiction, that there exists $\vec \xi \in \At (\mu_{\vec u}) \setminus L$.
Then we can find an $r>0$ such that $B(\vec \xi,r) \cap L = \varnothing$ and $\mu_{\vec u}(B( \vec \xi,r))>0$.
By \eqref{le:muu} we also have $|\vec \Lambda |(B(\vec \xi,r)\times\R^3) >0$.
We deduce that there exists $\vec f \in C^1_c(\Om\times \R^3,\R^3)$ with $\supp \vec f \subset B(\vec \xi,r)\times \R^3$ such that
 \[
 \int \vec f \cdot \dd \vec \Lambda \neq 0 .
\]
But since $\supp \vec f \subset B(\vec \xi,r)\times \R^3$ and $B(\vec \xi,r) \cap L = \varnothing$,  we find using \cite[Lemma 3.1]{BaHeMoRo_21},
that
\[
  \int \vec f \cdot \dd \vec \Lambda = \E_\vec u(\vec f) = \E_\vec u(\vec f,\Om \setminus \R\vec e_3)=0.
\]
This contradiction shows that $\At (\mu_{\vec u}) \subset L$.

  \smallskip
  
  \emph{Proof of \ref{it:dipoles}).}
  From \emph{\ref{it:jump_atoms})} it follows that
\begin{equation}\label{eq:Ju-1Gxx}
  J_{\vec u^{-1}} = \bigcup_{\vecg\xi, \vecg\xi' \in \At (\mu_{\vec u})} \Gamma_{\vecg\xi, \vecg \xi'} \qquad \Ha^2 \text{-a.e.}
\end{equation}	
  and the union is clearly disjoint. Moreover, by definition, for all $\vec y\in \Gamma_{\vecg\xi, \vecg\xi'}$ we have $[\vec u^{-1}] (\vec y) = \vecg \xi'-\vecg \xi$.
As $\vec u^{-1} \in SBV (\Om_{\vec b}, \R^3)$ we have, by standard properties of $SBV$ functions (see, e.g., \cite[(3.90) or (4.1)]{AmFuPa00}), that
\[
 D^s \vec u^{-1} = D^j \vec u^{-1} = [\vec u^{-1}] \otimes \vec \nu \, \Ha^2 \res J_{\vec u^{-1}} .
\]
As $\At (\mu_{\vec u})$ is countable, the conclusion follows.

\smallskip

  \emph{Proof of \ref{it:JG}).}
Similarly to \eqref{eq:Ju-1Gxx}, we also have
\begin{equation}\label{eq:Ju-1Gx+-}
  J_{\vec u^{-1}} = \bigcup_{\vecg\xi \in \At (\mu_{\vec u})} \Gamma_{\vecg\xi}^+ \quad \text{and} \quad J_{\vec u^{-1}} = \bigcup_{\vecg\xi \in \At (\mu_{\vec u})} \Gamma_{\vecg\xi}^- \qquad \Ha^2 \text{-a.e.}
\end{equation}
both with disjoint union.
The conclusion then follows.
\end{proof}

\begin{definition}\label{def:Delta_u_x_r}
Let \( \vec u \in \overline{\Ar_s}\) satisfy \( \E(\vec u) <\infty \).
Let  $\vec x\in\Om$ and $r>0$ be such that $B(\vec x,r)\subset\Om$.
We define 
\begin{equation*}
\Delta_{\vec u,\vec x,r} := \deg (\vec u, B(\vec x,r),\cdot) - \chi_{\imG(\vec u,B(\vec  x,r))}.
\end{equation*}
\end{definition}

Here we use the definition of the degree for maps in \(H^1\cap L^\infty\), cf.\ Definition \ref{prop:degree}.

\begin{proposition}\label{prop:proprety_Delta}
Let \( \vec u \in \overline{\Ar_s}\) satisfy \( \E(\vec u) <\infty \), and let $\vec x \in \Om$ and $r>0$ be such that $B(\vec x,r)\subset\Om$.
Then $\Delta_{\vec u,\vec x,r}\in BV(\R^3,\Z)$ and $|D \Delta_{\vec u, \vec x,r}|(\Om)\leq \E(\vec u)$.
Moreover, there exists $\Delta_{\vec u,\vec x} \in BV(\R^3,\Z)$ such that $\Delta_{\vec u,\vec x,r}$ tends weakly$^*$ in $BV(\R^3)$ and in $L^1(\R^3)$ to $\Delta_{\vec u,\vec x}$ as $r \to 0$.
Furthermore, $\Delta_{\vec u,\vec x}$ vanishes outside $B (\vec 0, \| \vec u\|_{L^{\infty} (\R^3, \R^3)})$ and for any $\vec g \in C^1_c (\R^3, \R^3)$,
\begin{equation}\label{eq:derivative_Delta}
 \langle D \Delta_{\vec u,\vec x} , \vec g \rangle = \overline{\E}_{\vec u}( \{ \vec x\}, \vec g) .
\end{equation}
\end{proposition}

\begin{proof}
Set, for simplicity, $B:= B (\vec x, r)$.
Since we assume that $\E(\vec u)<\infty$, we have that $\overline{\E}_{\vec u}(\cdot,\vec g)$ is a measure for all $\vec g\in C^1_c(\R^3,\R^3)$.
By \cite[Lemma 3.3]{HeMo12}, for a.e.\ $r > 0$ we have that 
\begin{align*}
\overline{\E}_{ \vec u}(B,\vec g)& =-\int_{\p B} \vec g(\vec u(\vec x))\cdot \cof D\vec u(\vec x) \, \vec \nu(\vec x) \dd \vec x + \int_{B} \dive \vec g(\vec u(\vec x)) \det D \vec u( \vec x) \dd \vec x \\
&= -\int_{\R^3}\dive \vec g( \vec y) \deg (\vec u, B,\vec y)\dd \vec y+\int_{\R^3}\chi_{\imG(\vec u,B)}\dive \vec g (\vec y) \dd \vec y \\
&= \langle D \left( \deg(\vec u, B,\cdot)-\chi_{\imG(\vec u,B)} \right),\vec g \rangle  \\
&= \langle D \Delta_{\vec u,\vec x,r}, \vec g \rangle. 
\end{align*}
In the second equality we used the integral formula for the degree (Definition \ref{prop:degree}) and the area formula (Proposition \ref{prop:area-formula}).
Thus,
\begin{equation}\label{eq:derivative_Delta_and_surface_energy}
\overline{\E}_{ \vec u}(B,\vec g)=\langle D \Delta_{\vec u,\vec x,r},\vec g \rangle \qquad \forall \vec g \in C^1_c(\R^3,\R^3)
\end{equation}
and, hence,
\begin{equation*}
 \left\| D \Delta_{\vec u,\vec x,r} \right\| = 
 \sup_{\substack{\vec g \in C^1_c (\R^3, \R^3) \\ \|\vec g\|_{L^\infty}\leq 1}} \langle D \Delta_{\vec u,\vec x,r},\vec g \rangle = \sup_{\substack{\vec g \in C^1_c (\R^3, \R^3) \\ \|\vec g\|_{L^\infty}\leq 1}} \overline{\E}_{\vec u}(B,\vec g) \leq \E(\vec u).
\end{equation*}
But since $\Delta_{\vec u,\vec x,r}$ vanishes outside $B(\vec 0,\|\vec u\|_{L^\infty})$ we have from the Poin\-car\'e inequality that $\Delta_{\vec u,\vec x,r}\in BV(\R^3,\Z)$ and that $\|\Delta_{\vec u,\vec x,r}\|_{BV(\R^3)}\leq M$ for some constant $M$ not depending on $\vec x$ or $r$.
From the compactness theorem in $BV$ there exists a function $\Delta_{\vec u,\vec x}\in BV(\R^3)$ and a subsequence $r_n\rightarrow 0$ such that $\Delta_{\vec u,\vec x,r_n}\weakcs \Delta_{\vec u,x}$ in $BV (\R^3)$ and $\Delta_{\vec u,\vec x,r_n}\rightarrow \Delta_{\vec u,\vec x}$ in $L^1(\R^3)$. Up to a further subsequence we can also assume the convergence a.e.\ and hence $\Delta_{\vec u,\vec x}\in BV(\R^3,\Z)$.
But by \eqref{eq:derivative_Delta_and_surface_energy} and dominated convergence, we have
\begin{equation*}
\langle D  \Delta_{\vec u,\vec x},\vec g \rangle =\lim_{r_n\rightarrow 0}\overline{\E}_{\vec u}(B(\vec x,r_n),\vec g) = \overline{\E}_{\vec u}( \{ \vec x\}, \vec g), \  \ \forall \vec g \in C^1_c(\R^3,\R^3).
\end{equation*}
As the right-hand side does not depend on the sequence $r_n\rightarrow 0$, the derivative $D \Delta_{\vec u,\vec x}$ does not depend on the subsequence either.
Since $\Delta_{\vec u,\vec x}$ vanishes at infinity we have that the limit $\Delta_{\vec u,\vec x}$ does not depend on the sequence $r_n\rightarrow 0$. Thus $\Delta_{\vec u,\vec x,r}$ tends weakly$^*$ in $BV(\R^3)$ and in $L^1(\R^3)$ to $\Delta_{\vec u,\vec x}$ as $r \to 0$.
\end{proof}

The previous proposition allows us to set the following definition.

\begin{definition}
 Let \(\vec u \in \overline{\Ar_s}\) be with \( \E(\vec u)<\infty\).
For every \( \vec \xi \in \Om\) we define \( \Delta_{\vec u, \vecg \xi} \in BV(\R^3,\Z)\) as the $L^1 (\R^3)$ 
limit of $\Delta_{\vec u, \vecg \xi, r}$ as $r \to 0$, where \(\Delta_{\vec u, \vecg \xi,r} \) is as in Definition \ref{def:Delta_u_x_r}.
\end{definition}

\begin{proposition} \label{prop:DeltaNonZero}
  Let \(\vec u \in \overline{\Ar_s}\) be with \( \E(\vec u)<\infty\). Then, $\Delta_{\vec u,\vecg\xi}$ is not identically zero if and only if $\vecg \xi \in\At(\mu_{\vec u})$.
\end{proposition}

\begin{proof}
As $\Delta_{\vec u,\vecg\xi}$ vanishes outside $B (\vec 0, \| \vec u\|_{L^{\infty} (\R^3, \R^3)})$, we have that $\Delta_{\vec u,\vecg\xi} = 0$ if and only if $D \Delta_{\vec u,\vecg\xi}=0$.
Together with \eqref{eq:derivative_Delta}, we obtain that
\[
 \Delta_{\vec u,\vecg\xi} = 0 \quad \text{if and only if} \quad \overline{\E}_\vec u(\{\vec\xi\},\vec g) = 0 \text{ for all } \vec g\in C^1_c(\R^3,\R^3) .
\] 

Assume first $\vec \xi \in \Om \setminus \At(\mu_{\vec u})$.
Since $\At (\overline{\E}_\vec u(\cdot,\vec g)) \subset \At(\mu_ \vec u)$ for every $\vec g\in C^1_c(\R^3,\R^3)$ we have that $\vec \xi$ is not an atom of $\overline{\E}_\vec u(\cdot,\vec g)$, so $\overline{\E}_\vec u(\{\vec \xi \},\vec g)=0$, and hence, $\Delta_{\vec u,\vecg\xi} = 0$.

Conversely, assume that $\overline{\E}_\vec u(\{\vec \xi \},\vec g)=0$ for all $\vec g\in C^1_c(\R^3,\R^3)$, and fix any such $\vec g$.
We can write 
\[
\overline{\E}_\vec u(\{ \vec \xi \},\vec g) = \lim_{j \to \infty} \overline{\E}_\vec u(\phi_j \, \vec g) ,
\]
where $\{ \phi_j \}_{j \in \N} \subset C^1_c (\Om)$ satisfies $\phi_j \rightarrow \chi_{\{\vec \xi\}}$ pointwise 
and \(\|\phi_j\|_\infty \leq 1\). Thus we find
\[
\overline{\E}_\vec u(\{\vec \xi \},\vec g) = \lim_{j \to \infty} \int_{\Om \times \R^3} \phi_j (\vec x) \, \vec g( \vec y) \cdot \dd \vec \Lambda(\vec x,\vec y) = \int_{\Om \times \R^3} \chi_{\{ \vec \xi \} }(\vec x) \, \vec g( \vec y) \cdot \dd \vec \Lambda(\vec x,\vec y) .
\]
Let $\lambda_{ \vec \xi}$ be the $\R^3$-measure in $\R^3$ defined by 
$$\lambda_{\vec \xi}(W)= \vec \Lambda (\{ \vec \xi\}\times W) \quad \text{ for any Borel } W \subset \R^3.$$
With this definition we have
\[
\overline{\E}_ \vec u(\{ \vec \xi \},\vec g) = \int_{\Om \times \R^3} \chi_{\{ \vec \xi  \}}(\vec x) \, \vec g(\vec y) \cdot \dd \vec \Lambda (\vec x, \vec y) = \int_{\R^3} \vec g(\vec y) \cdot \dd \lambda_{ \vec \xi} (\vec y).
\]
As $\overline{\E}_\vec u(\{\vec \xi \},\vec g)=0$, we have that $\int_{\R^3} \vec g \cdot\dd \lambda_{\vec \xi }=0$.
Since this is true for all $\vec g \in C^1_c (\R^3, \R^3)$, we obtain that $\lambda_{\vec \xi}=0$.
But by \eqref{le:muu}, 
\begin{equation*}
\mu_\vec u(\{ \vec \xi\})=|\vec \Lambda| (\{\vec \xi \}\times \R^3)=|\lambda_{\vec \xi}|(\R^3) = 0 .
\end{equation*}
This means that $\vec \xi$ is not an atom of $\mu_\vec u$.
\end{proof}

The following result is the closest we are to state that the surface created at each cavitation point (each atom of $\mu_{\vec u}$, or each Dirac mass of the distributional Jacobian in the classical cavitation problem)
 actually encloses a volume.
In classical cavitation (cf.\ \cite[Th.\ 4.8]{HeMo12}) 
this is 
$$\{\vec y \in \Gamma(\vec u): ({\vec u}^{-1})^{\pm} (\vec y) = \vecg\xi\} = \partial^* \imT(\vec u, \vecg\xi)\quad \Ha^2\text{-a.e.}
$$
In the present setting, it is the integer-valued function $\Delta_{\vec u,\vecg\xi}$ which can be thought of as the degree of \(\vec u\) with respect to the surface created from $\vecg\xi$.
The very possibility of defining a topological degree with respect to that surface suggests that it must be a closed surface
(a manifold without boundary),
at least in some weak sense.
That is the content of item \ref{item:surfaceI}) of the next proposition, which shows that $\Gamma_{\vec \xi}$ is, if not the boundary of a volume, at least the union of reduced boundaries of finite perimeter sets. It also goes a little further in the topological  description, 
stating that 
$\Delta_{\vec u, \vecg\xi}$ always jumps just by one and that it decreases in the direction towards the part of the body coming from $\vecg\xi$.
 Item \ref{item:surfaceII}) means that the surface of the generalized cavity created at \(\vec \xi\) can be approached by the image of small spheres surrounding \(\vec \xi\). 
Item \ref{item:surfaceIII}) suggests that the Dirichlet energy in a ball of radius $r$ centred at the singular points scales like $r^1$ and not as the volume $O(r^3)$ of the ball. The significance of that energy concentration is that it might be part of a future regularity argument yielding a competitor test function with less energy than an alleged minimiser producing harmonic dipoles. 

\begin{proposition}\label{prop:surface}
 Let \( \vec u \in \overline{\Ar_s}\) be with \( \E(\vec u)<\infty\) and let $\vec \xi \in \At (\mu_{\vec u})$.
Then
\begin{enumerate}[i)]
 \item\label{item:surfaceI} $\displaystyle D\Delta_{\vec u, \vecg\xi} = - \vecg\nu_{\vecg\xi} \Ha^2 \res \Gamma_{\vecg\xi}$ and
\[
 \Gamma_{\vecg\xi} = \sum_{k\in \N} \partial^* \{\vec y\in \R^3 : \Delta_{\vec u,\vecg\xi}(\vec y)= k\} \quad \Ha^2 \text{-a.e.}
\]
  
  \item\label{item:surfaceII} $\displaystyle \tilde{\vecg\nu}_{\vecg\xi, r} \Ha^2 \res \imG(\vec u,\partial B(\vecg\xi,r))
         \weakcs \vecg\nu_{\vecg\xi} \Ha^2 \res {\Gamma_{\vecg\xi}}$ as $r \to 0$, where 
	$\tilde{\vec\nu}_{\vec\xi, r}\big ( \vec u(\vec x)\big)$ is the unit normal to $\imG(\vec u, \partial B(\vec\xi, r))$ defined in Proposition \ref{pr:change_of_variables-surfaces}
and the sequence $r \to 0$ avoids a set of measure zero.
  
  \item\label{item:surfaceIII} $\displaystyle \liminf_{r\to 0} \int_{\partial B(\vecg\xi, r)} \frac{|D\vec u|^2}{2} \dd\Ha^2 \geq \Ha^2 (\Gamma_{\vecg\xi})$.
\end{enumerate}
\end{proposition}
   
\begin{proof}
We prove each item separately.

\smallskip

\emph{Proof of \ref{item:surfaceI})}.
We know that for any $\{ \phi_j \}_{j \in \N} \subset C^1_c (\Om)$ with $\phi_j \rightarrow \chi_{\{\vec \xi\}}$ 
pointwise and \(\|\phi_j\|_\infty \leq 1\) we have, for any $\vec g \in C^1_c (\R^3, \R^3)$,
\begin{equation*}
 \overline{\E}_\vec u(\{\vec \xi \},\vec g) = \lim_{j \to \infty} \E_{\vec u} (\phi_j \, \vec g).
\end{equation*}
We use the description of the surface energy from \cite[Th.\ 2\,iv)]{HeMo11} to obtain
\begin{align*}
\overline{\E}_\vec u(\{\vec \xi \}, \vec g)&= \lim_{j \to \infty} \int_{J_{\vec u^{-1}}} [\phi_j (\vec u^{-1}(\vec y))] \, \vec g(\vec y)\cdot \vec \nu(\vec y) \dd \Ha^2(\vec y) \\
&= \int_{J_{\vec u^{-1}}} \Bigl( \chi_{\{\vec \xi\}} \bigl( (\vec u^{-1} )^+ (\vec y) \bigr) - \chi_{\{\vec \xi\}}
\bigl( (\vec u^{-1})^- (\vec y) \bigr) \Bigr) \, \vec g(\vec y)\cdot \vec \nu(\vec y) \dd\Ha^2(\vec y) \\
&= \int_{\Gamma_{\vec \xi}^+} \vec g(\vec y)\cdot \vec \nu(\vec y) \dd\Ha^2(\vec y) -\int_{\Gamma_{\vec \xi}^-} \vec g(\vec y)\cdot \vec \nu(\vec y) \dd\Ha^2(\vec y) 
= - \int_{\Gamma_{\vec \xi}}\vec g(\vec y) \, \vec \nu_{\vec \xi}(\vec y)\dd\Ha^2(\vec y) ,
\end{align*}
due to the definition of $\vec \nu_{\vec \xi}$.
Thanks to \eqref{eq:derivative_Delta} we obtain that
\[
 \langle D \Delta_{\vec u,\vec x} , \vec g \rangle = - \langle \vec \nu_{\vec \xi} \,\Ha^2 \res \Gamma_{\vec \xi} , \vec g \rangle ,
\]
whence the conclusion follows.

\smallskip

\emph{Proof of \ref{item:surfaceII})}.
From Proposition \ref{prop:proprety_Delta} we have that $D \Delta_{\vec u,\vec \xi,r}  \overset{\ast}\rightharpoonup D \Delta_{\vec u,\vec \xi}$ in $\M (\R^3)$.
So from point \emph{\ref{item:surfaceI})} it suffices to prove that 
\begin{equation*}
D \Delta_{\vec u,\vec \xi,r} = - \tilde{\vec \nu}_{\vec \xi,r} \Ha^2 \res \imG(\vec u,\p B(\vec \xi,r)).
\end{equation*}
The last equality holds because, from \eqref{eq:derivative_Delta_and_surface_energy}, \cite[Lemma 3.3]{HeMo12} and the change of variables formula for surfaces (Proposition \ref{pr:change_of_variables-surfaces}), for any $\vec g\in C^1_c(\R^3,\R^3)$ and a.e.\ $r > 0$,
\begin{align*}
\langle D \Delta_{\vec u,\vec \xi,r},\vec g \rangle &= \overline{\E}_\vec u(B(\vec \xi,r),\vec g) \\
&=-\int_{\p B(\vec \xi,r)} \vec g(\vec u(\vec x)) \cdot \cof D\vec u(\vec x) \, \vec \nu(\vec x) \dd \Ha^2 (\vec x) + \int_{B(\vec \xi,r)}\dive \vec g(\vec u(\vec x))\det D\vec u(\vec x) \dd \vec x \\
&= -\langle \tilde{\vec \nu}_{\vec \xi,r} \Ha^2 \res \imG(\vec u,\p B(\vec \xi,r)),\vec g \rangle +\int_{B(\vec \xi,r)}\dive \vec g(\vec u(\vec x))\det D\vec u(\vec x) \dd \vec x
\end{align*}
and 
\begin{equation*}
\int_{B(\vec \xi,r)}\dive \vec g(\vec u(\vec x))\det D\vec u(\vec x) \dd \vec x \rightarrow 0 \text{ as } r \rightarrow 0.
\end{equation*}

\smallskip

\emph{Proof of \ref{item:surfaceIII})}.
By \emph{\ref{item:surfaceII})} and the semicontinuity of the total variation of a measure with respect to the weak$^*$ convergence, we obtain that
\[
 \Ha^2 (\Gamma_{\vec \xi}) \leq \liminf_{r \to 0} \Ha^2 \big (\imG(\vec u,\p B (\xi, r))\big ) .
\]
Now, for each $r>0$ thanks to the change of variable formula for surfaces (Proposition \ref{pr:change_of_variables-surfaces}), 
we find that 
\begin{align*}
 \Ha^2 \big (\imG(\vec u,\p B (\vec \xi, r))\big ) = \int_{\p B(\vec \xi,r)} \left|(\cof D\vec u) \vec \nu \right| \dd \Ha^2 \leq \int_{\p B(\vec \xi,r)} \frac{|D\vec u|^2}{2} \dd \Ha^2.
\end{align*}
The last inequality can be obtained by representing the linear transformation $\vec F=D\vec u(\vec x)$
in an orthonormal basis $\vec e_1'$, $\vec e_2'$, $\vec e_3'$ with $\vec e_3'=\vec\nu(\vec x)=\vec e_1'\wedge \vec e_2'$. Indeed, in that basis the columns of $\vec F$ are the coordinates of $\vec F\vec e_1'$, $\vec F\vec e_2'$ and $\vec F \vec\nu$, so that 
$$
	|(\cof \vec F)\vec \nu| = | (\vec F\vec e_1')\wedge (\vec F \vec e_2')|\leq |\vec F\vec e_1'||\vec F\vec e_2'|\leq \frac{|\vec F\vec e_1'|^2 +|\vec F\vec e_2'|^2}{2}\leq \frac{|\vec F|^2}{2},
$$  
where we recall that we are using the Frobenius norm for matrices.
The conclusion now follows.
\end{proof}

The following proposition shows that there must be a cancellation of the degrees; in particular, some must be negative unless $\Delta_{\vec u,\vecg\xi}=0$ for all $\xi \in\At(\mu_{\vec u})$. 
\begin{proposition} \label{prop:sumDegrees}
Let \(\vec u \in \overline{\Ar_s}\) be with \( \E(\vec u)<\infty\).
Then
    \[
	\sum_{\vecg\xi\in\At(\mu_{\vec u})} \Delta_{\vec u,\vecg\xi} = 0.
    \]
\end{proposition}

\begin{proof}
For each $\vec \xi \in \At (\mu_{\vec u})$, by Proposition \ref{prop:surface}.\emph{\ref{item:surfaceI})} 
and the definition of $\vec \nu_{\vec \xi}$,
\begin{equation}\label{eq:DDelta}
 D \Delta_{\vec u,\vec \xi} = - \vec \nu_{\vec \xi} \Ha^2 \res \Gamma_\vec \xi = \vec \nu \Ha^2 \res \Gamma_{\vec \xi}^+ - \vec \nu \Ha^2 \res \Gamma_{\vec \xi}^-.
\end{equation}
The series
\[
 \sum_{\vecg\xi\in\At(\mu_{\vec u})} D \Delta_{\vec u,\vecg\xi}
\]
is absolutely convergent in $\M (\R^3, \R^3)$ since, by \eqref{eq:DDelta} and Proposition \ref{pr:muproperties}.\emph{\ref{it:JG})},
\[
 \sum_{\vecg\xi\in\At(\mu_{\vec u})} \left\| D \Delta_{\vec u,\vecg\xi} \right\| = \sum_{\vecg\xi\in\At(\mu_{\vec u})} \left( \Ha^2 (\Gamma_{\vec \xi}^+) + \Ha^2 (\Gamma_{\vec \xi}^-) \right) = 2 \Ha^2 (J_{\vec u^{-1}}) .
\]
As each $\Delta_{\vec u,\vec \xi}$ vanishes outside $B (\vec 0, \| \vec u \|_{L^{\infty} (\R^3, \R^3)})$, the series
\[
 \sum_{\vecg\xi\in\At(\mu_{\vec u})} \Delta_{\vec u,\vecg\xi}
\]
is absolutely convergent in $BV (\R^3)$ and, hence, defines a $BV$ function.
Now, by \eqref{eq:DDelta} and \eqref{eq:Ju-1Gx+-},
\begin{align*}
D \left( \sum_{\vec \xi \in \At(\mu_\vec u)} \! \Delta_{\vec u,\vec \xi} \right) & = \! \sum_{\vec \xi \in \At(\mu_\vec u)} \! D \Delta_{\vec u,\vec \xi} = \vec \nu \left( \sum_{\vecg\xi\in\At(\mu_{\vec u})} \! \Ha^2 \res \Gamma_{\vec \xi}^+ - \sum_{\vecg\xi\in\At(\mu_{\vec u})} \!\Ha^2 \res \Gamma_{\vec \xi}^- \right) \\
& = \vec \nu \left( \Ha^2 \res J_{\vec u^{-1}} -  \Ha^2 \res J_{\vec u^{-1}} \right) = 0 .
\end{align*}
As $\sum_{\vec \xi} \Delta_{\vec u,\vec \xi}$ vanishes outside $B (\vec 0, \| \vec u \|_{L^{\infty} (\R^3, \R^3)})$, we obtain the conclusion.
\end{proof}

\section{Minimal connection length for the elastic harmonic dipoles}
It has already been observed that elasticity problems can be described through the theory of (Cartesian) currents (see, e.g., \cite{GiMoSo98I,GiMoSo98II,CoDeLe03,Mucci05,Mucci10a,Mucci10b}). We now recall some definitions and properties of currents. In this section, instead of $\R^3$ we work in $\R^n$ with $n\geq 1$ an integer.

\begin{definition}
Let $U\subset \R^n$ be a an open set and let $k\in \N$. A $k$-dimensional \textit{current} in $U$ is a linear form on the set $\mathcal{D}^k(U)$ of $C^\infty$ $k$-differential forms with compact support in $U$.
The set of those currents is denoted by $\mathcal{D}_k(U)$. \\
The \textit{boundary} $\p T$ of a $k$-dimensional current $T$ is the $(k-1)$-dimensional current defined by
\[ \langle \p T,\omega \rangle = \langle T, d \omega \rangle \text{ for every } \omega \in \mathcal{D}^{k-1}(U) \]
and the boundary of a $0$-dimensional current is set equal to $0$. \\
The \textit{mass} of a current $T$ is
\[ \mathbb{M}(T):=\sup\{ \langle T, \omega \rangle : \omega \in \mathcal{D}^k(U), |\omega| \leq 1 \}.\]
\end{definition}

\begin{definition}
\begin{enumerate}[a)]
\item A current $T$ is said \textit{normal} if $T$ and $\p T$ have finite mass.

\item A current is \textit{rectifiable} if it can be written as
\[ \langle T,\omega \rangle =\int_{\Rcal} \langle \omega(\vec x),\vec \tau(\vec x) \rangle \, m(\vec x)\dd\Ha^k(\vec x) \]
where
\begin{enumerate}[i)]
\item $\Rcal$ is a $k$-rectifiable set
\item $\vec \tau$ is a unit $k$-dimensional vector which spans $\operatorname{Tan}(\Rcal, \vec x)$ for $\Ha^{k}$-a.e.\ $\vec x\in \Rcal$; such a $\vec \tau$ is called an orientation.
\item $m$ is a real function, called the multiplicity, and which satisfies $\int_{\Rcal} |m |\dd\Ha^{k} <\infty$.
\end{enumerate}
If $T$ is a rectifiable current we write $T=\llbracket \Rcal,m,\vec \tau \rrbracket$.

\item A current is an \textit{integer multiplicity rectifiable current} if it is a rectifiable current such that the multiplicity $m$ takes integer values.
\end{enumerate}
\end{definition}

A particular case of an integer multiplicity rectifiable current is the one given by the integration on the graph of a function. Let $\Om \subset \R^n$ be a smooth bounded domain. We denote by $\A^1(\Om,\R^n)$ the class of vector-valued maps $\vec u:\Om\rightarrow \R^n$ that are a.e.\ approximately differentiable and such that all the minors of the Jacobian matrix $D\vec u$ are integrable.
When the domain and the space dimension are clear from the context, we will use the shorter notation $\A^1$.
For $\vec u \in \A^1$ we let
\[ M(D\vec u)(\vec x)=(\vec e_1,D\vec u(\vec x) \, \vec e_1)\wedge \cdots \wedge (\vec e_n,D\vec u(\vec x) \, \vec e_n) \]
where $\{\vec e_i\}_{i=1,\ldots,n}$ is the standard basis of $\R^n$, and here $\wedge$ denotes the exterior product (for the definition and properties we refer, e.g., to \cite[Sect.\ 2.1]{GiMoSo98I}). We also let
\[ \mathcal{G}_\vec u=\{(\vec x,\vec u(\vec x)) : \vec x\in \A_\vec u \} \]
where
\[\A_\vec u=\{\vec x \in \Om : \vec u \text{ is approximately differentiable at } \vec x \}. \]
For $\vec u \in \A^1$ we can define a current $G_\vec u \in \mathcal{D}_n(\R^n\times\R^n)$ by
\[ \langle G_\vec u, \omega \rangle = \int_{\R^n \times \R^n} \langle \xi, \omega \rangle \dd\Ha^n \res{\mathcal{G}_\vec u}\]
with $\xi=\frac{M(D \vec u)(\vec x)}{|M (D \vec u) (\vec x)|}$. We can show that $\mathcal{G}_\vec u$ is a countably rectifiable set and, so, $G_\vec u$ is an integer multiplicity rectifiable current.
The mass of this current is equal to
\[\mathbb{M}(G_\vec u)=\Ha^n(\mathcal{G}_\vec u)=\int_\Om|M (D \vec u)|\dd \vec x.\]
\noindent
If $\vec u\in W^{1,n-1}(\Om,\R^n)$ with $\det D\vec u\in L^1(\Om)$, we can see that $\vec u\in \A^1$.

As in nonlinear elasticity, in the theory of currents it is essential to distinguish the reference and deformed configurations.
It is customary to denote by $\R^n_{\vec x}$ the space where the reference configuration lies, which has coordinates $\vec x$, while $\R^n_{\vec y}$ is the space of the deformed configuration with coordinates $\vec y$.

We now introduce the concept of stratification of differential forms and of currents.
\begin{definition}
Let $\omega$ be an $n$-differential form on $\R^n_{\vec x} \times \R^n_{\vec y}$.
We can write \[\omega= \sum\limits_{\substack{\alpha, \beta  \\ |\alpha|+|\beta|=n}} f_{\alpha,\beta}\dd x_\alpha\wedge \dd y_\beta\] with $\alpha$ and $\beta$ some multi-indices. For every integer $h$ we then define
\[ \omega^{(h)}=\sum\limits_{\substack{|\alpha|+|\beta|=n \\|\beta|=h}} f_{\alpha,\beta} \dd x_\alpha \wedge \dd y_\beta. \]
Given a current $T$ on $\R^n_{\vec x} \times \R^n_{\vec y}$, we define its $h$-stratum $(T)_h$ by
\[ \langle (T)_h, \omega \rangle=\langle T, \omega^{(h)} \rangle. \]
\end{definition}

We can now make a link between the surface energy $\E$ and the theory of currents.

\begin{proposition}
Let $\vec u\in W^{1,n-1}\cap L^\infty(\Om,\R^n)$ be such that $\det D\vec u \in L^1(\Om)$ and let $\vec f \in C^1_c(\Om\times\R^n,\R^n)$.
We define the $n$-differential form $\omega_\vec f$ by
\begin{equation}
 \omega_\vec f(\vec x,\vec y)=\sum_{j=1}^n(-1)^{j-1}f^j(x,y)\widehat{\dd y^j} 
\end{equation}
where $\widehat{\dd y^j}=\dd y^1\wedge \cdots \wedge \dd y^{j-1}\wedge \dd y^{j+1}\wedge \cdots \wedge \dd y^n.$
Then
\[\E_\vec u(\vec f)=\langle \p G_\vec u, \omega_\vec f \rangle \quad \text{ and } \quad
 E(\vec u)=\mathbb{M}(\p G_\vec u). \]
\end{proposition}

\begin{proof}
The equality $\E_\vec u(\vec f)=\langle \p G_\vec u, \omega_\vec f \rangle$ is a consequence of the definitions, and has been observed in other places \cite[Sects.\ 4 and 7]{HeMo10}.
Since $\omega_\vec f$ is an $(n-1)$-vertical form (i.e., it is the $(n-1)$ stratum of itself), that equality implies that 
\[\E(\vec u)=\mathbb{M}((\p G_\vec u)_{n-1}).\]
Finally, we recall that if $\vec u\in W^{1,p}$ then $(\p G_\vec u)_h=0$ for all $h\leq p-1$ (this can be shown by approximation by smooth functions see, e.g., \cite[Prop.\ 3 and Rk.\ 3 p.\ 246]{GiMoSo98I}).
Therefore, $(\p G_\vec u)_{n-1} = \p G_\vec u$, so, in particular, $\mathbb{M}((\p G_\vec u)_{n-1}) = \mathbb{M}(\p G_\vec u)$.
\end{proof}
In particular if $\vec u\in \Ar_s$, where \(\Ar_s\) is defined in \eqref{eq:defAs}, then $\p G_\vec u=0$.

\begin{definition}
Let $\Om\subset \R^n$ be a smooth open bounded in $\R^n$. We say that $T$ is a Cartesian current in $\Om\times \R^n$ if
\begin{enumerate}[i)]
\item $T$ is an integer multiplicity rectifiable current $T=\llbracket \Rcal,m, \vec \tau \rrbracket$;
\item $\mathbb{M}(T) <\infty$ and
\[ \| T\|_1:=\sup\{ \langle T,|y|\varphi(x,y)\dd \vec x \rangle : \varphi \in C^1_c(\Om \times\R^n), |\varphi|\leq 1\} <\infty ; \]
\item $T \res {\dd x_1\wedge \cdots \wedge \dd x_n}$ is a positive Radon measure in $\Om\times \R^n$ and $\pi_\# T=\llbracket \Om\rrbracket $, with $\pi:\R^n_\vec x\times\R^n_\vec y \rightarrow \R^n_\vec x$ given by $\pi (\vec x,\vec y) = \vec x$;
\item $\p T \res {\Om\times \R^n}=0$.
\end{enumerate}
\end{definition}

We remark that if $T=G_\vec u$ for some $\vec u \in \A^1$ then $\|T\|_1=\|\vec u\|_{L^1}$.

We also define the support of a current.

\begin{definition}
Let $T$ be a $k$-dimensional current in $\Om\times \R^n$. The support of $T$ is the smallest closed set $F$ such that
$\langle T, \omega \rangle=0$ if the support of $\omega$ is contained in the complement of $F$.
In other words,
\begin{multline*}
 \supp T:= \bigcap \{ K\subset \Om\times \R^3: K \text{ is relatively closed in } \Om\times \R^n \text{ and} \\
 \langle T,\omega\rangle =0 \text{ for all } \omega \in \mathcal{D}^k(\Om\times \R^n) \text{ such that } \supp \omega \subset (\Om\times \R^n)\setminus K \} .
\end{multline*}
\end{definition}

The following result is taken from \cite[Props.\ 4.1 and 4.4]{HeRo18}.

\begin{proposition}\label{prop:previous_article}
Let $\Om \subset\R^3$ be an axisymmetric domain. Let $(\vec u_n)_n \subset \Ar_s$ be such that $\vec u_n \rightharpoonup \vec u$ in $H^1$.
Then, passing to a subsequence, $G_{\vec u_n}\rightharpoonup G_\vec u+S$ for a current $S$ with $S^{(1)}=S^{(3)}=0$ and $\supp S \subset L \times \R^3$.
\end{proposition}

Since all the $\vec u_n$ are sufficiently regular (they belong to $\Ar_s$), $\partial G_{\vec u_n}=0$ for all $n$ and this property is inherited by the limit current $G_{\vec u}+ S$. Hence, $\partial S =-\partial G_{\vec u}$. In the map by Conti \& De Lellis this gives 
$$\partial S= \{\vec 0'\}\times \Gamma - \{\vec 0\}\times \Gamma,
$$
using the same notation as in the beggining of Section \ref{sec:VI},
  namely,
$\vec 0$ and $\vec 0'$ are the cavitation points in the reference configuration and $\Gamma$ is the bubble created by $\vec u$, as seen in the deformed configuration.
As mentioned in \cite[Sect.\ 7]{CoDeLe03}, since the ``hole'' opened at $\vec 0$ has the ``wrong'' sign, the defect current $S$ must be a cylinder connecting $\{\vec 0\}\times \Gamma$ with $\{\vec 0'\}\times \Gamma$, as in the dipoles in harmonic map theory. 

We are now in position to prove the result of this section: that our candidate \eqref{eq:relaxedF} for the relaxed energy (cf.~Theorem \ref{th:previous_article}) can be expressed in the language of Cartesian currents by
\begin{equation*}
F(\vec u)= \int_\Om |D\vec u|^2+2\mathbb{M}(S),
\end{equation*}
with \(S\) defined in Proposition \ref{prop:previous_article}.
Note that even though the defect current $S$ might depend on the chosen subsequence, the mass $\mathbb{M}(S)$ does not, since, as shown in Proposition \ref{prop:relaxed_energy_current} below, $\mathbb{M}(S)$ admits two expressions in terms of quantities depending only on $\vec u$.

\begin{lemma}
	\label{le:defectCurrentXalpha}
Let \( \vec u\in \overline{\Ar_s}\), and let \((\vec u_k)\subset \Ar_s \) be such that \(\vec u_k \rightharpoonup \vec u\) in \(H^1\). We have, up to a subsequence, that \(G_{\vec u_k} \rightharpoonup G_\vec u +S\) from Proposition \ref{prop:previous_article}.
	Then, for any $i\in\{1,2,3\}$, any $\eta_1,\eta_2\in C^1_c(\Om)$, and any $g \in C^1_c(\R^3)$,
$$
	\langle S, \eta_1(\vec x) g(\vec y)\, \dd x_1\wedge \widehat {\dd y_i} \rangle 
= \langle S, \eta_2(\vec x) g(\vec y)\, \dd x_2\wedge \widehat {\dd y_i} \rangle =0.
$$
\end{lemma}

\begin{proof}
 Since $S$ is a $3$-rectifiable current, $S=\llbracket \Rcal, m, \vec \tau\rrbracket$
for some $3$-rectifiable set $\Rcal$, some integer-valued multiplicity $m$ and some unit tangent $3$-vector $\vec \tau$. 
Since $\supp S \subset L\times \R^3$, it may be assumed, without loss of generality, that for $\Ha^3$-a.e.\ $(\vec x, \vec y)\in \Rcal$,
$$\operatorname{Tan} (\Rcal, (\vec x, \vec y)) \subset \operatorname{Span}(\{\vec e_3, \overline {\vec e}_1, \overline{\vec e}_2, \overline{\vec e}_3\}),
$$
where $\vec e_3$ is the direction parallel to the symmetry axis in the reference configuration and $\overline{\vec e}_1$, $\overline{\vec e}_2$, $\overline{\vec e}_3$ is the canonical basis for the target ambient space $\R^3$. 
Therefore, for $\Ha^3$-a.e.\ $(\vec x, \vec y)\in \Rcal$, the unit $3$-vector $\vec \tau(\vec x, \vec y)$ is a linear combination 
of the $3$-vectors 
$$
	\vec e_3 \wedge \overline{\vec e}_1 \wedge \overline{\vec e}_2,
	\quad  \vec e_3 \wedge \overline{\vec e}_2 \wedge \overline{\vec e}_3,
	\quad \vec e_3 \wedge \overline{\vec e}_3 \wedge \overline{\vec e}_1
	\quad \text{and}\quad
	\overline{\vec e}_1\wedge \overline{\vec e}_2 \wedge \overline{\vec e}_3.
$$
The four of them are orthogonal to $\dd x_1 \wedge \widehat{\dd y_i}$ and to $\dd x_2\wedge \widehat{\dd y_i}$, hence 
$$
	\langle S, \eta_\alpha g \dd x_\alpha\wedge \widehat {\dd y_i}\rangle 
	= \int_{\Rcal} \langle \eta_\alpha g \dd x_\alpha\wedge \widehat {\dd y_i}, \vec \tau (\vec x, \vec y)\rangle \,m(\vec x, \vec y)\,\dd\Ha^3(\vec x, \vec y) =0. 
$$
\end{proof}

\begin{proposition}\label{prop:relaxed_energy_current}
Let \( \vec u\in \overline{\Ar_s}\), let \((\vec u_k)\subset \Ar_s \) such that \(\vec u_k \rightharpoonup \vec u\) in \(H^1\). We have, up to a subsequence, that \(G_{\vec u_k} \rightharpoonup G_\vec u +S\) from Proposition \ref{prop:previous_article}.
Then 
\begin{align*}
\|D^s \vec u^{-1}\|&= \mathbb{M}(S) 
\\ & = \sup \Bigl\{ \E_{\vec u}(\phi \, \vec g):\,\phi\in C^1_c(\Om),\, \vec g \in C^1_c(\R^3, \R^3),\, \|\nabla \phi|_{\overline\Om}\|_{\infty}\leq 1, \|\vec g\|_{\infty}\leq 1\Bigr\} .
\end{align*}
\end{proposition}

\begin{proof}
Let \(\phi \in C^1_c(\Om) \), \(\vec g \in C^1_c(\R^3,\R^3)\), and \(\vec f(\vec x, \vec y)=\phi(\vec x) \, \vec g(\vec y)\). From 
\cite[Lemma 5.2]{BaHeMoRo_21} we have that
\begin{equation}
	\label{eq:linkBV}
	\begin{split}
\E_{\vec u }(\phi\,\vec g) &= 
- \langle D^s (\phi \circ \vec u^{-1}),\vec g \rangle  \\
& =
 - \int_{\Om_{\vec b}} \nabla \phi(\vec u^{-1}(\vec y))\otimes \vec g(\vec y)\cdot \dd D^c \vec u^{-1}(\vec y)
 - \int_{J_{\vec u^{-1}}} [ \phi \circ \vec u^{-1}] \, \vec g \cdot \vec \nu \, \dd \Ha^2 .
 \end{split}
\end{equation}
Therefore,  if \(\omega_{\vec f} = \sum_{i=1}^3 (-1)^{i-1}\phi g_i\widehat{\dd y_i}\), 
\begin{align*}
\langle D^s(\phi \circ \vec u^{-1}), \vec g \rangle 
&= 
-\E_{\vec u}(\phi\,\vec g)=-\langle \partial G_{\vec u}, \omega_{\vec f}\rangle 
= \langle \partial S, \omega_{\vec f} \rangle 
\\
&= \langle S, \sum_{i,\alpha}\p_{x_\alpha} \phi g_i \dd x_\alpha \wedge \widehat{\dd y_i} \rangle 
\leq  \mathbb{M}(S) \|\nabla \phi|_{\overline\Om}\|_{\infty}\|\vec g\|_{\infty},
\end{align*}
where in the last inequality we have used that 
\( \supp S \subset L \times \R^3\).
In particular one finds
\begin{equation*}
\| D^s(\phi\circ \vec u^{-1})\|\leq 
\mathbb{M}(S)\|\nabla \phi|_{\overline\Om}\|_{\infty}.
\end{equation*}
Now we can take a sequence of functions \(\phi_n\in C^1_c(\Om)\) such that \(\phi_n(x_1,x_2,x_3) \rightarrow x_3\)
 in \(C^1\) in a neighbourhood of \(L\) and \( \|\nabla \phi_n|_{\overline\Om}\|_{\infty}\rightarrow 1\). 
 We find \(\|D^s u^{-1}_3\|\). 
Since the first two components of $\vec u^{-1}$ have Sobolev regularity
\cite[Prop.~5.1]{BaHeMoRo_21}, then \(\|D^s (\vec u^{-1})\|\).

The reverse inequality goes as follows. With \(\eta_1,\, \eta_2,\, \eta_3 \in C^1_c(\Om) \), \(\vec g \in C^1_c(\R^3,\R^3)\) we set 
$$
\zeta(\vec x,\vec y):= \sum_{i,\alpha =1}^3 (-1)^{-(i-1)}\eta_\alpha g_i\, \dd x_\alpha\wedge \widehat{\dd y_i}.
$$ We also take an 
$\phi \in C^1_c(\Om_{\vec b})$ such that 
\(\phi (\vec x):=\int_0^{x_3}\eta_3(0,0,s) \dd s\) 
 for all $\vec x$ in a neighbourhood of $L$. 
Set \(\omega =\sum_{i=1}^3 (-1)^{-(i-1)}\phi(\vec x)g_i(\vec y) \widehat{\dd y_i}\). 

Then we have, thanks to \eqref{eq:linkBV} and Lemma \ref{le:defectCurrentXalpha},
\begin{align}\nonumber 
\langle S, \zeta\rangle 
&=\langle S,\dd \omega\rangle 
= 
-\langle \partial G_{\vec u}, \omega\rangle = -\E_{\vec u}(\phi\,\vec g)
\\ 
	\label{eq:linkCurrents}
&=  \int_{\Om_{\vec b}} \nabla \phi(\vec u^{-1}(\vec y))\otimes \vec g(\vec y)\cdot \dd D^c \vec u^{-1}(\vec y)+ \int_{J_{\vec u^{-1}}} [ \phi \circ \vec u^{-1}] \, \vec g \cdot \vec \nu \, \dd \Ha^2 .
\end{align}
By a proof similar to that of Proposition \ref{pr:muproperties}.\ref{it:sym-axis}), based now on \eqref{eq:linkBV}, it can be seen that $\vec u^{-1}(\vec y)\in L$ for $|D^c\vec u^{-1}|$-a.e.\ $\vec y\in \Om_{\vec b}$. 
On the other hand, for $\Ha^2$-a.e.\ $\vec y\in J_{\vec u^{-1}}$ the jump of $\phi\circ \vec u^{-1}$ can be rewritten as
$$
	[\phi\circ \vec u^{-1}](\vec y) = \int_{\big (u_3^{-1}\big )^-(\vec y)}^{\big ( u_3^{-1}\big)^+(\vec y)} \partial_{x_3} \phi(0,0,s) \dd s.
$$
Therefore,
$$
	|\langle S, \zeta \rangle|=|\langle S,\dd \omega\rangle |
 	\leq \|\nabla \phi|_{\overline\Om}\|_\infty |D^su_3^{-1}|\big (\Om_{\vec b}  \setminus J_{\vec u^{-1}}\big ) 
+ \|\nabla \phi|_{\overline\Om}\|_\infty\int_{J_{\vec u^{-1}}} |[\vec u^{-1}]|\dd\Ha^2 .
$$
Now, we have \(\| \nabla \phi|_{\overline\Om} \|_{\infty}\leq  \|\eta_3\|_{\infty}\leq |\zeta|_\infty.\)

We obtain \(\mathbb{M}(S) \leq \|D^s u_3^{-1}\|\) 
and
\begin{equation*}
\mathbb{M}(S) = \|D^s u_3^{-1}\|.
\end{equation*}

It remains to show that
$$
	\|D^s \vec u^{-1}\|
 = \sup \Bigl\{ \E_{\vec u}(\phi \, \vec g):\,\phi\in C^1_c(\Om),\, \vec g \in C^1_c(\R^3, \R^3),\, \|\nabla \phi|_{\overline\Om}\|_{\infty}\leq 1, \|\vec g\|_{\infty}\leq 1\Bigr\} .
$$
From the equality
\begin{multline*}
-\E_{\vec u}(\phi\,\vec g) = 
\int_{\Om_{\vec b}} \nabla \phi(\vec u^{-1}(\vec y))\otimes \vec g(\vec y)\cdot \dd D^c \vec u^{-1}(\vec y)
\\ + \int_{J_{\vec u^{-1}}} \Bigg ( \int_{\big (u_3^{-1}\big )^-(\vec y)}^{\big ( u_3^{-1}\big)^+(\vec y)} \partial_{x_3} \phi(0,0,s) \dd s \Bigg ) \, \vec g \cdot \vec \nu \, \dd \Ha^2,
\end{multline*}
the claim can be obtained arguing as above. 
This concludes the proof. 
\end{proof}

We note, in passing, that from \eqref{eq:linkCurrents} and  Lemma \ref{le:defectCurrentXalpha}, arguing as in \cite[Prop.~4]{HeMoXu15}, \cite[Th.\ 2.3]{Ambrosio95}, it is possible to obtain an alternative proof of the Sobolev regularity of $u^{-1}_1$, and $u^{-1}_2$ established in \cite[Prop.~5.1]{BaHeMoRo_21}.

In order to give an additional intuition about the supremum in Proposition \ref{prop:relaxed_energy_current}, suppose, to fix ideas, that condition INV is satisfied
and $\vec u$ has finite surface energy (i.e., that $D\vec u^{-1}$ has no Cantor parts and
$\Ha^2(J_{\vec u^{-1}})<\infty$). In that case, as mentioned in Section \ref{sec:VI}, the created surface comes from an at most countable collection $C(\vec u)$ of singular points, and it is known \cite[Th.\ 4.6]{HeMo12} that
$$
	\sup \{\E{\vec u}(\phi\,\vec g):\, \phi\in C^1_c(\Om),\, \vec g \in C^1_c(\R^3, \R^3),\, \|\phi\|_{\infty}\leq 1, \|\vec g\|_{\infty}\leq 1\Bigr\} = \sum_{\vec \xi\in C(\vec u)} \Per \imT(\vec u, \vec \xi).
$$
This, in turn, can be rewritten as
$$
	\sum_{\vec \xi\in C(\vec u)} \Per \imT(\vec u, \vec \xi) = \sum_{\vec\xi\in C(\vec u)} \Ha^2(\Gamma_{\vec\xi})= \Ha^2 (J_{\vec u^{-1}}) = \int_{J_{\vec u^{-1}}} 1\cdot \dd\Ha^2.
$$
Proposition \ref{prop:relaxed_energy_current} shows that when the supremum is taken under the constraint $\|\nabla \phi|_{\overline{\Om}}\|_\infty\leq 1$ on the gradient of $\phi$ instead of under the constraint $\|\phi\|_\infty\leq 1$, what is attained is 
$$
	\sum_{\vec\xi,\,\vec \xi'\,\in C(\vec u)} |\vec \xi-\vec \xi'|\Ha^2(\Gamma_{\vec \xi,\,\vec \xi '}) = 
	\int_{J_{\vec u^{-1}}} |[\vec u^{-1}]|\dd\Ha^2 
	=
\|D^s \vec u^{-1}\|.
$$
This contrast is reminiscent of the models for cohesive fracture where a Barenblatt surface energy is added instead of the Griffith term characteristic of brittle fracture (see, e.g., \cite{BoFrMa08,DalMasoIurlano13,FocardiIurlano14,Barchiesi18,Francfort21}).
Also, note that in one case attention is paid only to the deformed state of the body, whereas in the quantity 
$\|D^s\vec u^{-1}\|$ appearing in the neo-Hookean problem the length of the dipole in the reference configuration is also important. In particular, the total energy can be reduced by moving the singular points (with topological charges of opposite sign) closer to each other, using inner variations that produce no increment in the area of the bubble in the deformed configuration. If the singularities are moved by a distance $\delta$ this would pressumably come with a cost of order $\delta^3$ in the elastic energy, but this will be compensated with the gain of $O(\delta^1)$ in the singular energy.

Proposition \ref{prop:previous_article} and Proposition \ref{prop:relaxed_energy_current} strengthen the analogy 
between our candidate for the relaxed energy and the relaxed energy of Bethuel-Brezis-Coron for the harmonic map 
problem form \(\mathbb{B}^3\) (the unit ball in $\R^3$) to \(\Sf^2\) in \cite{BeBrCo90}. 
Indeed, from Proposition \ref{prop:relaxed_energy_current} we see that \( \|D^s\vec u^{-1}\|\) 
is the perfect analogue of the quantity \(L (\vec u)\) in \cite[Eq.~1]{BeBrCo90}. Furthermore their relaxed energy was reformulated in the context of Cartesian currents in \cite{Giaquinta_Modica_Soucek_1989a} (see also \cite{GiMoSo98II}). It is of the form \( \frac12 \int_{\mathbb{B}^3} |D \vec u|^2+ \mathbb{M}(S)\) where \(S\) is the defect current associated to \( \vec u\) (we refer to \cite{Giaquinta_Modica_Soucek_1989a} for the precise definition). 

In the context of harmonic maps, the relaxed energy was investigated in order to explore the question of existence of smooth minimising harmonic maps from \(\mathbb{B}^3\) to \(\Sf^2\) for a given boundary data with zero degree. This question, raised by Hardt and Lin in \cite{Hardt_Lin_1986} (see question (2) at the end of the paper), is still unsolved to this day. It was also investigated in an axisymmectric setting in \cite{Hardt_Lin_1992}, where the authors showed that axially symmetric minimisers can have singularities. Concerning the (partial) regularity of minimisers of the relaxed energy for harmonic maps, we refer to \cite[Th.\ 1 p. 424]{GiMoSo98II} and \cite{Bethuel_Brezis_1991}. Our paper indicates that the problem of existence of minimising configurations for neo-Hookean materials is of the same type of the Hardt--Lin problem, making a bridge between these two areas. Furthermore, in the context of nonlinear elasticity, the supplementary term has a clear physical interpretation which may help in the understanding of the regularity problem, and thus shed new light on both problems.

\appendix

\section{Technical lemmas about the zenith angle function of the bubble}
\begin{lemma}
        \label{le:f_and_derivatives}
    Let $f_\e(r)$ be as in \eqref{eq:def_stereo}. Then,   for \(\e\) small enough (\(\e <\min(1/e,\sqrt{3}/2)\)),
    \begin{enumerate}[a)]
     \item 
     \label{it:p_r-cos_f}
        $ \displaystyle \frac{1}{2}\cdot \e^2r(\e^4+r^2)^{-3/2} \leq \partial_r \Big (-\cos f_\e(r)\Big ) \leq 
        6\e^2r(\e^4+r^2)^{-3/2} $.
    \item 
    \label{it:int-p_r-cos_f}
     $ \displaystyle \int_{r=0}^\e r\big |\partial_r\big (\cos f_\e(r)\big )\big |\dd r 
     \leq 12\e^2|\ln\e|$.
    
    \item
    \label{it:p_r-complicated}
     $ \displaystyle   \left |\partial_r \Bigg ( \frac{r}{\partial_r\Big (\cos f_\e(r)\Big)}
        \Bigg )
        \right | \leq 64
        \e^{-2}r(\e^4+r^2)^{1/2}\leq 64\sqrt{2}$.

    \end{enumerate}
    
\end{lemma}

\begin{proof}
\underline{Part \ref{it:p_r-cos_f})}
We write 
 $$
    f(r) = \underbrace{\arctan \big ( r/\e^{-2} \big )}_{:=A(r)} + 
    \underbrace{\alpha_\e \frac{r}{\e}}_{:=B(r)}, \quad   f'(r) = \frac{\e^2}{\e^4+r^2}+\underbrace{\frac{\alpha_\e}{\e}}_{\geq 0}.
 $$
 First we observe that
 \begin{align*}
  \sin^2A = \frac{\tan^2 A}{1+\tan^2A} = \frac{r^2}{\e^4+r^2}, \quad   \cos^2 A =\frac{1}{1+\tan^2A}
 \end{align*}
 so
 \begin{equation}
    \label{eq:sin_A}
  \sin A = \frac{r}{\sqrt{\e^4+r^2}}, \quad \cos A = \frac{\e^2}{\sqrt{\e^4+r^2}}.
 \end{equation}
 Hence,
 \begin{equation}\label{eq:p_r-cos-explicit}
  \partial_r \big ( -\cos f(r)\big )
  = \sin f(r)\,f'(r)
  =(\sin A \cos B + \underbrace{\sin B \cos A}_{\geq 0} ) f'(r)
 \geq \frac{r}{(\e^4 +r^2)^{1/2}}\cdot \cos B \cdot \frac{\e^2}{\e^4 + r^2}.
 \end{equation}
Considering that $
    \cos^2 B = 1 - \sin^2 B \geq 1-B^2 \geq 1-\e^2,$
 the above expression yields the lower bound in the statement.
 
 In order to obtain the upper bound, first note that 
 \begin{align*}
  \frac{\alpha_\e}{\e} = \frac{\arctan(\e)}{\e}
  \leq 1 \leq \frac{2\e^2}{\e^4+r^2},
 \end{align*}
 so that 
 \begin{equation}
    \label{eq:u_bound_f'}
  f'(r) \leq \frac{3\e^2}{\e^4 +r^2}.
 \end{equation}
 Regarding $\sin f(r)$, note that
 \begin{align*}
  \sin f(r) &= \sin A \cos B + \sin B \cos A
  = \frac{r}{\sqrt{\e^4+r^2}}\cos B + \sin B \,\frac{\e^2}{\sqrt{\e^4+r^2}}
  \\
  &\leq \frac{r}{\sqrt{\e^4+r^2}} + \frac{\e^2}{\sqrt{\e^4+r^2}}\sin B.
 \end{align*}
Now, \(
 0\leq \sin B\leq B = \frac{\arctan(\e)}{\e}r \leq r,
\)
therefore
\begin{align*}
 0\leq \sin f(r)
 \leq \frac{r}{\sqrt{\e^4+r^2}} (1+\e^3) 
 \leq \frac{2r}{\sqrt{\e^4+r^2}}.
\end{align*}
That, combined with \eqref{eq:u_bound_f'}, yields
the upper bound in \emph{\ref{it:p_r-cos_f})}.
\medskip

\underline{Part \ref{it:int-p_r-cos_f}):}
We split the integration interval into two parts: $(0,\e^2)$ and $(\e^2, \e)$.
In the first part, we use the upper bound obtained in part a) and $r^2\leq \e^4$ to write that
\begin{align*}
 \int_0^{\e^2} r|\partial_r\big(\cos f(r)\big ) |\dd r 
 &
 \leq \int_0^{\e^2}
 \frac{6\e^2r^2\dd r}{(\e^4+r^2)^{3/2}}
 \leq \int_0^{\e^2}
 \frac{6\e^6\dd r}{\e^6}=6\e^2.
\end{align*}
In the second part, using again the upper bound in a) and using that
$\e^4+r^2 \geq r^2$
the following is obtained:
\begin{equation*}
 \int_\e^{\e^2}
 r|\partial_r\big(\cos f(r)\big ) |\dd r 
 \leq
 \int_\e^{\e^2}
 \frac{6\e^2r^2\dd r}{(r^2)^{3/2}}
 =6\e^2\ln \frac{1}{\e}.
\end{equation*}
If $\e<1/e$ then $1<|\ln\e|$. Hence,
\begin{align*}
 \int_0^{\e^2}
 r|\partial_r\big(\cos f(r)\big ) |\dd r 
 &
 \leq 6\e^2(1+|\ln \e|)\leq 12\e^2|\ln \e|.
\end{align*}

\underline{Part \ref{it:p_r-complicated}):}
Going back to \eqref{eq:p_r-cos-explicit} one finds that
\begin{align*}
 \frac{\partial_r\big (\cos f(r)\big )}{r}
 =  \left (\frac{\sin A}{r}\cos B + \cos A \,\frac{\sin B}{r}\right ) f'(r).
\end{align*}
Since 
$$\frac{r}{\partial_r\big (\cos f(r)\big )} 
= 1\Bigg / \frac{\partial_r\big (\cos f(r)\big )}{r},$$
differentiation with respect to $r$ gives
\begin{multline*}
 -\Bigg[ 
 \frac{\partial_r\big (\cos f(r)\big )}{r}
 \Bigg ]^{2} \partial_r \Bigg ( 
 \frac{r}{\partial_r\big (\cos f(r)\big )}
 \Bigg )
 = f''(r) \left (\frac{\sin A}{r}\cos B + \cos A \,\frac{\sin B}{r}\right ) 
 \\
 + f'(r) \Bigg(
   \Big ( \frac{\sin A}{r}\Big )'\cos B 
   - \frac{\sin A}{r}\sin B\,B'
   + (\cos A)'\frac{\sin B}{r} + \cos A \Big ( \frac{\sin B}{r}\Big )'\Bigg ).
\end{multline*}

From equations \eqref{eq:sin_A} it is easy to see that
\begin{align*}
 & f''(r) = - 2\e^2 r (\e^4+r^2)^{-2},
 \quad
  \frac{\sin A}{r} = (\e^4 +r^2)^{-1/2},\quad
  \Big (\frac{\sin A}{r}\Big )' =-r(\e^4 +r^2)^{-3/2},
\\
 & \cos A =\e^2(\e^4+r^2)^{-1/2}, \quad \text{and}\quad 
 (\cos A)'= -\e^2r(\e^4+r^2)^{-3/2}.
\end{align*}
Also, \(
 \left|\frac{\sin B}{r}\right | \leq \frac{B}{r}=\frac{\arctan(\e)}{\e}\leq 1, \) and we see that
\[
  \left |\Big (\frac{\sin B}{r}\Big )' \right |
 = \left|\frac{\cos B\cdot r\cdot B' - \sin B}{r^2}\right |
 \leq r^{-2}\big (|B-\sin B| + |B||1-\cos B|\big )
 \leq r^{-2}\Big ( 
 \frac{B^3}{3!} + \frac{B^3}{2!}\Big )\leq r. 
\]

Combining the lower bound in Part \ref{it:p_r-cos_f}),
\eqref{eq:u_bound_f'}, and the above calculations yields
\begin{multline*}
 \frac{1}{4}\e^4 (\e^4+r^2)^{-3}
 \Big | \partial_r \Big ( \frac{r}{\cos f(r)}\Big ) \Big |
 \leq 
    2\e^2 r (\e^4 +r^2)^{-2}
    \Big ( (\e^4 +r^2)^{-1/2}\cdot 1 
    + \e^2 (\e^4+r^2)^{-1/2}\cdot 1 \Big )
\\
    + 3\e^2(\e^4+r^2)^{-1} \Big ( r(\e^4 + r^2)^{-3/2}\cdot 1 
    + (\e^4 +r^2 )^{-1/2}\cdot r \cdot 1
    \\
    + \e^2 r(\e^4 + r^2)^{-3/2}\cdot 1
    +
    \e^2 (\e^4 + r^2)^{-1/2} r \Big ).
\end{multline*}

Therefore we find that
\begin{align*}
    \left| \partial_r \Big ( \frac{r}{\cos f(r)}\Big ) \right|
    \leq 4\e^{-4}(\e^4+r^2)^3\cdot (2(1+\e^2) + 3(2+2\e^2))\e^2 r(\e^4 +r^2)^{-5/2}.
\end{align*}
From here, the claim follows, considering that $r\leq \e$ and $\e^4+r^2\leq 2\e^2$.
\end{proof}

\begin{lemma}
	\label{le:positive_h}
	For every $\varphi\in [0,\frac{\pi}{2})$ and $s\in [0,1]$,
	$$
		g_\e(\varphi)<\e,\qquad 
		\e - g_\e(\varphi)\sin \varphi >0
		\quad
		\text{and}
		\quad
		h_\e(s,\varphi)>0.
	$$
\end{lemma}

\begin{proof}
	For every $r<\e$
	$$
		f(r)= \arctan(r/\e^2) + \alpha_\e\frac{r}{\e} < \arctan(1/\e) + \arctan(\e)=\frac{\pi}{2}=f(\e).	
	$$
	Hence 
	$$
		\varphi <\frac{\pi}{2}\ \Rightarrow\ g(\varphi) < \e
		\ \Rightarrow\ \e -g(\varphi)\sin\varphi > \e (1-\sin \varphi) >0.
	$$
	The remaining terms in the formula for $h(s,\varphi)$ are easily seen to be also strictly positive. 
\end{proof}

\begin{lemma}
	\label{le:firstHalf}
For all	$0\leq \varphi\leq \frac{\pi}{4}$ it holds that $g_\e(\varphi)\leq \e^2$.
\end{lemma}
\begin{proof}
	Set $r:=g(\varphi)$. Since
	$$
		\frac{\pi}{4}\geq \varphi= \arctan(r/\e^2) + \alpha_\e \frac{r}{\e} \geq \arctan(r/\e^2),
	$$
	then \(	1\geq \tan(\varphi) \geq r/\e^2\).
\end{proof}

\begin{lemma}
		\label{le:gPrime}
	The first and second derivatives of $g_\e$ are given by
	$$
		g'_\e(\varphi) = \frac{1}{\frac{\e^2}{r^2+\e^4} + \frac{\alpha_\e}{\e}},\qquad 
		g''_\e(\varphi) = 
		2 \e^2 r (r^2+\e^4)^{-2} \left ( \frac{\e^2}{r^2+\e^4}+ \frac{\alpha_\e}{\e} \right ) ^{-3},
		\qquad
		r=g(\varphi).
	$$
	Furthermore,
\begin{align*}
 & \frac{\e^2}{2}\leq g'_\e(\varphi)\leq 2 && \text{and} \qquad \qquad  |g''_\e(\varphi)|\leq 4\e^{-1} , & \forall \varphi \in [0,\frac{\pi}{2}] , \\
 & |g'_\e(\varphi)|\leq 2\e^2 && \text{and} \qquad \qquad  |g''_\e(\varphi)|\leq 4\e^2 , & \forall \varphi \in [0,\frac{\pi}{4}] .
\end{align*}

\end{lemma}

\begin{proof}
	Since $g$ is the inverse of $f$,
	$$
		g'(\varphi) = \frac{1}{f'(g(\varphi))},\qquad g''(\varphi) = \frac{-f''(g(\varphi))}{f'(g(\varphi))^3}.
	$$
	The first bound for $g'$ comes by noting that 
	$$
		0\leq r\leq \e\ \Rightarrow\ r^2 +\e^4 \leq \e^2 + \e^2 
		\ \Rightarrow\ 
		\frac{1}{2} = \frac{\e^2}{2\e^2}\leq \frac{\e^2}{r^2+\e^4} + \frac{\alpha_\e}{\e}  \leq \e^{-2}+
		\frac{\alpha_\e}{\e}  
		\leq 2\e^{-2}
	$$
	(recall that $g(\varphi)\leq \e$ for all $\varphi$ because it is the inverse of the function $f$ whose domain is $[0,\e]$).
	The finer estimate in the left-half interval $[0,\frac{\pi}{4}]$ is obtained similarly but using Lemma \ref{le:firstHalf} and  noting that $r^2 +\e^4\leq \e^4+\e^4$ when $r\leq \e^2$. 
	The estimates for $|g''|$ follow from the inequality
	$$
		g''(\varphi) \leq 2\e^2 r (r^2+\e^4)^{-2}\left ( \frac{\e^2}{r^2+\e^4}\right ) ^{-3}
		= 2\e^{-4} r(r^2+\e^4).
	$$
\end{proof}

\begin{lemma}
		\label{le:hE2Cos}
	For all $\varphi \in (0,\frac{\pi}{2}]$, $s\in [0,1]$, and $\e \leq \frac{1}{\pi}$,
\begin{align*}
 &		(1-s) \frac{g_\e(\varphi)}{\sin\varphi}+s\e \leq \sqrt{2}\e,
		\qquad
		0< (1-s)g'_\e(\varphi)\cos \varphi + s(\e - g_\e(\varphi)\sin\varphi) 
	\leq \frac{3\pi}{2}\cos\varphi,
\\ &
		|h_\e(s,\varphi)|\leq \frac{3\pi\sqrt{2}}{2}\, \e^2\cos\varphi 
		\quad 
		\text{and}
		\quad
		|\partial_\varphi h_\e(s,\varphi)| = O(\e).
\end{align*}
\end{lemma}

\begin{proof}
By Lemma \ref{le:positive_h}, for all $\varphi$
	$$
		g(\varphi)\leq \e 
		\quad
		\text{and}
		\quad 
		0< \e- g(\varphi)\sin \varphi \leq \e.
	$$
	For $\varphi \in [0, \frac{\pi}{4}]$, since $\cos \varphi \geq \frac{1}{\sqrt{2}}$, it holds, in particular, that 
	$$
		\e - g(\varphi)\sin \varphi \leq \sqrt{2}\,\cos \varphi.
	$$
	For $\varphi \in [\frac{\pi}{4}, \frac{\pi}{2}]$, since
\begin{equation*}
	-(\e - g(\varphi)\sin \varphi)'
	= \underbrace{g'(\varphi)}_{\leq 2}\sin\varphi +
	 \underbrace{g(\varphi)}_{\leq \e}\cos\varphi \leq 3,
\end{equation*}
then
\begin{equation*}
	0\leq \e - g(\varphi)\sin \varphi = -\int_\varphi^{\frac{\pi}{2}} (\e - g(t)\sin t)'\dd t \leq 3(\frac{\pi}{2} - \varphi).
\end{equation*}
The relation 
$$
	\frac{\pi}{2}-\varphi = \frac{\frac{\pi}{2}-\varphi}{\sin(\frac{\pi}{2}-\varphi)}\cos \varphi 
	\leq 
	\frac{\pi}{2} \cos\varphi 
$$
then yields 
\begin{equation*}
	0\leq \e - g(\varphi)\sin \varphi \leq \frac{3\pi}{2} \cos \varphi
	\quad \forall\varphi \in [0,\frac{\pi}{2}]
\end{equation*}
and
$$
	0\leq (1-s)\underbrace{g'(\varphi)}_{\leq 2}\cos \varphi + s(\e - g(\varphi)\sin\varphi) 
	\leq \big ( (1-s)+ s\big ) \frac{3\pi}{2}\cos\varphi.
	$$
	
So as to estimate the first factor in $h$, since $g(0)=0$ and $\sup_{[0,\frac{\pi}{4}]} |g'(\varphi)|\leq 2\e^2$, 
	\begin{align}
			\label{eq:gNearZero}
		0\leq \varphi<\frac{\pi}{4}\ \Rightarrow\ 
		0\leq g(\varphi)\leq 2\e^2 \varphi.
	\end{align}
	Using that $\frac{\sin \varphi}{\varphi}$ is decreasing (and taking its value at $\varphi=\frac{\pi}{2}$ for simplicity) we obtain that 
	$$
		\frac{g(\varphi)}{\sin \varphi}
		\leq 2\e^2 \frac{\varphi}{\sin \varphi}
		\leq 2\e^2 \frac{\frac{\pi}{2}}{\sin \frac{\pi}{2}} = \pi \e^2\leq \e.
	$$
	On the other hand, for $\varphi\in [\frac{\pi}{4}, \frac{\pi}{2}]$,
	$$
		\frac{g(\varphi)}{\sin \varphi}
		\leq \frac{\e}{\frac{1}{\sqrt{2}}}
		=\sqrt{2}\e.
	$$
	Therefore,
	\begin{align*}
		(1-s) \frac{g(\varphi)}{\sin \varphi} + s\e \leq (1-s)\sqrt{2}\e +s\e \leq \sqrt{2}\e\quad
		\text{and}
		\quad
		|h(s,\varphi)|\leq \frac{3\pi\sqrt{2}}{2}\e^2 \cos\varphi.
	\end{align*}
	
	In order to estimate $\partial_\varphi \Big ( \frac{g(\varphi)}{\sin \varphi} \Big )$, as before consider first the case when $\varphi < \frac{\pi}{4}$. 
	By \eqref{eq:gNearZero}
	and Lemma \ref{le:firstHalf},
	$$
		|g''(\varphi)|\leq 2\e^{-4} g(\varphi)(g(\varphi)^2 + \e^4)
		\leq 2\e^{-4} (2\e^2\varphi) ((\e^2)^2 + \e^4) = 8 \e^2 \varphi.
	$$
	Hence
	$$
		|(g'(\varphi)\sin\varphi - g(\varphi)\cos \varphi)'|
		= | g''(\varphi)+g(\varphi)| \sin\varphi
		\leq 10\e^2 \varphi^2.
	$$
	Since 
	$$
		\lim_{\varphi\to 0^+} g'(\varphi) \sin \varphi - g(\varphi)\cos\varphi
		= \e^2 \cdot 0 - 0\cdot 1=0,
	$$
	integrating its derivative we obtain that 
	$$
		| g'(\varphi) \sin \varphi - g(\varphi)\cos\varphi| \leq \frac{10}{3}\e^2 \varphi^3.
	$$
	Using again that $\frac{\sin \varphi}{\varphi}$ is decreasing  we find that
	\begin{align*}
		\left | \partial_\varphi \Big ( \frac{g(\varphi)}{\sin \varphi} \Big ) \right |
		\leq \frac{\frac{10}{3}\e^2 \varphi^3}{\sin^2\varphi} 
		\leq \frac{10}{3}\e^2
		\Big ( \frac{\frac{\pi}{2}}{\sin \frac{\varphi}{2}} \Big )^2 \varphi 
		= \frac{10\pi^2}{12} \e^2 \varphi.
	\end{align*}
	
	When $\frac{\pi}{4}\leq \varphi\leq \frac{\pi}{2}$,
	$$
	\left | \partial_\varphi \Big ( \frac{g(\varphi)}{\sin \varphi} \Big ) \right |
		\leq \frac{|g'(\varphi)| + |g(\varphi)|}{\sin^2\varphi}
		\leq 
		\frac{2 +\e}{1/2}\leq 5. 
	$$

	The derivative of $h$ is 
	\begin{align*}
		\partial_\varphi h(s,\varphi ) &= \e (1-s) \underbrace{\partial_\varphi \Big ( \frac{g(\varphi)}{\sin \varphi} \Big )}_{\leq 5}
		\Big ( (1-s) \underbrace{g'(\varphi)}_{\leq 2}\cos\varphi
	+ s \underbrace{(\e-g(\varphi)\sin\varphi)}_{\leq \e}\Big )
	\\
	&\quad + 
	\e \Big ( \underbrace{(1-s) \frac{g(\varphi)}{\sin \varphi} + s\e}_{\leq \sqrt{2} \e} \Big ) 
	\Big ( (1-s) \underbrace{\big (g'(\varphi)\cos \varphi )'}_{=O(\e^{-1})}
	- s \underbrace{\big ( g(\varphi) \sin\varphi \big )'}_{=O(1)} \Big ).
	\end{align*}
\end{proof}

\begin{lemma}
		\label{le:lowerEMinusG}
	For all $\varphi \in [0,\frac{\pi}{2}]$,
	$$
		 \frac{1}{3} \leq \frac{\e-g(\varphi)}{\max\{\e,\frac{g(\varphi)}{\e}\}\cos\varphi}
		 \leq 2\sqrt{2},
		 \quad 
		 \text{and}
		 \quad 
		\frac{\e-g(\varphi)\sin\varphi}{g'(\varphi)\cos\varphi}\leq \frac{8}{\max\{\e,\frac{g(\varphi)}{\e}\}}.
	$$
\end{lemma}

\begin{proof}
	Set $r:=g(\varphi)$. By definition of $g$,
	$$
		\varphi= \arctan(\frac{r}{\e^2}) + \frac{\alpha_\e r}{\e}.
	$$ 
	Since $\alpha_\e\frac{r}{\e}\leq \alpha_\e=\arctan(\e)=O(\e)$, the mean value theorem applied to $t\mapsto \tan(t)$ yields
	\begin{align*}
		\e - \tan (\frac{\alpha_\e r}{\e}) &=
		\tan (\alpha_\e) - \tan ( \frac{\alpha_\e r}{\e}) \leq 
		2(\alpha_\e - \frac{\alpha_\e r}{\e}) = 2\frac{\alpha_\e}{\e} (\e-r)
		\leq 2(\e-r).
	\end{align*}
	On the other hand, the formula $\frac{\cos (x+y)}{\sin(x+y)} = \frac{1-\tan x\tan y}{\tan x+\tan y}$ applied to $x=\arctan(\frac{r}{\e^2})$, $y=\frac{\alpha_\e r}{\e}$ gives
	\begin{align*}
		\frac{\cos\varphi}{\sin\varphi}
		&= \frac{1- \frac{r}{\e^2}\tan(\frac{\alpha_\e r}{\e})}{\frac{r}{\e^2}
		+ \tan(\frac{\alpha_\e r}{\e})}
		\leq \frac{1- \frac{r}{\e^2}\tan(\frac{\alpha_\e r}{\e})}{\frac{r}{\e^2}} = \frac{\e^2}{r} - \e + \e - \tan(\frac{\alpha_\e r}{\e})
		\leq  \frac{\e}{r} (\e-r) + 2(\e-r).
	\end{align*}
	If $r\geq \e^2$ then 
	$\max\{\e, \frac{r}{\e}\}= \frac{r}{\e}$, so
	$$
		r\geq \e^2 \ \Rightarrow\ \max\{\e, \frac{r}{\e}\}\cos \varphi \leq \frac{r}{\e}
		(\frac{\e}{r}+2)\sin\varphi\, (\e-r) \leq 3 (\e-r). 
	$$
	If $r\leq \e^2$ then $\e-r>\e/2$
	and 
	$$
		\max\{\e, \frac{r}{\e}\}\cos \varphi \leq \e  \leq 2(\e-r),
	$$
	proving the first lower bound.
	
	Analogously,
	\begin{equation*}
		\frac{\cos\varphi}{\sin\varphi}
		= \frac{1- \frac{r}{\e^2}\tan(\frac{\alpha_\e r}{\e})}{\frac{r}{\e^2}
		+ \tan(\frac{\alpha_\e r}{\e})}
		\geq \frac{1-\frac{r}{\e^2}\cdot \e}{\frac{r}{\e^2} + \e}
		= \frac{\e(\e-r)}{r+\e^3}
		\geq \frac{\e (\e-r)}{r+r}.
	\end{equation*}
		If $r\geq \e^2$ then $\varphi\geq \arctan(1)+ \frac{\alpha_\e}{\e} \geq \frac{\pi}{4}$ and
		\begin{align*}
		\e - r \leq \frac{2r}{\e \sin \varphi} \cos \varphi \leq 2\sqrt{2} \frac{r}{\e}\cos\varphi.
	\end{align*}
	If $r\leq \e^2$ then $\varphi \leq \frac{\pi}{4} + O(\e)$, $\cos\varphi \geq \frac{1}{2}$, and
	\begin{equation*}
		\e -r \leq \e = 2\cdot \e\cdot \frac{1}{2}\cdot \e \leq 2\e \cos\varphi.
	\end{equation*}
	In both cases
	\begin{equation*}
		\e - r \leq 2\sqrt{2} \max\{\e, \frac{r}{\e}\}\cos\varphi,
	\end{equation*}
	which proves the upper bound in the first claim.
		
	Note now that 
	\begin{align*}
		\frac{1}{g'(\varphi)} = f'(r)= \frac{\e^2}{r^2 +\e^4} + \frac{\alpha_\e}{\e} \leq 2\frac{\e^2}{\max\{r,\e^2\}^2}
		= \frac{2}{\max\{\frac{r}{\e}, \e\}^2}.
	\end{align*}
	Thus,
	\begin{equation*}
		\frac{\e-g(\varphi)}{g'(\varphi)\cos\varphi}\leq \frac{4\sqrt{2}}{\max\{\frac{r}{\e}, \e\}}.
	\end{equation*}
	To obtain the last claim, write
	\begin{align*}
		\frac{\e - g(\varphi)\sin\varphi}{g'(\varphi)\cos\varphi} = \frac{\e - g(\varphi)}{g'(\varphi)\cos\varphi}
		+ \frac{g(\varphi)}{g'(\varphi)} \frac{1-\sin \varphi}{\cos \varphi}
		\leq \frac{4\sqrt{2}}{\max\{\frac{r}{\e}, \e\}}
		+\frac{2r}{\max\{\frac{r}{\e}, \e\}^2}
		\frac{1-\sin \varphi}{\cos\varphi}.
	\end{align*}
	The second factor is bounded:
	\begin{align}
			\label{eq:1menosSinCos}
		\frac{1-\sin \varphi}{\cos\varphi}
		= \frac{\sin \Big ( \frac{\frac{\pi}{2} - \varphi}{2}\Big ) }{\cos \Big ( \frac{\frac{\pi}{2} - \varphi}{2}\Big ) }\leq \tan \frac{\pi}{4} = 1.
	\end{align}
	The proof is finished by noting that 
	$r\leq \max\{\frac{r}{\e},\e\}^2$ for all $r\in [0,\e]$. 
\end{proof}

\begin{lemma}
	\label{le:logEpsilon}
	If $a$, $b$, and $\lambda\ne 1$ are any positive numbers such that $\frac{b}{a}<\lambda$, then 
	$$
		\int_{s=0}^1 \frac{\dd s}{(1-s)a+sb} \leq  \frac{1}{1-\lambda^{-1}}\frac{\ln (\lambda)}{b}.
	$$
\end{lemma}
\begin{proof}
	Changing variables to $u=(1-s)a+sb$
	it follows that
	$$
		\int_{s=0}^1 \frac{\dd s}{(1-s)a+sb} = \frac{\ln (b/a)}{b-a}
		= \frac{1}{b} \frac{ \ln t}{1- t^{-1}} \Big |_{t=b/a}.
	$$
	This yields the conclusion because that function is increasing in $t$.
\end{proof}

\begin{lemma}
	For all $\varphi \in [0, \frac{\pi}{2})$,
	$$
		g'(\varphi)\cos \varphi < \e - g(\varphi)\sin \varphi .
	$$
\end{lemma}

\begin{proof}
It suffices to observe that
$\e - g(\varphi)\sin\varphi - g'(\varphi)\cos\varphi$
vanishes as $\varphi\to \frac{\pi}{2}$ and
\begin{align*}
	\Big ( \e - g(\varphi)\sin\varphi - g'(\varphi)\cos\varphi)\Big )' 
	&= -\big ( g(\varphi) + g''(\varphi)\big ) \cos\varphi <0.
\end{align*}
\end{proof}

\begin{lemma}
		\label{le:logNablaVarphi}
	For all $\varphi \in [0,\frac{\pi}{2})$,
	\begin{align*}
		\int_{s=0}^1\frac{\dd s }{(1-s)g'(\varphi)\cos \varphi + s(\e - g(\varphi)\sin\varphi)}
		\leq \frac{2}{\e-g(\varphi)} |\ln \e|.
	\end{align*}
\end{lemma}

\begin{proof}
	By Lemma \ref{le:positive_h},
	$$
		0<\e - g(\varphi)< \e - g(\varphi)\sin \varphi.
	$$
	Combining that with the previous lemmas,
	with 
	$$
		a= g'(\varphi)\cos \varphi, \quad b= \e - g(\varphi)\sin\varphi,\quad \lambda = \frac{8}{\e},
	$$
	it is obtained that
	\begin{align*}
		\int_{s=0}^1\frac{\dd s }{(1-s)g'(\varphi)\cos \varphi + s(\e - g(\varphi)\sin\varphi)} 
		&\leq \frac{1}{1-\frac{\e}{8}} \frac{\ln \frac{8}{\e}}{\e - g(\varphi)\sin \varphi}.
	\end{align*}

\end{proof}

\section{Working with axially symmetric maps}\label{sec:Appendix_B}

In this appendix, we describe axially symmetric maps in cylindrical and spherical coordinates. We compute the differential of an axially symmetric map in those coordinates system. This allows us to express the Dirichlet energy, the cofactor matrix and the Jacobian determinant of these maps. All these quantities are needed in Section \ref{sec:towards_upper_bound}  to understand the construction of  test functions which proves that the lower bound of the relaxed energy found in Theorem \ref{th:previous_article} is optimal for the limiting map of the Conti--De Lellis example.

\subsection{Axially symmetric maps in cylindrical and spherical coordinates}
We denote by \( (r,\theta,x_3)\in \R^+\times [0,2\pi) \times \R\) the cylindrical coordinates and \(( \rho,\theta,\varphi) \in \R^+ \times [0,2\pi) \times [0,\pi]\) the spherical coordinates. Hence we have
\[ (x_1,x_2,x_3)=(r\cos \theta,r\sin\theta, x_3)= (\rho \cos \theta \sin \varphi,\rho \sin \theta \sin \varphi, \cos \varphi). \]
Note that  \( \theta\) is the same angle for cylindrical and spherical coordinates.
We consider axisymmetric maps with respect to the \(x_3\)-axis. Let \(\Om\) be an axisymmetric domain. The map \(\vec u: \Om \rightarrow \R^3\) is axisymmetric if, in cylindrical coordinates, it can be expressed in the following way
\[ \vec u(r\cos \theta,r\sin \theta,z)=v_1(r,z) \vec e_r +v_2(r,z) \vec e_3 =v_1(r,z) (\cos \theta \vec e_1 +\sin \theta \vec e_2)+v_2(r,z)\vec e_3,\]
for some \( \vec v=(v_1,v_2): \vec P(\Om) \rightarrow \R^+ \times \R\), with \(\vec P\) defined in \eqref{def:pi_P}.
In this case the cylindrical coordinates \( (u_r,u_\theta,u_{x_3})\) of the map \(\vec u\) are \( (v_1(r,z),\theta, v_2(r,z))\). We can also express spherical coordinates of the map
\[ (u_\rho,u_\theta,u_\varphi)=\left(\sqrt{v_1^2(r,z)+v_2^2(r,z)},\theta, \arccos \left(\frac{v_2(r,z)}{v_1^2(r,z)+v_2^2(r,z)}\right)\right). \]

Thus, a  map \(\vec u: \Om \rightarrow \R^3\) is axisymmetric if its spherical coordinates \( (u_\rho,u_\theta,u_\varphi)\) satisfy that \( u_\rho, u_\varphi\) do not depend on \(\theta\) and \(u_\theta=\theta\).

Next we give the expression of the differential of a map given in different coordinate systems.

\subsection{Use of cylindrical-cylindrical coordinates}
We recall from the Appendix in \cite{HeRo18} that if \(\vec u: \Om \rightarrow \R^3\) is axisymmetric and is given in cylindrical coordinates by \( \vec u(r\cos \theta, r\sin \theta, x_3)= v_1(r,x_3) \vec e_r +v_2(r,x_3) \vec e_3\) then
\[
D\vec u= \begin{pmatrix}
\p_rv_1 & 0 & \p_{x_3} v_1 \\
0 & \frac{v_1}{r} & 0 \\
\p_rv_2 & 0 & \p_{x_3}v_2
\end{pmatrix},
\quad
\cof D \vec u = \begin{pmatrix}
\frac{v_1}{r}\p_{x_3}v_2 & 0 & -\frac{v_1}{r}\p_rv_2 \\
0 & \det D v & 0 \\
-\frac{v_1}{r}\p_{x_3}v_1 & 0 & \frac{v_1}{r}\p_rv_1
\end{pmatrix},
\quad
\det D\vec u = \frac{v_1}{r}  \det D \vec v,
\]
and the neo-Hookean energy is given by
\begin{equation*}
E(\vec u) =2\pi \int_{\pi(\Om)}\left(|\p_r \vec v|^2+|\p_{x_3} \vec v|^2\right)rdrd x_3 +2\pi \int_{\pi(\Om)}\frac{v_1^2}{r}\dd r\dd x_3.
\end{equation*}
\subsection{Use of spherical-spherical coordinates}\label{subse:spherical-spherical}
Let us assume that \( \vec u \) is an axisymmetric map given in spherical coordinates in the domain and in the target. We can write
\[
\vec u(\rho \cos\theta \sin \varphi, \rho \sin \theta \sin \varphi, \rho \cos \varphi)=u_\rho(\rho,\varphi) \vec e_\rho, \ \text{ with } \
\vec e_\rho: =( \cos \theta \sin u_\varphi, \sin \theta \sin u_\varphi, \cos u_\varphi).
\]

Let \( \vec \gamma(t)= (\rho(t),\theta(t), \varphi(t))\) be a \(C^1\) curve in spherical coordinates. We compute \( \frac{d}{dt} \left[ \vec u\circ \vec \gamma \right]\). By the chain rule,
\begin{align*}
\frac{d[u\circ \gamma](t)}{dt}=\p_\rho u_\rho \dot{\rho} \vec e_\rho +u_\rho \dot{\rho}\p_\rho \vec e_\rho +\p_\rho u_\varphi \dot{\varphi} \vec e_\rho +u_\rho \dot{\varphi}\p_\varphi \vec e_\rho +u_\rho \dot{\theta} \p_\theta \vec e_\rho.
\end{align*}
We let
\begin{equation*}
\vec e_\theta  := (-\sin \theta, \cos \theta, 0) , \qquad \vec e_\varphi := (\cos \theta \cos u_\varphi, \sin \theta \cos u_\varphi,-\sin u_\varphi).
\end{equation*}
The basis \( (\vec e_\rho, \vec e_\varphi, \vec e_\theta)\) is orthonormal and $\vec e_\rho \wedge \vec e_\varphi = \vec e_\theta$. We can verify that
\begin{align*}
\p_\rho \vec e_\rho &= \p_\rho u_\varphi \vec e_\varphi , \\
\p_\theta \vec e_\rho &= (-\sin \theta \sin u_\varphi, \cos \theta \sin u_\varphi,0)=\sin u_\varphi \vec e_\theta , \\
\p_\varphi \vec e_\rho &= \p_\varphi u_\varphi (\cos \theta \cos u_\varphi, \sin \theta \cos u_\varphi,-\sin u_\varphi)=\p_\varphi u_\varphi \vec e_\varphi.
\end{align*}
Thus we find
\[
\frac{d[u\circ \gamma](t)}{dt}=\p_\rho u_\rho \dot{\rho}\vec e_\rho +u_\rho \dot{\rho} \p_\rho u_\varphi \vec e_\varphi+\p_\varphi u_\rho \dot{\varphi} \vec e_\rho +u_\rho \dot{\varphi}\p_\varphi u_\varphi \vec e_\varphi+u_\rho \dot{\theta}\sin u_\varphi \vec e_\theta.
\]
On the other hand, \( \gamma(t)= \rho \vec \e_\rho+\rho \dot{\theta}\sin \varphi \vec \e_\theta+\rho \dot{\varphi}\vec \e_\varphi\) where \( (\vec \e_\rho,\vec \e_\theta,\vec \e_\varphi)\) is the orthonormal basis given by
\begin{equation*}
\begin{split}
\vec \e_\rho& =(\cos \theta(t)\sin \varphi(t), \sin \theta(t) \sin \varphi(t), \cos \varphi(t)) , \\
\vec \e_\theta& =( -\sin \theta(t), \cos \theta(t),0) , \\
\vec \e_\varphi& = (\cos \theta(t) \cos \varphi(t), \sin \theta(t) \cos \varphi(t), \sin \varphi(t)),
\end{split}
\end{equation*}
which satisfies $\vec \e_\rho \wedge \vec \e_\theta = \vec \e_\varphi$.
 Hence we can express
\begin{multline*}
\frac{d[\vec u\circ \vec \gamma](t)}{dt} = \p_\rho u_\rho (\vec \gamma'\cdot \vec \e_\rho) \vec e_\rho +u_\rho \p_\rho u_\varphi (\vec \gamma'\cdot \vec\e_\rho) \vec e_\varphi +\p_\varphi u_\rho \frac{ \vec \gamma'\cdot \vec \e_\varphi}{\rho} \vec e_\rho \\
+u_\rho \p_\varphi u_\varphi \frac{\vec \gamma'\cdot \vec \e_\varphi}{\rho}\vec e_\varphi+u_\rho \sin u_\varphi \frac{\vec\gamma'\cdot \vec \e_\theta}{\rho \sin \varphi}\vec e_\theta.
\end{multline*}
Now we use that for three vectors, \(\vec a,\vec b, \vec h\) we have \( \vec a \otimes \vec b . \vec h=(\vec b\cdot \vec h ) \vec a\) and we find
\begin{multline*}
\frac{d[\vec u \circ \vec \gamma](t)}{dt}=\bigl[ \p_\rho u_\rho \vec e_\rho \otimes \vec \e_\rho +u_\rho \p_\rho u_\varphi \vec e_\varphi \otimes \vec \e_\rho +\frac{1}{\rho}\p_\varphi u_\rho \vec e_\rho \otimes \vec \e_\varphi \\
\frac{1}{\rho}\p_\varphi u_\rho \vec e_\rho \otimes \vec \e_\varphi+\frac{1}{\rho}u_\rho \p_\varphi u_\varphi \vec e_\varphi \otimes \vec \e_\varphi 
+\frac{u_\rho \sin u_\varphi}{\rho \sin \varphi}\vec e_\theta\otimes \vec \e_\theta \bigr].\vec \gamma'.
\end{multline*}
Hence in the bases \( (\vec e_\rho,\vec e_\theta,\Vec e_\varphi), (\vec \e_\rho, \vec \e_\theta,\vec \e_\varphi)\) (in the reference and deformed configurations, respectively) we have
\begin{equation*}
D \vec u=\begin{pmatrix}
\p_\rho u_\rho &0& \frac{1}{\rho}\p_\varphi u_\rho \\
0 & \frac{u_\rho \sin u_\varphi}{\rho \sin \varphi} & 0 \\
u_\rho \p_\rho u_\varphi & 0 & \frac{1}{\rho}u_\rho \p_\varphi u_\varphi
\end{pmatrix} .
\end{equation*}
Thus we find
\begin{equation*}
\begin{split}
|D \vec u|^2= |\p_\rho u_\rho|^2+|u_\rho \p_\rho u_\varphi|^2+\frac{1}{\rho^2}|\p_\varphi u_\rho|^2
+\frac{1}{\rho^2}|u_\rho \p_\varphi u_\varphi|^2 +\frac{|u_\rho \sin u_\varphi|^2}{\rho^2\sin^2\varphi},
\\
\cof D \vec u = \begin{pmatrix}
\frac{u_\rho^2 \sin u_\varphi \p_\varphi u _\varphi}{\rho^2 \sin \varphi}& 0 & \frac{-u_\rho^2 \sin u_\varphi \p_\rho u_\varphi}{\rho \sin \varphi} \\
0 & \frac{u_\rho}{\rho}(\p_\rho u_\rho \p_\varphi u_\varphi- \p_\rho u_\varphi \p_\varphi u_\rho) & 0 \\
-\frac{u_\rho \sin u_\varphi \p _\varphi u_\rho}{\rho^2 \sin \varphi}& 0 & \frac{u_\rho \p_\rho u_\rho \sin u_\varphi}{\rho \sin \varphi}
\end{pmatrix}
\end{split}
\end{equation*}
and
\begin{equation}\label{eq:det_spherical_spherical}
\det D \vec u = \frac{u_\rho^2 \sin u_\varphi}{\rho^2 \sin \varphi}(\p_\rho u_\rho \p_\varphi u_\varphi-\p_\rho u_\varphi\p_\varphi u_\rho).
\end{equation}
Note that the Dirichlet energy can be expressed by
\begin{multline}\label{eq:Dirichlet_energy_spherical_sperical}
\int_\Om |D \vec u|^2=2\pi \int_{\pi(\Om)} \Bigl[ |\p_\rho u_\rho|^2+|u_\rho \p_\rho u_\varphi|^2+\frac{1}{\rho^2}|\p_\varphi u_\rho|^2 \\
+\frac{1}{\rho^2}|u_\rho \p_\varphi u_\varphi|^2 +\frac{|u_\rho \sin u_\varphi|^2}{\rho^2\sin^2\varphi}\Bigr] \rho^2 \sin \varphi \dd \rho \dd\varphi.
\end{multline}

\subsection{Use of cylindrical-spherical coordinates}
Now we find the differential of a map given in spherical coordinates in the target but in cylindrical coordinates in the domain. If we assume that \(\vec u \) is axisymmetric then we can write

\begin{align*}
\vec u(r\cos \theta, r\sin \theta,x_3)&=u_\rho(r,x_3) (\cos \theta \sin u_\varphi(r,x_3) \vec e_1 +\sin \theta \sin u_\varphi(r,x_3) \vec e_2) +\cos u_\varphi(r,x_3) \vec e_3 \\
&=u_\rho(r,x_3) \sin u_\varphi(r,x_3) \vec e_r +u_\rho(r,x_3) \cos u_\varphi(r,x_3) \vec e_3.
\end{align*}
By using the same method as in Subsection \ref{subse:spherical-spherical},
 we find that, in the basis \( (\vec e_r,\vec e_\theta,\vec e_3)\),
\begin{equation*}
D\vec u=\begin{pmatrix}
\p_ru_\rho \sin u_\varphi +u_\rho \p_ru_\varphi \cos u_\varphi &0 & \p_{x_3} u_\rho \sin u_\varphi+u_\rho\p_{x_3} u_\varphi\cos u_\varphi \\
0 & \frac{u_\rho \sin u_\varphi}{r} & 0 \\
\p_r u_\rho \cos u_\varphi-u_\rho\p_r u_\varphi \sin u_\varphi &0 & \p_{x_3}u_\rho \cos u_\varphi-u_\rho \p_{x_3} u_\varphi \sin u_\varphi
\end{pmatrix} ,
\end{equation*}
\begin{equation}\label{eq:det_spherical_cylindrical}
\det D\vec u= \frac{u_\rho^2\sin u_\varphi}{r}(\p_r u_\varphi \p_{x_3}u_\rho-\p_ru_\rho\p_{x_3} u_\varphi).
\end{equation}
In this case the Dirichlet energy is given by 
\begin{multline}\label{eq:Dirichlet_energy_cyl_coord}
\int_\Om |D\vec u|^2 =2\pi\int_{\pi(\Om)} \Bigl[|\p_ru_\rho|^2+|u_\rho \p_ru_\varphi|^2+\frac{1}{r^2}|u_\rho \sin u_\varphi|^2+|\p_{x_3}u_\rho|^2 +|u_\rho \p_{x_3}u_\varphi|^2 \Bigr]r \dd r \dd x_3.
\end{multline}

\section{A lemma in measure theory}

In Proposition \ref{pr:muproperties} we need the following property.

\begin{lemma} \label{le:var_sum}
  Let $\{f_k\}_{k\in \N}$ be a countable family of functions from $X\to Y$. 
  If the images $f_k(X)$ are disjoint, then 
  \[
	\left | \sum_k {f_k}_{\#} {\vecg\mu}\right | = \sum_k \big | {f_k}_{\#} {\vecg\mu}\big |.
  \]
\end{lemma}

\begin{proof}
  Take any set $A\subset Y$. Fix $M\in \N$ and $\e>0$. Refining the partitions if necessary,
  from the definition of the total variation of a vector-valued measure 
  it can be seen that a partition $A=\bigcup A_i$ exist such that 
  \[
  \sum_i |{f_k}_{\#}\vecg\mu(A_i)| 
   \geq |{f_k}_{\#}\vecg\mu|(A) -\frac{\e}{M}
    \text{ for each } k\in \{1,\ldots, M\}.
  \]
  Set $A_{ik}:=A_i\cap f_k(X)$. Since the $A_{ik}$ are disjoint, 
  for every $k$ and $i$ we have that
  \[
    {f_k}_{\#}\vecg\mu(A_i) = {f_k}_{\#}\vecg\mu(A_{ik})
      = \left (\sum_{j\in\N} {f_j}_{\#}\vecg\mu \right )(A_{ik}).
  \]
  Therefore,
  \[
   \sum_{k=1}^M |{f_k}_{\#}\vecg\mu|(A) -\e \leq 
   \sum_{k=1}^M \sum_i |{f_k}_{\#}\vecg\mu(A_i)| 
   = \sum_{k=1}^M \sum_i \left |\left (\sum_{j\in\N} {f_j}_{\#}\vecg\mu \right )(A_{ik})\right |.
  \]
  By definition of total variation, the right-hand side is less than or equal to $\left |\sum_{j\in\N} {f_j}_{\#}\vecg\mu \right | (A)$. 
  Since the inequality holds for all $M$ and $\e$, 
  we find that $\sum |{f_k}_{\#}\vecg\mu|(A)\leq \left |\sum {f_k}_{\#}\vecg\mu \right | (A)$.
  The reverse inequality is easier to prove: it is deduced from the definition of $|\sum {f_k}_{\#}\vecg\mu  |$
  and the triangle inequality.  
\end{proof}

\section{Surface energy of a harmonic dipole}
\label{se:surface_energy_dipole}

Let $\vec u$ be the Conti--De Lellis map defined in Section \ref{se:LimitMap}. Here we prove that its 
surface energy $\E(\vec u)$, as it is defined in \cite{HeMo10}, coincides with twice the area of its created surface, 
that is, $2\pi$. Set $\Om:=B(\vec 0, 3)$. Given $\vec f$ any test function in $C_c^1(\Om \times \R^3, \R^3)$,
	\begin{multline*}
		\E_{\vec u}(\vec f) =
		\int_\Om \cof \nabla \vec u \cdot \nabla_{\vec x} \vec f\big ( \vec x, \vec u(\vec x)\big )
		+ \det \nabla \vec u(\vec x) \dive_{\vec y} \vec f\big ( \vec x, \vec u(\vec x)\big )
		\dd \vec x
		\\ =
		\lim_{\rho\to 0} \int_{\Om_\rho} \cof \nabla \vec u \cdot \nabla_{\vec x} \vec f\big ( \vec x, \vec u(\vec x)\big )
		+ \det \nabla \vec u(\vec x) \dive_{\vec y} \vec f\big ( \vec x, \vec u(\vec x)\big )
		\dd \vec x,
	\end{multline*}
	where $\Om_\rho:=\Om \setminus \Big ( B\big ((0,0,0), \rho\big )\cup B\big ((0,0,1),\rho\big )\Big)$. Using that  $\vec u$ is smooth in $\Om_{\rho}$, by changing variables we obtain
	\begin{align*}
	 \E_{\vec u}(\vec f) &=
	 \lim_{\rho\to 0} \int_{\vec u (\Om_\rho)} \dive \Big ( \vec f\big (\vec u^{-1}(\vec y), \vec y\big )\Big ) \dd\vec y
	 =
	 \lim_{\rho\to 0} \int_{\partial \vec u (\Om_\rho)}
	 \vec f \big ( \vec u^{-1}(\vec y),\vec y \big ) \cdot
	 \vecg \nu(\vec y) \dd\Ha^2(\vec y).
	\end{align*}

Choosing as the Borel orientation of the bubble $\Gamma$ the unit vector $\vecg\nu_{\Gamma}$ 
pointing outside $B\big ( (0,0,\frac{1}{2}), \frac{1}{2} \big )$, the above integral can be written as
	\begin{align*}
	 \E_{\vec u}(\vec f) &=
	 \int_\Gamma
	 \vec f \Big ( \big (\vec u^{-1}\big )^-(\vec y),\vec y \Big ) \cdot
	 \vecg \nu_\Gamma (\vec y) \dd\Ha^2(\vec y)
	 - \int_\Gamma
	 \vec f \Big ( \big (\vec u^{-1}\big )^+(\vec y),\vec y \Big ) \cdot
	 \vecg \nu_\Gamma (\vec y) \dd\Ha^2(\vec y)
	 \\ &= \int_\Gamma
	 \vec f \Big ( (0,0,0),\vec y \Big ) \cdot
	 \vecg \nu_\Gamma (\vec y) \dd\Ha^2(\vec y)
	 - \int_\Gamma
	 \vec f \Big ( (0,0,1),\vec y \Big ) \cdot
	 \vecg \nu_\Gamma (\vec y) \dd\Ha^2(\vec y).
	\end{align*}
	Since $|\vec f \cdot \vecg\nu|\leq \|\vec f\|_\infty$, this integral is bounded above by $2\Ha^2(\Gamma)= 2\pi$
	for any $\vec f$ with $\|\vec f\|_\infty \leq 1$.
	Taking a test function $\vec f$ such that
	\begin{align*}
	 \vec f\Big ( (0,0,0),\vec y \Big ) = \vecg\nu_\Gamma(\vec y)
	 \quad \text{and} \quad
	 \vec f\Big ( (0,0,1),\vec y \Big ) = -\vecg\nu_\Gamma(\vec y)
	\end{align*}
	for all $\vec y$ on $\Gamma$
	(which is possible since the singular points are separated in the reference configuration),
	it follows that the supremum over all such $\vec f$ is exactly $2\pi$; that is, $\E(\vec u)= 2\pi$, as claimed.

\addtocontents{toc}{\setcounter{tocdepth}{1}}
\subsection*{Financial Support}

Marco Barchiesi has been supported by project VATEXMATE.

Duvan Henao has received funding from FONDECYT project 1231401 and from the
Center for
Mathematical Modeling through ANID/Basal
project FB210005.

Carlos Mora-Corral has been supported by the Agencia Estatal de Investigaci\'on of the Spanish Ministry of Research and Innovation, 
through project PID2021-124195NB-C32 and the Severo Ochoa Programme for Centres of Excellence in R\&D CEX2019-000904-S, 
by the Madrid Government (Comunidad de Madrid, Spain) under the multiannual Agreement with UAM in the line for the Excellence 
of the University Research Staff in the context of the V PRICIT (Regional Programme of Research and Technological Innovation), and by the ERC Advanced Grant 834728.

R\'emy Rodiac has been partially supported by the ANR project BLADE Jr. ANR-18-CE40-0023.

\addtocontents{toc}{\setcounter{tocdepth}{2}}

{\small
\bibliography{biblio} \bibliographystyle{siam}
}

\end{document}